\documentclass{acta_2011}
\usepackage{harvard_acta}
\usepackage{amsmath}
\pagerange{\pageref{firstpage}--\pageref{lastpage}}
\allowdisplaybreaks
\usepackage[twoside,
	paperheight=247mm,
	paperwidth=174mm,
	marginratio={3:5,2:3},
	left=18mm,
	top=20mm]{geometry}
\numberwithin{figure}{section}
\numberwithin{table}{section}
\usepackage{graphicx}
\usepackage{epsf}
\usepackage{amsfonts,amsmath,amssymb}
\newcommand{\ie}{{\it i.e.}}
\newcommand{\Reals}{{\mathbb{R}}}   % Real numbers
\newcommand{\beq}{\begin{equation}}
\newcommand{\eeq}{\end{equation}}
\def\cov{\operatorname{cov}}
\def\E{\mathbb{E}}
\def\PiBG{\Pi_{\operatorname{BG}}}
\def\Exp{\operatorname{Exp}}
\def\tr{\operatorname{trace}}
\newtheorem{theorem}{Theorem}[section]
\newtheorem{proposition}{Proposition}[section]
\newtheorem{remark}{Remark}[section]
\newtheorem{example}{Example}[section]
\newtheorem{algorithm}{Algorithm}[section]
\begin{document}

\title[Geometric integrators and  Hamiltonian Monte Carlo]
{Geometric integrators and the Hamiltonian Monte Carlo method}

\author[N. Bou-Rabee and J. M. Sanz-Serna]
{Nawaf Bou-Rabee\\
Department of Mathematical Sciences \\
Rutgers University Camden \\ 311 N 5th Street \\ Camden, NJ 08102 USA \\
E-mail: {\sf nawaf.bourabee@rutgers.edu}\\
\and
J. M. Sanz-Serna \\
Departamento de  Matem\'aticas \\
Universidad Carlos III de Madrid \\
Avenida de la Universidad 30 \\
E--28911 Legan\'es (Madrid), Spain \\
E-mail: {\sf jmsanzserna@gmail.com}
}

\maketitle
\label{firstpage}

\vspace*{1.5cm}
\begin{abstract}
\end{abstract}

\tableofcontents

\vspace{2mm}
\section{Introduction}
This paper surveys the relations between numerical integration and the Hamiltonian (or Hybrid) Monte Carlo
Method (HMC), an important and widely used Markov Chain Monte Carlo algorithm \cite{Di2009} for sampling from
probability distributions. It is written for a general audience and requires no background on numerical
algorithms for solving differential equations. We hope that it will be useful to mathematicians, statisticians
and scientists, especially because the efficiency of HMC is largely dependent on the performance of the
numerical integrator used in the algorithm.

%MCMC
Named one of the top ten algorithms of the twentieth century \cite{Ci2000}, MCMC originated in statistical
mechanics \cite{AlTi1987,FrSm2002,Kr2006,LeRoSt2010,tuckerman2010statistical,LaBi2014} and is now a
cornerstone in statistics \cite{geman1984stochastic,gelfand1990sampling,geyer1992practical,tierney1994markov};
in fact Bayesian approaches only became widespread once MCMC made it possible to sample from completely
arbitrary distributions. In conjunction with Bayesian methodologies, MCMC has enabled applications of
statistical inference to biostatistics \cite{jensen2004computational}, population modelling
\cite{link2009bayesian}, reliability/risk assessment/uncertainty quantification
\cite{sullivan2015introduction}, machine learning \cite{andrieu2003introduction}, inverse problems
\cite{stuart2010inverse}, data assimilation \cite{chen2003bayesian,evensen2009data}, pattern recognition
\cite{webb2003statistical,bishop2006pattern}, artificial intelligence \cite{ghahramani2015probabilistic}, and
probabilistic robots \cite{thrun2005probabilistic}.   In these applications, MCMC is used to evaluate the
expected values necessary for Bayesian statistical inference, in situations where other methods like numerical
quadrature, Laplace approximation, and Monte Carlo importance sampling are impractical or inaccurate.
Additionally, MCMC is used as a tool to set the invariant distribution of numerical methods for first and
second-order Langevin stochastic differential equations
\cite{KiYoMaWa1991,RoTw1996A,RoTw1996B,BoVa2010,BoVa2012,BoHa2013,BoDoVa2014,Bo2014,Fa2014,FaHoSt2015} and
stochastic partial differential equations \cite{BeRoStVo2008,Bo2017}.  Even though MCMC algorithms are often
straightforward to program, there are numerous user-friendly, general-purpose software packages available to
carry out statistical analysis including BUGS \cite{lunn2000winbugs,lunn2009bugs,lunn2012bugs}, STAN
\cite{carpenter2016stan}, MCMCPack \cite{martin2011mcmcpack},  MCMCglmm \cite{hadfield2010mcmc} and PyMC
\cite{patil2010pymc}.

HMC itself was invented in 1987 \cite{DuKePeRo1987} to study lattice models of quantum field theory, and about
a decade later popularized in data science \cite{Li2008,Ne2011}. A simple introduction to this algorithm may
be found in \cite{Sa2014}. A key feature of HMC is that it offers the possibility of generating proposal moves
that, while being far away from the current state of the Markov chain, may be accepted with high probability,
thus avoiding random walk behaviour and reducing the correlation between samples. Such a possibility exists
because proposals are generated by numerically integrating a system of Hamiltonian differential equations. The
distance between the proposal and the current state may in principle be large if the differential equations
are integrated over a suitably long time interval; the acceptance probability of the proposals may be made
arbitrarily close to \(100\%\) by carrying out the integration with sufficient accuracy. Of particular
significance for us is the fact that HMC requires that the numerical integration be performed with a
\emph{volume-preserving, reversible method}.

Since the computational cost of HMC mainly lies in the numerical integrations, it is of much interest to
perform these as efficiently as possible. At present, the well-known velocity Verlet algorithm is the method
of choice, but, as it will be apparent, Verlet may not be the most efficient integrator one could use. What
does it take to design a good integrator for HMC? A key point of this paper is that, due to the specificities
of the situation, a number of concepts traditionally used to analyze numerical integrators (including the
notions of order of consistency/converengence, error constants, and others) are of limited value in our
context. On the one hand and as we have already mentioned, HMC requires methods that have the \emph{geometric
properties} of being volume-preserving and reversible and this limits the number of integrators that may be
applied. It is fortunate that, in the last twenty-five years, the literature on the numerical solution of
differential equations has given much attention to the construction of integrators with relevant geometric
properties, to the point that \emph{geometric integration} (a term introduced in \cite{geometric}) is by now a
well-established subfield of numerical analysis \cite{arieh}. On the other hand, the properties of
preservation of volume and reversibility have important quantitative implications on the integration error
(Theorem~\ref{thm:mean_energy_error}), which in turn have an impact on the acceptance rate of proposals. As a
consequence, it turns out that, for HMC purposes, the order of the integrator is effectively \emph{twice} its
nominal order; for instance the Verlet algorithm behaves, within HMC, as a \emph{fourth order} integrator. In
addition, in HMC, integrators are likely to be operated with large values of the step size, with the
implication that analyses that are only valid in the limit of vanishing step size may not be very informative.
One has rather to turn to studying the performance of the integrator in well-chosen model problems.

Sections~\ref{sec:deandtheirflows}--\ref{sec:montecarlomethods} provide the necessary background on
differential equations, numerical methods, geometric integration and Monte Carlo methods respectively. The
heart of the paper is in Section~\ref{secc:numintHMC}. Among the topics considered there, we mention the
investigation of the impact on the energy error of the properties of preservation of volume and reversibility
(Theorem~\ref{thm:mean_energy_error}) and a detailed study of the behaviour of the integrators in the Gaussian
model. Also presented in Section~\ref{secc:numintHMC} is the construction of integrators more efficient than
the Verlet algorithm. Sections~\ref{sec:highD_hmc} and \ref{sec:path} consider, in two different scenarios,
the behaviour of HMC as the dimensionality of the target distribution increases.  Section~\ref{sec:highD_hmc},
based on \cite{BePiRoSaSt2013} studies the model problem where the target is a product of many independent,
identical copies. The case where the target arises from discretization of an infinite-dimensional distribution
is addressed in Section~\ref{sec:path}; our treatment, while related to the material in \cite{BePiSaSt2011},
has some novel features because we have avoided the functional analytic language employed in that reference.
The final Section~\ref{sec:supplemenary} contains supplementary material.

\vspace{2mm}

\section{Differential equations and their flows}
\label{sec:deandtheirflows}

In this section we introduce some notation and review  background material on differential equations, with
special emphasis on the Hamiltonian and reversible systems  at the basis of HMC algorithms. The section ends
with a description of the Lie bracket which appears in the analysis of the integrators to be used later.

\subsection{Preliminaries}

We are concerned with autonomous systems of differential equations in \( \Reals^D\)
\beq\label{eq:ode}
 \frac{d}{dt} x = f(x);
\eeq
the function (vector field) \(f\) is assumed throughout to be defined in the whole of \(\Reals^D\) and to be
sufficiently smooth. An important role is played  by the particular case
\beq\label{eq:newton} \frac{d}{dt} q = M^{-1}p,\qquad \frac{d}{dt} p = F(q), \eeq
where \(x= (q,p)\in\Reals^D\), \(D=2d\), \( q\in\Reals^d\), \( p\in\Reals^d\) and \(M\) is a constant,
invertible matrix, so that \(f= (M^{-1}p, F(q))\). By eliminating \(p\), \eqref{eq:newton} is seen to be equivalent to
\[
M\frac{d^2}{dt^2} q = F(q);
\]
this is not the most general autonomous system of second order differential equations in \(\Reals^d\) because
the derivative \((d/dt)q\) does not appear in the right-hand side. When the forces depend only on the
positions, {\em Newton's second law} for a mechanical system gives rise to differential equations of the form
\eqref{eq:newton}; then \(q\), \((d/dt)q\),  \(p\), \(F\) are respectively the vectors of coordinates,
velocities, momenta, and forces, and \(M\) is the matrix of masses.

 We denote by \( \varphi_t\) the \(t\)-{\em flow} of the system \eqref{eq:ode} under consideration. By definition,
for fixed real \(t\), \(\varphi_t:\Reals^D \to\Reals^D\) is the map that associates with each
\(\alpha\in\Reals^D\) the value at time \(t\) of the solution  of \eqref{eq:ode} that at the initial time 0
takes the initial value \(\alpha\).

\begin{example}
As a very simple but important example, we consider  the  standard \emph{harmonic oscillator}, the system in
\(\Reals^2\) of the special form \eqref{eq:newton} given by
\beq \label{eq:harmonic}
\frac{dq}{dt} = p,\qquad \frac{dp}{dt} = -q.
\eeq
For future reference, we note that, with matrix notation, the solutions satisfy
 \beq\label{eq:rotation}
\left[
\begin{matrix}q(t)\\p(t)\end{matrix}\right] = M_t\left[
\begin{matrix}q(0)\\p(0)\end{matrix}\right],\qquad M_t = \left[ \begin{matrix}\phantom{-}\cos t & \sin t\\
-\sin t & \cos t\end{matrix}\right].
\eeq
Thus the flow has the expression
\beq\label{eq:harmonicflow} \varphi_t(\xi,\:\eta) = (\xi \cos t +\eta \sin t,\: -\xi
\sin t + \eta \cos t). %
\eeq
When  \( \alpha =(\xi,\eta)\in\Reals^2 \) is fixed and   \(t\) varies, the right-hand side of
\eqref{eq:harmonicflow} yields the solution that at \(t=0\) takes the initial value \((\xi,\eta)\). The
notation \(\varphi_t(\xi,\:\eta)\) emphasizes that, in the flow, it is the parameter \(t\) that  is seen as
fixed, while \((\xi,\eta)\in\Reals^2\) is regarded as a variable. Geometrically, \(\varphi_t\) is the
clockwise rotation of angle \(t\) around the origin of the \((q,p)\)-plane.
\end{example}

For a given system \eqref{eq:ode}, it is well possible that for some choices of  \( \alpha\) and \(t\), the
vector \( \varphi_t(\alpha)\in\Reals^D\) is not defined; this will happen if \(t\) is outside the interval in
which the solution of \eqref{eq:ode} with initial value \(\alpha\) exists. For simplicity in the statements,
we shall  assume hereafter that \( \varphi_t(\alpha)\) is always defined.

Flows possess the {\em group property:} \(\varphi_0\) is the indentity map in \(\Reals^D\) and, for arbitrary
real \( s\) and \(t\),
\beq\label{eq:flow} \varphi_t \circ \varphi_s = \varphi_{s+t}. \eeq
 In particular, for each \(t\),
\beq
 \label{eq:inverseflow} (\varphi_t)^{-1} = \varphi_{-t}, \eeq
\ie\  \(\varphi_{-t}\) is the inverse of the map  \(\varphi_t\). For the harmonic oscillator example, the
group property simply states that a rotation of angle \(s\) followed by a rotation of angle \(t\) has the same
effect as a rotation of angle \(s+t\).

\subsection{Hamiltonian systems}
The Hamiltonian formalism  \cite{MaRa1999,arnold} is essential to understand HMC algorithms.
\subsubsection{Hamiltonian vector fields}
Assume that the dimension \(D\) of   \eqref{eq:ode} is even, \(D = 2d\), and write \(x = (q,p)\) with
\(q,p\in\Reals^d\). Then the system \eqref{eq:ode} is said to be \emph{Hamiltonian} if there is a function
\(H:\Reals^{2d}\rightarrow \Reals\) such that, for \(i= 1,\dots,d\), the scalar components \(f^i\) of \(f\)
are given by
\[
f^i(q,p) = +\frac{\partial H}{\partial q^i}(q,p),\qquad
f^{d+i}(q,p) = -\frac{\partial H}{\partial p^i}(q,p).
\]
Thus, the system is
\[
\frac{dq^i}{dt} = +\frac{\partial H}{\partial q^i}(q,p),\qquad
\frac{dp^i}{dt} = -\frac{\partial H}{\partial p^i}(q,p),
\]
or, in vector notation \cite{SaCa1994},
\beq\label{eq:hamsys}
\frac{d}{dt} \left[\begin{matrix}q\\p\end{matrix}\right] = J^{-1} \nabla H(q,p),
\eeq
where
\[
\nabla H = \left[\frac{\partial H}{\partial q^1},\dots, \frac{\partial H}{\partial q^d},
\frac{\partial H}{\partial p^1},\dots, \frac{\partial H}{\partial p^d}\right]^T
\]
and
\[
J = \left[
\begin{matrix}
0_{d\times d}& -I_{d\times d}\\
I_{d\times d}&\phantom{-} 0_{d\times d}
\end{matrix}
\right].
\]
The function \(H\) is called the \emph{Hamiltonian}, \(\Reals^{2d}\) is the \emph{phase space}, and \(d\) is
the \emph{number of degrees of freedom}.

A  system of the special form \eqref{eq:newton} is Hamiltonian if and only if \(F= -\nabla\mathcal{U}(q)\) for
a suitable real-valued function \(\mathcal{U}\), \ie\
\beq\label{eq:newton2}
\frac{d}{dt} q = M^{-1}p,\qquad \frac{d}{dt} p = -\nabla\mathcal{U}(q);
\eeq
when that is the case,
\beq\label{eq:sepham} H(q,p) = \mathcal{T}(p)+\mathcal{U}(q),\qquad {\rm with} \qquad \mathcal{T}(p) =
\frac{1}{2} p^TM^{-1}p. \eeq
In applications to mechanics, \(\mathcal T\) and \(\mathcal U\) are respectively the \emph{potential} and
\emph{kinetic} energy and \(H\) represents the \emph{total energy} in the system. The harmonic oscillator
\eqref{eq:harmonic} provides the simplest example; there \( \mathcal{T} = (1/2)p^2\), \( \mathcal{U} =
(1/2)q^2\).

\subsubsection{Symplecticness and preservation of oriented volume}
A mapping \(\Phi: \Reals^{2d} \rightarrow \Reals^{2d}\) is said to be \emph{symplectic} or \emph{canonical} if, at each point
\((q,p)\in\Reals^{2d}\),
\beq\label{eq:symplectic} \Phi^\prime(q,p)^T J \Phi^\prime(q,p) = J \eeq
(\(\Phi^\prime(q,p)\) denotes the \(2d\times 2d\) Jacobian matrix of \(\Phi\)). The (analytic) condition
\eqref{eq:symplectic} has a geometric interpretation in terms of preservation of two-dimensional areas \cite{arnold}; such
interpretation is not required to understand the rest of the paper.

When \(d=1\), if we set \(\Phi(q,p) = (q^*,p^*)\), the condition \eqref{eq:symplectic}, after multiplying the
matrices in the left-hand side, is seen to be equivalent to
\[
\frac{\partial q^*}{\partial q} \frac{\partial p^*}{\partial p}
- \frac{\partial q^*}{\partial p} \frac{\partial p^*}{\partial q} = 1.
\]
The left-hand side is the Jacobian determinant of \(\Phi\) and therefore the transformation \(\Phi\) is symplectic
if and only if the mapping \((q,p)\mapsto (q^*,p^*)\) preserves  oriented area in the \((q,p)\) plane, \ie\
for any domain \(D\) the oriented area of the image \(\Phi(D)\subset\Reals^2\) coincides with the oriented area of
\(D\).\footnote{Preservation of the oriented area means that \(D\) and \(\Phi(D)\) have the same orientation
and (two-dimensional Lebesque) measure. The transformation (symmetry) \((q,p)\mapsto (q,-p)\) preserves
measure but not oriented area.} For instance, for each \(t\), the rotation in \eqref{eq:harmonicflow} is a
symplectic transformation in \(\Reals^2\).

For general \(d\) the following result holds \cite[Section 38]{arnold}:

\begin{proposition}\label{eq:con_volume} For a symplectic transformation the determinant of  \(\Phi^\prime\)
 equals 1.
 Therefore symplectic transformations preserve the oriented volume in
  \(\Reals^{2d}\), \ie\ \(\Phi(D)\)  and \(D\) have the
same oriented volume for each domain \(D\subset\Reals^{2d}\).
\end{proposition}

For \(d>1\) preservation of oriented volume is a strictly weaker property than symplecticness.

The proof of the following two results is easy using \eqref{eq:symplectic}.
\begin{proposition}\label{prop:compsymp}
The composition \(\Phi_1\circ \Phi_2\) of two symplectic mappings is itself symplectic.
\end{proposition}

\begin{proposition}\label{prop:canonical change} The change of variables \((q,p) = \Phi(\bar q,\bar p)\)
with \(\Phi\) symplectic transforms the Hamiltonian system of differential equations \eqref{eq:hamsys} into a
system for \((\bar q,\bar p)\) that is also Hamiltonian. Moreover, the Hamiltonian function \(\bar H\) of the
transformed system is the result of changing variables in \(H\), \ie\ \(\bar H = H\circ \Phi\).
\end{proposition}

In view of the following important general result  \cite[Proposition 2.6.2]{MaRa1999} the symplecticness of the rotation
\eqref{eq:harmonicflow} noted above is a manifestation of the Hamiltonian character of the harmonic
oscillator.
\begin{theorem}\label{th:symplectic}
Let \(D =2d\). The  system \eqref{eq:ode} with flow \(\varphi_t\) is Hamiltonian if and only if, for each real
\(t\), \(\varphi_t\) is a symplectic mapping.
\end{theorem}

Thus symplecticness  is a characteristic property that allows us to decide whether a differential system is
Hamiltonian or otherwise in terms of its \emph{ flow}, without knowing  the \emph{vector field} (right-hand
side) \(f\) of the equation.

We recall that a flow preserves oriented volume if and only if the corresponding vector field \(f\) is
divergence-free (\(\nabla\cdot f = \sum_i (\partial/\partial x^i)f^i =0\)). If \(d>1\), there are
divergence-free differential systems in \(\Reals^{2d}\) that are not Hamiltonian; their flows preserve
oriented volume but are not symplectic.

The behavior of the solutions of Hamiltonian problems is very different from that encountered in \lq
general\rq\ systems; some features that are \lq the rule\rq\ in Hamiltonian systems are exceptional in
non-Hamiltonian systems. Such special behaviour of Hamiltonian solutions may always be traced back to the
symplecticness of the flow. As a very simple example, we consider once more the harmonic oscillator
\eqref{eq:harmonic}. The origin is a center: a neutrally stable equilibrium surrounded by periodic
trajectories. Small perturbations of the right-hand side of \eqref{eq:harmonic} generically destroy the
center; after perturbation the trajectories become spirals and the origin becomes either an asymptotically
stable node (trajectories spiral in) or an unstable node (trajectories spiral out). However, if the
perturbation is such that the perturbed system is also Hamiltonian, then the center will not disappear under
small perturbations.

\subsubsection{Preservation of energy}
If \((q(t),p(t))\) is a solution of \eqref{eq:hamsys}, then
\[
\frac{d}{dt} H(q(t),p(t)) = \nabla H(q(t),p(t))^T J^{-1}\nabla H(q(t),p(t)) = 0,
\]
because \(J^{-1}\) is skew-symmetric. Therefore we may state:
\begin{theorem} \label{th:consenergy}
 The value of the Hamiltonian function is preserved by the flow of the corresponding Hamiltonian system, \ie\
\( H\circ \varphi_t = H\) for each real \(t\).
\end{theorem}

 In applications to the physical sciences, this result is usually   the mathematical expression of the
principle of conservation of energy. Unlike symplecticness, which is a characteristic property, conservation
of energy on its own does not ensure that the underlying system is Hamiltonian. There are many examples of
non-Hamiltonian systems whose solutions preserve the value of a suitable energy function.
\subsubsection{Preservation of the canonical probability measure}
Let \(\beta\) denote a positive constant and assume that \(H\) is such that
\[
Z = \int_{\Reals^{2d}} \exp(-\beta H(q,p))\,dqdp < \infty.
\]
Then we have the following result, which is  a direct consequence of the fact that \(\varphi_t\) preserves
both the volume element \(dqdp\) (Proposition~\ref{eq:con_volume}) and the value of \(\exp(-\beta H)\)
(because, according to Theorem~\ref{th:consenergy}, it preserves the value of \(H\)).
\begin{theorem}\label{th:bg}
The probability measure \(\mu\) in \(\Reals^{2d}\) with density  (with respect
to the ordinary Lebesgue measure) \(Z^{-1} \exp(-\beta H(q,p))\) is preserved by the  flow of the Hamiltonian system \eqref{eq:hamsys}, \ie\
\(\mu(\varphi_t(D)) =\mu(D)\) for each domain \(D\subset\Reals^{2d}\) and each real \(t\).
 \end{theorem}

In statistical mechanics \cite{AlTi1987,FrSm2002,tuckerman2010statistical}, if \eqref{eq:hamsys} describes the
dynamics of a physical system and \(\beta\) is the inverse of the absolute temperature, then  \(\mu\) is the
\emph{canonical} measure that governs the distribution of \((q,p)\) over an ensemble of many copies of the
given system when the system is in contact with a heat bath at constant temperature, \ie\ \(Z^{-1} \exp(-\beta
H(q,p))dqdp\) represents the fraction of copies with momenta between \(p\) and \(p+dp\), and configuration
between \(q\) and \(q+dq\). Note that under the canonical distribution, (local) minima of the energy \(H\)
correspond to (local) maxima of the probability density function, \ie\ to \emph{modes} of the distribution.
Also, if the temperature decreases (\(\beta\) increases), it is less likely to find the system at a location
\((q,p)\) with high energy.

For Hamiltonian functions of the particular form in \eqref{eq:sepham}, we may factorize
\[
\exp(-\beta H(q,p)) = \exp(-\frac{1}{2}\beta\, p^TM^{-1}p)\times \exp(-\beta\, \mathcal{U}(q))
\]
and therefore, under the canonical distribution, \(q\) and \(p\) are stochastically independent. The
(marginal) distribution of the configuration variables \(q\)  has probability density function proportional to
\(\exp(-\beta\, \mathcal{U}(q))\). The momenta \(p\) possess a Gaussian distribution with zero mean and
covariance matrix equal to \(M\). These distributions are associated with the names of Boltzmann, Gibbs and
Maxwell. Hereafter we refer to the canonical measure as the Boltzmann-Gibbs distribution.

\subsection{Reversible systems}

Assume now that \(S\) is a linear \emph{involution} in \(\Reals^D\), \ie\ a linear map such that \(S(S(x)) =
x\) for each \(x\). A mapping \(\Phi:\Reals^D\rightarrow\Reals^D\) is said to be \emph{reversible} (with
respect to \(S\)) if, for each \(x\), \(S(\Phi(x)) = \Phi^{-1}(S(x))\) or, more compactly,
\beq\label{eq:reversibility}
 S\circ \Phi = \Phi^{-1}\circ S.
\eeq

The following results have  easy proofs.
\begin{proposition}\label{prop:jacob} If \(\Phi\) is \(S\) reversible, then
\[
\left| \det \Phi'\big(S ( \Phi(x)) \big) \right| = \left| \det \Phi'(x) \right|^{-1},
\]
for each \(x\).
\end{proposition}
\begin{proposition}\label{prop:comprever}
If \(\Phi_1\) is \(S\)-reversible, then \(\Phi_1\circ\Phi_1\) is \(S\)-reversible. If \(\Phi_1\) and
\(\Phi_2\) are \(S\)-reversible, then the symmetric composition \(\Phi_1\circ \Phi_2\circ \Phi_1\) is
\(S\)-reversible.
\end{proposition}

\begin{theorem}\label{th:reverflow}
Consider the system \eqref{eq:ode} with flow \(\varphi_t\). The following statements are
equivalent:
\begin{itemize}
\item For each \(t\), \(\varphi_t\) is an \(S\)-reversible mapping.
\item For each \(x\in\Reals^D\), \( S(f(x)) = -f(S(x)) \), \ie\ \(S\circ f = -f\circ S\).
\end{itemize}
\end{theorem}

Systems of differential equations that satisfy the conditions in the theorem are said to be \emph{reversible}
(with respect to \(S\)). Systems of the particular form \eqref{eq:newton} are reversible with respect to the
\emph{momentum flip} involution
\beq\label{eq:mflip}
S(q,p) = (q,-p).
\eeq
 If \eqref{eq:newton} describes a mechanical system, then \eqref{eq:reversibility}
expresses the well-known time-reversibility of mechanics: if \((q_{in},p_{in})\) is the initial state of a
system and \((q_{f},p_{f})\) the final state after \(t\) units of time have elapsed, then the  state
\((q_{f},-p_{f})\) evolves in \(t\) units of time to the state \((q_{in},-p_{in})\).

\begin{proposition}\label{prop:mflip}
The Hamiltonian system \eqref{eq:hamsys} is reversible with respect to the momentum flip involution
\eqref{eq:mflip}
 if and only if \(H\) is an even function of \(p\), \ie\ \(H(q,-p) = H(q,p)\) for all \(q\) and
\(p\).
\end{proposition}

Figure~\ref{fig:rever_energy} illustrates the reversibility of the Hamiltonian flow corresponding to a
one-degree-of freedom double-well potential.

As it is the  case for Hamiltonian systems, reversible reversible systems have flows with special geometric
properties, not shared by \lq general\rq\ systems \cite{LaRo1998}.

\subsection{The Lie bracket}
If \(\varphi_t^{(f)}\) and \(\varphi_t^{(g)}\) denote respectively the flows of the \(D\)-dimensional systems
\[
\frac{d}{dt} x = f(x), \qquad \frac{d}{dt} x = g(x),
\]
in general \(\varphi_s^{(g)}\circ\varphi_t^{(f)}\neq \varphi_t^{(f)}\circ\varphi_s^{(g)}\); in fact, a Taylor expansion shows that, as \(t\), \(s\) approach \(0\),
\[
\varphi_s^{(g)}\big(\varphi_t^{(f)}(x)\big)- \varphi_t^{(f)}\big(\varphi_s^{(g)}(x)\big) = st\, [f,g](x) + \mathcal{O}(t^3+s^3),
\]
where the \emph{Lie bracket} (or \emph{Lie-Jacobi} bracket or commutator) \([f,g]\) of \(f\) and \(g\)
\cite[Section 39]{arnold} is the mapping \(\Reals^D\rightarrow \Reals^D\) that at \(x\in\Reals^D\) takes the
value
\beq\label{eq:liebracket}
[f,g](x) = g^\prime(x)f(x)-f^\prime(x)g(x)
\eeq
Thus the magnitude of \([f,g]\) measures the lack of commutativity of the corresponding flows. The following result holds:
\begin{theorem}\label{th:lie}
With the preceding notation, \(\varphi_s^{(g)}\circ\varphi_t^{(f)}=\varphi_t^{(f)}\circ\varphi_s^{(g)}\) for arbitrary
\(t\) and \(s\) if and only if the Lie bracket \([f,g](x)\) vanishes at each \(x\in\Reals^D\). When these conditions hold, we say that \(f\) and \(g\) commute.

For commuting \(f\) and \(g\),
\(\varphi_t^{(g)}\circ\varphi_t^{(f)} = \varphi_t^{(f)}\circ\varphi_s^{(g)}\) provides the \(t\) flow of the system
\[
\frac{d}{dt} x = f(x)+g(x).
\]
\end{theorem}

The Lie bracket is \emph{skew symmetric:} \([f,g] = - [g,f]\) for arbitrary \(f\) and \(g\). In addition it satisfies the
\emph{Jacobi identity},
\[
[f_1,[f_2,f_3]]+[f_2,[f_3,f_1]]+[f_3,[f_1,f_2]]= 0,
\]
for any \(f_1\), \(f_2\), \(f_3\). In this way the vector space of all vector fields together with the operation \([\cdot,\cdot]\) is a \emph{Lie algebra}.

For \emph{Hamiltonian} systems it is possible to work in terms of the so-called Poisson bracket of the
Hamiltonian functions \cite[Section 40]{arnold}, rather than in terms of the Lie bracket of the fields:

\begin{theorem}\label{th:poisson}
If the fields \(f\) and \(g\) are Hamiltonian with Hamiltonian functions \(H\) and \(K\) respectively, \ie\ \(f = J^{-1}\nabla H\), \(g= J^{-1}\nabla K\), then \([f,g]\) is also a Hamiltonian vector field. Moreover the Hamiltonian function of \([f,g]\) is given by \(-\{H,K\}\),\footnote{The minus sign here could be avoided by reversing the sign in the definition of the Poisson bracket. The definition of \(\{\cdot,\cdot\}\) used here is the one traditionally used in mechanics.} where \(\{H,K\}\) is the \emph{Poisson bracket} of the functions
\(H\) and \(K\), defined as
\[
\{H,K\} = (\nabla H)^T J^{-1} \nabla K.
\]
\end{theorem}

For real-valued functions  in \(\Reals^{2d}\), the Poisson bracket operation is skew symmetric
\(\{H,K\} = - \{K,H\}\) and satisfies the Jacobi identity:
\[
\{H_1,\{H_2,H_3\}\}+ \{H_2,\{H_3,H_1\}\}+\{H_3,\{H_1,H_2\}\} = 0.
\]

Let  \(f\) and \(g\) be \emph{reversible} vector fields.  Differentiation in \(g(S(x)) = -S(g(x))\) implies
for the Jacobian that \(g^\prime(S(x))S = -S(g^\prime(x))\) and it follows that \(g^\prime(S(x))f(S(x))) =
S(g^\prime(x)f(x))\) (no minus sign!). Then the Lie-bracket of two reversible fields is \emph{not} reversible,
but rather satisfies the following property.

\begin{proposition}\label{prop:revercomm}
If two vector fields \(f\) and \(g\) are \(S\)-reversible then
\[
[f,g](S(x)) = S\big([f,g](x)\big).
\]
\end{proposition}

For three \(S\)-reversible vector fields, the iterated commutator \( [f_1,[f_2,f_3]] \) is \(S\)-reversible.

\vspace{2mm}
\section{Integrators}
\label{sec:integrators}

 In the sampling algorithms studied later, differential systems like \eqref{eq:ode} or
\eqref{eq:newton} have to be numerically integrated. In this section we review the required material. The
works \cite{HaNoWa1993,HaWa2010,Bu2016} provide extensive, authoritative treatises on the subject. A more
concise introduction is given by \cite{GrHi2010}. We begin by recalling some basic definitions
 and later focus on splitting integrators and fixed \(h\) stability,
 as both play important roles in the implementation of HMC algorithms.

\subsection{Preliminaries}
Each {\em one-step numerical method or one-step integrator} for \eqref{eq:ode} is described by a map
\(\psi_h:\Reals^D\to \Reals^D\) that depends on a real parameter \( h\), the {\em step size}. Given the
initial value \(\alpha\), and a value of \(h\) (\(h\neq 0\)), the integrator generates a numerical trajectory,
\(x_0\), \(x_1\), \(x_2\), \dots, defined by \(x_0=\alpha\) and, iteratively,
\beq\label{eq:numiteration}
x_{n+1} = \psi_h(x_n), \qquad n = 0, 1, 2,\dots
\eeq
To compute  \(x_{n+1}\) when \(x_n\) has already been found  is to perform a (time) {\em step}. For each
\(n\), the vector \(x_n\) is an approximation to the value at time \(t_n =nh\) of the solution \(x(t)\) of
\eqref{eq:ode} with initial condition \(x(0) = \alpha\), \ie\ to \(\varphi_{t_n}(\alpha)\). Typically \(h\) is
positive and then the integration is forward in time \(0=t_0<t_1<t_2< \dots\), but in some applications it may
be of interest to use \(h<0\) so as to get \(0=t_0>t_1>t_2> \dots\)

The simplest and best known integrator, {\em Euler's rule}, with
\beq\label{eq:euler}
 x_{n+1} = x_n+hf(x_n),
 \eeq
corresponds to the
mapping \(\psi_h(x) = x+hf(x)\). It uses one evaluation of \(f\) per step. \emph{Explicit \(s\)-stage Runge-Kutta}
formulas use \(s\) evaluations of \(f\) per step, \(s=1,2,\dots\), and are therefore \(s\) times more
 expensive \emph{per step} than Euler's rule; examples include  Runge's method
\[
\psi_h(x) = x + h f\big(x+\frac{h}{2}f(x)\big)
\]
(with two stages), several well-known formulas of Kutta with four stages\footnote{One of these formulas is
known in some circles as \emph{the} Runge-Kutta method; this terminology should be avoided as there are
infinitely many Runge-Kutta methods.} and the
 formulas within  the popular MATLAB function ode45. A method with \(s\) stages will be
competitive with Euler's rule only if it gives more accurate approximations than Euler's  rule when this is
operated with a step size \(s\) times shorter, so as to equalize computational costs.

\emph{Implicit} Runge-Kutta integrators are also  used in practice; in them \(\psi_h\) is defined by means of  algebraic equations. For
instance,  the \emph{midpoint} rule has
\[
\psi_h(x) = x +h f\big( \frac{1}{2} (x+\psi_h(x))\big),
\]
\ie
\beq\label{eq:midpoint}
x_{n+1} = x_n +h f\big( \frac{1}{2} (x_n+x_{n+1}) \big).
\eeq
Here, computing \(x_{n+1}\) when \(x_n\) is already known
 requires to solve a system of algebraic equations in \(\Reals^D\). There are many more useful examples of implicit Runge-Kutta methods,
  including the so-called \emph{Gauss} methods.

\begin{remark}\label{rem:timescale}
Note that in the formulas displayed above \(h\) and \(f\) do not appear separately, but always in the
combination \(hf\). This clearly implies that, if, for a given method,  \(x_n\), \(n= 0,1,2\dots\) is a
numerical trajectory with step size \(h\) corresponding to the system \eqref{eq:ode}, it is also a numerical
trajectory with step size \(\lambda h\) for the system \((d/dt)x = \lambda^{-1}f(x)\) (\(\lambda \neq 0\)
denotes a  constant). \emph{We shall always assume that we deal with integrators having this property.}
\end{remark}

A one-step integrator is called \emph{symmetric} or \emph{self-adjoint}\footnote{Even though monographs like
\cite{SaCa1994} or \cite{HaLuWa2010} use the term self-adjoint, there are reasons against that terminology
\cite{sirev}.} if
\beq\label{eq:inversepsi}
(\psi_h)^{-1} = \psi_{-h},
\eeq
so as to mimic the property \eqref{eq:inverseflow} of the exact solution flow. The midpoint rule
\eqref{eq:midpoint} provides an example. Explicit Runge-Kutta methods are never symmetric.

Multistep integrators, including the well-known Adams formulas, where the computation of \(x_{n+1}\) requires
the knowledge  of \(k\geq 2\) past values \(x_{n}\), \(x_{n-1}\), \dots, \(x_{n-k+1}\),  may be very
efficient, but will not be considered in this paper; they have seldom been applied within sampling algorithms.

\subsection{Order}
Since   for the exact solution, the sequence \(x(0)=\alpha\), \(x(t_1)\), \(x(t_2)\),\dots\ satisfies
\(x(t_{n+1}) =\varphi_h(x(t_n))\), \( n= 0,1,\dots\), rather than \eqref{eq:numiteration}, for the numerical
integrator to make sense it is necessary that \(\psi_h\) be an approximation to \(\varphi_h\). The integrator
is said to be \emph{consistent} if, at each fixed \(x\in\Reals^D\), \(\psi_h(x) -\varphi_h(x) =
\mathcal{O}(h^{2})\) as \(h\rightarrow 0\). If \(\psi_h(x) -\varphi_h(x) = \mathcal{O}(h^{\nu+1})\), \(\nu\) a
positive integer, then the integrator is (consistent) of \emph{order} \(\geq\nu\). A method of order \(\geq
\nu\) that is not of order \(\geq \nu+1\) is said to have order \(\nu\).  Euler's rule \eqref{eq:euler} has
order 1, the four-stage formulas of Kutta have order 4 and the midpoint rule \eqref{eq:midpoint} has order 2.
We shall see later (Theorem~\ref{th:ordercond}) that the order of a symmetric integrator is an \emph{even}
integer. All integrators to be considered hereafter are assumed to be consistent.

The vector \(\psi_h(x) -\varphi_h(x)\) is called the {\em local error} at \(x\): it is the difference between
the result of a single time-step of the numerical method starting from \( x\) and the result of the
application of the \(h\)-flow to the \emph{same} point \(x\). For future reference, we note that the expansion
of the local error for Euler's method is given by
\begin{eqnarray}
\nonumber
\psi_h(x)-\varphi_h(x) &=& \Big(x+hf(x)\Big)-\Big(x+ hf(x)+\frac{h^2}{2} f^\prime(x)f(x)+\mathcal{O}(h^3)\Big)\\
& = & -\frac{h^2}{2} f^\prime(x)f(x)+\mathcal{O}(h^3).
\label{eq:eulerexpansion}
\end{eqnarray}
In the expansion of the true solution flow we have used that \((d/dt) x = f(x)\) and \((d^2/dt^2) x =
f^\prime(x)(d/dt) x = f^\prime(x)f(x)\).

The local error does not give {\em per se} information on the {\em global error} at \(t_n\), \ie\ on the
difference \(x_n-x(t_n)\). This is because \(x_n=\psi_h(x_{n-1})\) while \(x(t_n)\) is the result of the
application of the \(h\)-flow to \(x(t_{n-1})\): \(\psi_h\) and \(\varphi_h\) are \emph{not} applied to the
\emph{same} point.\footnote{
In order to bound the global error in terms of the local error, \ie\ to obtain
convergence from consistency, a \emph{stability} property is needed. Hence the well-known slogan \lq\lq
 stability + consistency imply convergence.\rq\rq\ The one-step integrators considered here always have the
required stability and therefore, for them, consistency implies convergence as in
Theorem~\ref{th:convergence}. The concept of stability that features in the slogan is \emph{different} from
the concept of fixed \(h\) stability studied in Section~\ref{ss:fixedh} \cite{lancaster}.}
 However the following result holds:

\begin{theorem}\label{th:convergence}
 If the (one-step) integrator is consistent of order \(\nu\), then,
for each fixed initial value  \(x_0=x(0)\) and  \(T>0\),
\beq\label{eq:convergence}
\max_{0\leq t_n\leq T}| x_n-x(t_n)| = \mathcal{O}(h^{\nu}),\qquad h\rightarrow 0+.
\eeq
\end{theorem}

This expresses the fact that the integrator is {\em convergent} of order \(\nu\). Of course, the corresponding
result holds if the integration is carried out backward in time over \(-T\leq t\leq 0\).

\begin{remark}\label{rem:asymptotic}
The global errors \(x_n-x(t_n)\) in addition to possessing an \(\mathcal{O}(h^{\nu})\) bound as above, have an
\emph{asymptotic expansion} in powers of \(h\). For instance, if we restrict the attention to the leading
\(h^\nu\) term in the expansion, we have
\[
x_n - x(t_n) = h^\nu a(x(0),t_n)+h^\nu r(x(0),t_n;h),
\]
where the function \(a\) is independent of \(h\) and, for each fixed initial value \(x(0)\), \(\max_{0\leq
t\leq T}|r(x(0),t,h)|\) tends to \(0\) as \(h\rightarrow 0+\). Thus, for \(h\) sufficiently small, the global
error is approximately equal to \(h^\nu a(x(0),t_n)\): halving \(h\) divides the error by a factor \(2^\nu\).
\end{remark}

\begin{remark}\label{rem:variable_steps}In the description above, the step points \(t_n\) are uniformly spaced.
General purpose software uses \emph{variable time steps:} \(t_{n+1} = t_n+h_n\) where \(h_n\) changes with
\(n\). For each \(n\), the value of \(h_n\) is chosen by the algorithm to ensure  that the local error at
\(x_n\) is below a user-prescribed tolerance. Since, for reasons to be explained in
Section~\ref{ss:geometricintegration}, HMC integrations are carried out with constant step size, we shall not
concern ourselves with variable step sizes.
\end{remark}

\subsection{Splitting methods}
\label{ss:splitting}  Euler's method and the other integrators mentioned above may be applied to any given
system \eqref{eq:ode}.
 Other techniques only make sense for \emph{particular} classes of systems and cannot therefore be
incorporated into general purpose software; however they have gained much popularity
 in the last decades due, among other things, to their role in \emph{geometric integration}
(see Section~\ref{ss:geometricintegration}). Of special importance to us is the class of {\em splitting}
integrators that we consider next. The monograph \cite{BlaCa2016} is a very good source of information.
\subsubsection{The Lie-Trotter formula}
Splitting methods are applicable to cases where the right hand-side of \eqref{eq:ode} may be split into two
parts
\beq\label{eq:splitode}
\frac{d}{dt} x = f(x) = f^{(A)}(x)+f^{(B)}(x),
\eeq
in such a way that the flows \(\varphi_t^{(A)}\) and \(\varphi_t^{(B)}\) of the {\em split systems}
\beq\label{eq:splitodeAB}
 \frac{d}{dt} x = f^{(A)}(x),\qquad \frac{d}{dt} x = f^{(B)}(x),
\eeq
are available analytically. To avoid trivial cases, we shall hereafter assume that the Lie bracket \([
f^{(A)},  f^{(B)}]\) does not vanish identically, because otherwise Theorem~\ref{th:lie} shows that
\eqref{eq:splitode} may also be solved analytically without resorting to numerical approximations.

Systems of the particular form \eqref{eq:newton} provide an important example. When they
are  split by taking
\beq\label{eq:A}
 (A):\qquad\frac{d}{dt} q = M^{-1}p,\qquad \frac{d}{dt} p = 0.
\eeq
and
\beq\label{eq:B}
  (B):\qquad\frac{d}{dt} q = 0,\qquad \frac{d}{dt} p = F(q),
\eeq
the flows are explicitly given by
\[\varphi_t^{(A)}(q,p) = (q+tM^{-1}p, p)
\]
and
\[\varphi_t^{(B)}(q,p) = (q, p+tF(q)).
\]
In molecular dynamics \cite{schlick} these mappings are respectively called a {\em drift} (\(q\) advances with
constant
 speed and the momentum
\(p\) remains constant) and a {\em kick} (the system stays in its current configuration, and the momentum is
incremented by the action of the force). A direct computation of the Lie bracket shows that, except in the
trivial case where \(F\) vanishes identically,  the vector fields in \eqref{eq:A}--\eqref{eq:B} do not
commute.

A simple Taylor expansion proves that the Lie-Trotter formula
\beq\label{eq:lt}
\psi_h = \varphi_h^{(B)}\circ \varphi_h^{(A)}
\eeq
defines a first-order integrator for \eqref{eq:splitode}:
\beq\label{eq:ltlocal}
\psi_h(x) -\varphi_h(x) = \frac{h^2}{2}[f^{(A)},f^{(B)}](x)+\mathcal{O}(h^3).
\eeq
(Observe that, not unexpectedly, the coefficient of the leading power of \(h\) is proportional to the Lie
bracket.)
 While \( f^{(A)}\) and \( f^{(B)}\) contribute simultaneously to the change of \(x\) in
\eqref{eq:splitode}, they do so successively  in the Lie-Trotter integrator. By swapping the roles of \(A\)
and \(B\), we have the alternative integrator
\[
\psi_h = \varphi_h^{(A)}\circ \varphi_h^{(B)}.
\]

\subsubsection{Strang's formula}
The most popular splitting integrator for \eqref{eq:splitode} corresponds to Strang's  formula \cite{strang}
\beq\label{eq:vv}
\psi_h = \varphi_{(1/2)h}^{(B)}\circ \varphi_{h}^{(A)}\circ\varphi_{(1/2)h}^{(B)}
\eeq
and, as a Taylor expansion shows, has second order accuracy, \(\psi_h(x) -\varphi_h(x) = \mathcal{O}(h^2)\) as \(h\rightarrow 0\). More precisely,
\begin{eqnarray}\label{eq:vvlocal}
\psi_h(x) -\varphi_h(x) &= &\frac{h^3}{12}[f^{(A)}, [f^{(A)},f^{(B)}]](x)\nonumber\\
&&\qquad+
\frac{h^3}{24}[f^{(B)}, [f^{(A)},f^{(B)}]](x)+\mathcal{O}(h^4),
\end{eqnarray}
so that the leading term of the local error is  a linear combination of two so-called \emph{iterated Lie
brackets}.

 When applied to the
particular case  \eqref{eq:newton}, the formula  \eqref{eq:vv} yields the well-known \emph{velocity Verlet}
integrator, the method of choice in molecular dynamics; the step \(n\rightarrow n+1\) comprises two kicks of
duration \(h/2\) separated by a drift of duration \(h\):
\begin{eqnarray*}
p_{n+1/2} & = & p_n+\frac{h}{2} F(q_n),\\
q_{n+1} & = & q_n+hM^{-1}p_{n+1/2},\\
p_{n+1} & =& p_{n+1/2} + \frac{h}{2} F(q_{n+1}).
\end{eqnarray*}
The evaluations of \(F\)  represent the bulk of the computational cost of the algorithm. The value
\(F(q_{n+1})\) to be used in the first kick of the next step, \(n+1\rightarrow n+2\), is the same used in the
second kick of the current step. In this way, while the very first step requires two evaluations of \(F\),
subsequent steps only need one.

When \( N\) steps of the method \eqref{eq:vv} are taken, the map that advances the numerical solution
from \(x_0\) to \(x_N\), \ie\
\[ \psi_h^N = \overbrace{
\Big(\varphi_{(1/2)h}^{(B)}\circ \varphi_{h}^{(A)}\circ\varphi_{(1/2)h}^{(B)}\Big)
\circ \cdots\circ
\Big(\varphi_{(1/2)h}^{(B)}\circ \varphi_{h}^{(A)}\circ\varphi_{(1/2)h}^{(B)}\Big)
}^{N \:\rm times}
\]
may be rewritten with the help of the group property \eqref{eq:flow} in the \emph{leapfrog} form
\[
\psi_h^N = \varphi_{(1/2)h}^{(B)}\circ \overbrace{\Big( \varphi_{h}^{(A)}\circ\varphi_{h}^{(B)}\Big) \circ \cdots\circ
\Big( \varphi_{h}^{(A)}\circ\varphi_{h}^{(B)}\Big)
}^{N-1\: \rm times}\circ\,\varphi_{h}^{(A)}\circ
\varphi_{(1/2)h}^{(B)};
\]
now the right hand-side only uses \(N+1\) times the flow \(\varphi_t^{(B)}\). In the particular case
\eqref{eq:A}--\eqref{eq:B}, the combination \(\varphi_{h}^{(A)}\circ\varphi_{h}^{(B)}\) corresponds to the
following formulas to advance the numerical solution, \(n = 1,\dots, N-1\),
\begin{eqnarray*}
p_{n+1/2} & = & p_{n-1/2}+h F(q_n),\\
q_{n+1} & = & q_n+hM^{-1}p_{n+1/2};
\end{eqnarray*}
\(p\) jumps over \(q\) and then \(q\) jumps over \(p\) as children playing leapfrog.
The leapfrog implementation makes apparent the truth of an earlier observation: \(N\) steps of the velocity Verlet integrator
 may be implemented with \(N+1\) evaluations of \(F\).

Strang's method is symmetric:
\[
\begin{aligned}
(\psi_h)^{-1} &= \big(\varphi_{(1/2)h}^{(B)}\big)^{-1}\!\circ
\big(\varphi_{h}^{(A)}\big)^{-1}\!\circ\big(\varphi_{(1/2)h}^{(B)}\big)^{-1}\\
& = \varphi_{-(1/2)h}^{(B)}\circ \varphi_{-h}^{(A)}\circ\varphi_{-(1/2)h}^{(B)}\\
& = \psi_{-h}.
\end{aligned}
\]
It is clear that the symmetry is a consequence of the  \emph{palindromic} structure of \eqref{eq:vv} \ie\ the
formula reads the same from left to right as from right to left.

The roles  of \(A\) and \(B\) in \eqref{eq:vv}, may be interchanged:
\beq\label{eq:pv}
\psi_h =
\varphi_{(1/2)h}^{(A)}\circ \varphi_{h}^{(B)}\circ\varphi_{(1/2)h}^{(A)}.
\eeq
For \eqref{eq:A}--\eqref{eq:B} one then obtains the \emph{position Verlet} integrator: one step comprises two
drifts of duration \(h/2\) and one kick of duration \(h\).

\subsubsection{Splitting formulas with more stages}
It is of course possible to use splitting formulas more sophisticated than Strang's. For instance, for any
choice of the real parameter  \(b\),   we may consider the integrator
\beq \label{eq:comp5} \psi_{h} = \varphi_{bh}^{(B)}\circ\varphi_{(1/2)h}^{(A)} \circ
\varphi_{(1-2b)h}^{(B)}\circ\varphi_{(1/2)h}^{(A)} \circ \varphi_{bh}^{(B)},
\eeq
where we observe that, in one step, the \(A\) and \(B\) flows of the systems \eqref{eq:splitodeAB} act for a
total duration of \(h\) units of time each to ensure consistency. Due to its palindromic structure, this
integrator is symmetric and from Theorem~\ref{th:ordercond} its order is at least  two. It turns out that the
order is exactly two for all choices of \(b\) \cite{BlCaSa2014}. The order of splitting integrators is
discussed later in relation with the concept of modified equations.

Even though three \(B\) flows and two \(A\) flows feature in \eqref{eq:comp5}, \(N\) steps of the integrator
only require the computation of \(2N+1\) \(B\) flows and \(2N\) \(A\) flows; this is seen by combining flows
as we did above for the Strang case. We say that \eqref{eq:comp5}  is a \emph{palindromic two stage}
integrator.\footnote{But \eqref{eq:comp5} is still a \emph{one-step} integrator, because \(x_{n+1}\) is
determined by \(x_n\); for two-step schemes the computation of \(x_{n+1}\) requires the knowledge of both
\(x_n\) \emph{and} \(x_{n-1}\). The term \emph{stage} is borrowed from the Runge-Kutta literature.} The method
\eqref{eq:comp5} may be denoted by
\begin{equation}
\label{eq:twostagefamily}
\big(b,1/2, (1-2b), 1/2, b\big).
\end{equation}

Similarly, one may consider the two-parameter family of \emph{palindromic three-stage} splittings
\begin{equation}
\label{eq:threestagefamily}
\big( b, a, 1/2-b, 1-2a, 1/2-b, a, b
\big).
\end{equation}
A full description of this family is given by \cite{CaSa2017}; this reference suggests parameter choices for
various applications. There is a unique choice of \(a\) and \(b\) resulting in a fourth-order method often
associated with Yoshida's name \cite{yoshida}; for all other choices, the order is \(\nu = 2\).

The family of \emph{palindromic \(s\)-stage}  splitting formulas is given by
\beq\label{eq:even}
 \big(b_1,a_1,b_2,a_2, \dots, a_{s^\prime},b_{s^\prime+1}, a_{s^\prime},\dots,a_2,b_2,a_1,b_1\big),
\eeq
if \(s=2s^\prime\) is even, and by
\beq\label{eq:odd}
\big(b_1,a_1,b_2,a_2, \dots ,b_{s^\prime}, a_{s^\prime},b_{s^\prime},\dots,a_2,b_2,a_1,b_1\big),
\eeq
if \(s=2s^\prime-1\). After imposing the consistency requirement that at each step the \(A\) and \(B\) flows
act during \(h\) units of time each, the family has \(s-1\) parameters left. By taking \(s\) sufficiently high
it is possible to achieve any desired order \cite[Section 13.1]{SaCa1994}. Clearly, increasing the number of
stages does \emph{not} lead to integrators with a more complicated implementation; regardless of the value of
\(s\) numerical integrations just consist of a sequence of flows of the split systems.

It is not necessary to add that to each of the integrators just described there corresponds a second
integrator found by swapping the roles of the systems \(A\) and \(B\), just as  \eqref{eq:pv} corresponds to
\eqref{eq:vv}.

So far our attention has focused on palindromic formulas as these are important later in connection with the
property of reversibility (Theorem~\ref{th:revintegrator}). It is also possible to consider splittings of the
form:
\[
\psi_{h} = \varphi_{b_rh}^{(B)}\circ\varphi_{a_rh}^{(A)} \circ
\cdots
\circ\varphi_{b_2h}^{(B)} \circ \varphi_{a_2h}^{(A)}
\circ\varphi_{b_1h}^{(B)} \circ \varphi_{a_1h}^{(A)},
\]
that we denote as
 \beq\label{eq:generalsplit}
 \big(b_r, a_r, \dots, b_2,a_2,b_1,a_1\big).
 \eeq
(\(\sum_i a_i= \sum_i b_i =1\)). The palindromic formulas in \eqref{eq:even}--\eqref{eq:odd} are particular
instances of this general format because it is always possible to set \(a_1= 0\) in \eqref{eq:generalsplit}.
\begin{remark}
When a splitting integrator of the general form \eqref{eq:generalsplit} is applied to the splitting
\eqref{eq:A}--\eqref{eq:B} of the system \eqref{eq:newton}, the result coincides with the application of a
so-called symplectic, explicit Runge-Kutta-Nystr\"{o}m (RKN) method. The properties (order, stability, etc.)
of such an integrator may therefore be studied either by using techniques pertaining to  splitting integrators
(as done in this paper) or  by employing an RKN approach. The second methodology was favoured in the early
years of geometric integration \cite{acta,SaCa1994}.
\end{remark}

\subsection{Fixed \(h\) stability}
\label{ss:fixedh}
 If two numerical schemes are candidates to integrate an initial value problem for a given
system \eqref{eq:ode}, then the scheme that leads to smaller global errors for a given computational cost may
seem more desirable (but, in the context of geometric integration, the geometric properties may play a role
when choosing the integrator, see Section~\ref{ss:geometricintegration} below). Even though global errors may
be bounded as in \eqref{eq:convergence}, in practice,
 it is almost always impossible, for the problem at hand, to estimate realistically  the error constant implied
in the \(\mathcal{O}(h^\nu)\) notation in the bound. For this reason, the literature on numerical integrators
has traditionally resorted to well-chosen \emph{model problems} where both the numerical  and true solutions,
\(x_n\) and \(x(t_n)\), may be written down in closed form; the performance of the various integrators on the
model problem may then be investigated analytically and is taken as an indication of their performance when
applied to realistic problems. Note that
 the actual numerical solution of the
model problem cannot be expected to be of real practical interest, since in that problem the true solution is
available analytically.

\subsubsection{One degree of freedom}
For our purposes, the model problem of choice is the harmonic oscillator \eqref{eq:harmonic}.  For  Runge-Kutta methods, for the
splitting algorithms in Section~\ref{ss:splitting} (and in fact for all one-step integrators of practical
interest) a time-step $(q_{n+1},p_{n+1})= \psi_h(q_n,p_n)$ may be expressed as
\beq\label{eq:harmonicintegrator}
\left[ \begin{matrix}q_{n+1}\\p_{n+1}\end{matrix}\right] = \tilde{M}_h\left[
\begin{matrix}q_n\\p_n\end{matrix}\right],\qquad \tilde{M}_h= \left[ \begin{matrix}A_h& B_h\\ C_h &
D_h\end{matrix}\right]
\eeq
for suitable method-dependent coefficients $A_h$, $B_h$, $C_h$, $D_h$. The evolution over $n$ time-steps is then given by
\beq\label{eq:harmonicintegrator2}
\left[ \begin{matrix}q_{n}\\p_{n}\end{matrix}\right] = \tilde{M}_h^n\left[
\begin{matrix}q_0\\p_0\end{matrix}\right],
\eeq
an expression to be compared with \eqref{eq:rotation}.

If a given \(h>0\) is such that  \(| \tilde{M}_h^n| \rightarrow \infty\) as \(n\rightarrow \infty\), the
magnitude of numerical solution \((q_n,p_n)\) will grow unboundedly, while the true solution of course remains
bounded as \(t\rightarrow \infty\). Necessarily, global errors will be large for large \(n\). Then the
integrator is said to be \emph{unstable} for that particular choice of \(h\).

\begin{example}\label{ex:stabeuler}
For Euler's rule \eqref{eq:euler} we find
\[
A_h = D_h =1,\quad B_h = -C_h =h.
\]
 The
eigenvalues of \(\tilde{M}_h\) are \(1\pm i h\) with modulus \((1+h^2)^{1/2} \) and, therefore, for \emph{any}
fixed \(h> 0\), the powers \(\tilde{M}_h^n\) grow exponentially as \(n\) increases. Due to this numerical
instability, Euler's rule is completely unsuitable to integrate the harmonic oscillator, and this rules it out
as a method to integrate more complicated oscillatory problems.

More precisely, in the step \(n\rightarrow n+1\), the radius \(r = (q^2+p^2)^{1/2}\), which remains constant
for the true solution, grows like
\beq\label{eq:eulergrowth}
 r_{n+1} = (1+h^2)^{1/2} r_n = \left(1+\frac{h^2}{2}+\mathcal{O}(h^4)\right) r_n
\eeq
for the Euler solution, so that \(r_n = (1+h^2)^{n/2}r_0\). Note that, taking limits as \(n\rightarrow
\infty\), \(h\rightarrow 0\), with fixed \(nh\), we find \(r_n \rightarrow r_0=r(0) = r(nh)\), as it
corresponds to a convergent method.
\end{example}

\begin{example}\label{ex:midpoint}The midpoint rule \eqref{eq:midpoint} has
\[
A_h = D_h = \frac{1-\frac{h^2}{4}}{1+\frac{h^2}{4}},\quad
B_h = -C_h = \frac{h}{1+\frac{h^2}{4}}.
\]
The characteristic equation of \(\tilde M_h\) is
\beq\label{eq:characteristic}
\lambda^2-2A_h\lambda +1
\eeq
and therefore the product of the eigenvalues is 1. For \(h\neq 0\), \(|A_h| <1\) and the matrix has a pair of
complex conjugate eigenvalues of unit modulus. Then the powers \(\tilde M_h^n\) remain bounded as \(n\)
increases and the method is stable for any \(h\).
\end{example}

\begin{example}\label{ex:vv}
Let us next consider Strang's splitting \eqref{eq:vv}, applied with the splitting \eqref{eq:A}--\eqref{eq:B},
which yields the velocity Verlet algorithm. This has
 \[
A_h = D_h =   1-h^2/2,\quad
B_h = h,\quad
C_h = -h+h^3/4;
\]
the characteristic equation is again of the form \eqref{eq:characteristic}. For  \(h>2\), \(A_h<-1\) and the
eigenvalues are real and distinct, so that one of them has modulus \(>1\), and therefore the powers
\(\tilde{M}_h^n\) grow exponentially. For \(h<2\) the eigenvalues are complex of unit modulus and the powers
\(\tilde{M}_h^n\) remain bounded. For \(h=2\), \(\tilde{M}_h\) is a nontrivial Jordan block whose powers grow
linearly (weak instability). To summarize, the integrator is unstable for \(h\geq 2\) and has the
\emph{stability interval} \(0<h<2\). Integrations with \(h>2\) lead to extremely large global errors as we
shall see below.
\end{example}

\begin{example}\label{ex:pv}
For the alternative  Strang formula \eqref{eq:pv},  which yields the position Verlet algorithm, the
coefficients are
 \[
A_h = D_h =   1-h^2/2,\quad
B_h = h-h^3/4,\quad
C_h = -h.
\]
 The discussion is almost identical to
the one in the preceding example. The stability interval is also \(0<h<2\), as one may have guessed from the
equal role that \(q\) and \(-p\) play in the harmonic oscillator \eqref{eq:harmonic}.
\end{example}

\begin{table}[t]
\begin{center}
\begin{tabular}{c c c}
\(h\) & \(t = T_{per}\)   &   \(t = 10\,T_{ per}\) \\
\hline
\(T_{per}/4\) & 6.49e-1 & 2.00e0 \\
\(T_{per}/8\) & 1.60e-1 & 1.48e0 \\
\(T_{per}/16\) & 4.03e-2 & 4.00e-1\\
\(T_{per}/32\) & 1.01e-2 & 1.01e-1
\end{tabular}
\end{center}
\caption{Velocity Verlet integration of the harmonic oscillator.
Relative errors after one or ten oscillation periods}
\label{tab:errorsvv}
\end{table}

\begin{example}\label{ex:v}
 In the situation of Example~\ref{ex:vv}, are
values of \(h\) below the upper limit 2 satisfactory? The answer of course depends on the accuracy required.
Table~\ref{tab:errorsvv} gives, for the initial condition \(q=1\), \(p=0\), and different stable values of
\(h\), the relative error in the Euclidean norm
 \[
\frac{|(q_n-q(t_n), p_n-p(t_n))|} {|(q(t_n),p(t_n))|}
\]
 at the final integration time, when the integration is carried out over an interval of length either one
oscillation period (second column) or ten oscillation periods (third column). Columns are consistent with the
order of convergence as in \eqref{eq:convergence}. The last rows reveal that the error increases linearly with
\(t\) (this will not be true when integrating other differential systems). In the first row the error grows
more slowly than \(t\): the numerical solution \((q_n,p_n)\) remains close to the unit circle for all values
of \(n\) and therefore errors cannot be substantially larger than the diameter of the circle. From the table
we see that if we are interested in  errors  below \(10\%\) over an interval of length equal to ten
oscillation periods, then \(h\) has to be taken below \(2\pi/32\approx 0.20\), \ie\ well below the upper end,
\(h=2\), of the stability interval. To provide an indication of the effect of using unstable values of \(h\),
we mention that with \(h = \pi\) the error after one period is \(\approx 46.4\) and after ten periods
\(\approx 4.68\times 10^{17}\).
\end{example}

Before we move to more complicated models, let us make two observations, valid for  Runge-Kutta, splitting integrators and any other method of practical interest.

\begin{remark}
Because the problem is linear and rotationally invariant, the magnitude of the {\em
relative} errors is independent of the initial condition \((q_0,p_0)\neq(0,0)\).
\end{remark}

\begin{remark}\label{rem:omega}
Replacing the model \eqref{eq:harmonic} with the apparently more general system
\beq \label{eq:harmonicomega}
\frac{dq}{dt} =\omega p,\qquad \frac{dp}{dt} = -\omega q,\qquad \omega >0,
\eeq
with oscillation period \(2\pi/\omega\), does not really change things: in view of Remark~\ref{rem:timescale}
integrating \eqref{eq:harmonicomega} with step size \(h\) is equivalent to integrating \eqref{eq:harmonic}
with step size \(h/\omega\). Because the length of the integration interval and the step size in
Table~\ref{tab:errorsvv} are given in terms of the oscillation period, the results displayed are valid for
\eqref{eq:harmonicomega} for any choice of the value of \(\omega\). Regardless of the initial condition, if we
are interested in relative errors below \(10\%\) over an interval of length \(20\pi/\omega\) (equal to ten
oscillation periods), then \(h\) has to be taken below \(2\pi/(32\omega)\). Of course, for
\eqref{eq:harmonicomega}, stability requires that \(h < 2/\omega\).
\end{remark}

\subsubsection{Several degrees of freedom. Stability restrictions on \(h\)}
Let us now move to the model  with \(d\) degrees of freedom
\[
\frac{d}{dt} q = M^{-1}p,\qquad \frac{d}{dt} p = -Kq,
\]
where \(M\) and \(K\) are \(d\times d\), symmetric, positive-definite matrices. This is the Hamiltonian
system with
\beq\label{eq:MKmodel}
H = \frac{1}{2}p^TM^{-1}p+\frac{1}{2}q^TKq.
\eeq
In mechanics, \(M\) and \(K\) are the mass and stiffness matrices respectively.

This model may transformed into \(d\) uncoupled one-degree-of-freedom oscillators.
 In fact,  factor \(M = LL^{T}\) (which may be done in infinitely many ways)
and diagonalize the symmetric, positive-definite matrix \(L^{-1}KL^{-T}\) as
\(U^{T}L^{-1}KL^{-T}U =\Omega^2\), with \(U\) orthogonal and \(\Omega\) diagonal with diagonal entries
\(\omega_i\), \(i= 1,\dots, d\). A simple computation yields the next result.

\begin{proposition}\label{prop:KM}
With the notation as above, the (non-canonical) change of dependent variables, \(q= L^{-T}U\bar q\), \(p
=LU\Omega\bar p\), decouples the system into a collection of \(d\) harmonic oscillators (superscripts denote
components):
\[
\frac{d\bar q^i}{dt} = \omega_i\bar p^i,\qquad
\frac{d\bar p^i}{dt} = -\omega_i \bar q^i, \qquad i = 1,\dots, d.
\]

%In terms of the new variables the function in \eqref{eq:MKmodel} is given by
%\[
%H = \frac{1}{2}\bar{p}^T\Omega^2\bar{p} + \frac{1}{2}\bar{q}^T\Omega^2\bar{q}.
%\]
\end{proposition}

Now, for all integrators of practical interest, \emph{decoupling and numerical integration commute:} carrying
out the integration  in the old variables \((q,p)\) yields the same result as successively (i) changing
variables in the system, (ii) integrating each of the uncoupled oscillators, (iii) translating the result to
the old variables. \emph{Therefore stability may be analyzed under the assumption that the integration is
performed in the uncoupled version.}

\begin{example}\label{ex:stabilty2df}Consider a particular case with  \(d=2\), \(M= L = I\), \(\omega_1 = 1\) and \(\omega_2 = 100\)
and the initial condition
\[
\bar q^1(0) = 1, \quad \bar p^1(0) = 0, \quad \bar q^2(0) = 0.01,\quad \bar p^2(0) = 0.
\]
We wish to integrate with the velocity Verlet algorithm over \(0\leq t \leq 20\pi\) (ten periods of the slower
oscillation) and aim at absolute errors of magnitude \(\approx 0.1\) in the Euclidean norm in the \(\Reals^4\)
space of the variables \((q^1,q^2,p^1,p^2)\) or, equivalently, because here \(q=U\bar{q}\) with \(U\)
orthogonal, in the Euclidean norm of the variables \((\bar{q}^1,\bar{q}^2,\bar{p}^1,\bar{p}^2)\). On
\emph{accuracy} grounds it would be sufficient to take \(h\approx 0.2\): from Example~\ref{ex:v} we know this
is enough to integrate the first uncoupled oscillator with the desired accuracy, and the second oscillator
should not contribute significantly to the error, due to the smallness of \(\bar q^2(t)\) and \(\bar p^2(t)\)
for all \(t\). However, unless we take \(h\omega_2 < 2\), \ie\ \(h< 0.02\), the errors in \(\bar q^2(t)\) and
\(\bar p^2(t)\) will grow exponentially due to \emph{instability}. In this example, and in many situations
arising in practice, to avoid instabilities, the value of \(h\) has to be chosen much smaller than accuracy
would require. As a result the computational effort to span the time interval of interest would be much higher
than it may be expected on accuracy grounds. These situations are called \emph{stiff}. To deal with stiffness
one may resort to suitable implicit integrators such as the midpoint rule or to explicit integrators with
large stability intervals.
\end{example}
\begin{remark} It may seem that the initial condition in the
pre\-ceding example is somehow contrived as the sizes of two components of \(\bar{q}\) are so unbalanced. This
is not so: it is easily checked that, in terms of the original energy in \eqref{eq:MKmodel},   the
contributions to \(H\) of the \(\omega_1\) and \(\omega_2\) oscillations are both equal to \(1/2\). Those
oscillations are called \emph{normal modes} in mechanics.
\end{remark}

The present discussion is continued in Section~\ref{ss:hmp}.

 We emphasize that, while the preceding material refers to the quadratic Hamiltonian
\eqref{eq:MKmodel}, one may expect that it has some relevance for Hamiltonians close to that model. However
there are cases where the behaviour of a numerical integrators departs considerably from its behaviour when
applied to \eqref{eq:MKmodel}. This point is discussed in Section~\ref{sec:qcubed}. \vspace{2mm}

\section{Geometric integration} \label{ss:geometricintegration}
Classically, the development of numerical integrators for ordinary differential equations focused on
general-purpose methods (such as linear multistep or Runge-Kutta formulas) that were selected after analyzing
their local error and fixed \(h\) stability properties. Geometric integration \cite{geometric} is a newer
paradigm of numerical integration, where the interest lies in meth\-ods tailored to the problem at hand with a
view to preserving some of its geometric features. The development of geometric integration started in the
1980's with the study of symplectic integrators for Hamiltonian systems by Feng Kang and others
\cite{acta,SaCa1994}. Useful monographs are \cite{LeRe2004,HaLuWa2010,FengQin,BlaCa2016}. We review the
geometric integration of Hamiltonian and reversible systems and study the use of modified equations, a key
tool for our purposes. We also examine the behaviour of geometric integrators in the harmonic model problem.
The section concludes showing the optimality of the stability interval of the Strang/Verlet integrator.

\subsection{Hamiltonian problems} Splitting integrators are \emph{symplectic} in the following sense.
Assume that the system \eqref{eq:splitode} is Hamiltonian and that is split in such a way that both  split
systems \eqref{eq:splitodeAB} are also Hamiltonian. Then the splitting integrator mapping \(\psi_h\) is
symplectic, as a composition (Proposition~\ref{prop:compsymp}) of
 flows
 that are individually symplectic (Theorem~\ref{th:symplectic}). This ensures
that the numerical solution shares the specific properties of Hamiltonian flows that derive from
symplecticness. Note that to have a symplectic \(\psi_h\) it is not enough that the system being integrated be
Hamiltonian; if the split vector fields \(f^{(A)}\) and \(f^{(B)}\) are not Hamiltonian themselves, then
\(\psi_h\) cannot be expected to be symplectic. Of course, all possible splittings of a Hamiltonian vector
field as a sum of Hamiltonian vector fields, \(f=f^{(A)}+f^{(B)}\), may be obtained by splitting the
Hamiltonian \(H=H^{(A)}+H^{(B)}\) in an arbitrary way and then setting \(f^{(A)}=J^{-1} \nabla H^{(A)}\),
\(f^{(B)}=J^{-1} \nabla H^{(B)}\). To sum up:

\begin{theorem}\label{th:sympintegrator}
Assume that the Hamiltonian of the system \eqref{eq:hamsys} is written as \(H=H^{(A)}+H^{(B)}\) and  split
correspondingly \(f = J^{-1}\nabla H\) into \(f^{(A)}=J^{-1} \nabla H^{(A)}\), \(f^{(B)}=J^{-1} \nabla
H^{(B)}\). For any splitting integrator \eqref{eq:generalsplit} and any \(h\), the mapping \(\psi_h\) is
symplectic. In particular, \(\psi_h\) conserves  oriented volume.
\end{theorem}

Note that the \(n\)-fold composition \(\psi_h^n\) that advances the numerical solution over \(n\) time-steps
is then also symplectic (Proposition~\ref{prop:compsymp}).

Some \emph{implicit} Runge-Kutta methods, including the midpoint rule and Gauss methods, are also symplectic:
\(\psi_h\) is a symplectic map whenever the system \eqref{eq:ode} being integrated is Hamiltonian as in
\eqref{eq:hamsys}. No explicit Runge-Kutta method is symplectic. In particular, Euler's rule is not
symplectic; according to Example~\ref{ex:stabeuler} the Euler discretization of the harmonic oscillator
\emph{increases} area, as
 \(\tilde{M}_h\) has determinant \(1+h^2\). This increase in area is related to
 the estimate \eqref{eq:eulergrowth} that shows that the Euler solution spirals outward, a
non-Hamiltonian behaviour.

In the particular case of the Hamiltonian in \eqref{eq:sepham}, taking \(\mathcal{T}\) and \(\mathcal{U}\) to
play the roles of \(H^{(A)}\) and \(H^{(B)}\) respectively leads to the splitting of the differential system
given in \eqref{eq:A}--\eqref{eq:B} with   \( F = -\nabla \mathcal{U}\).

 It would also be desirable to have integrators that preserved energy when applied to the Hamiltonian system
\eqref{eq:hamsys}, \ie\ \(H\circ\psi_h = H\). Unfortunately, for realistic problems such a requirement is
incompatible with \(\psi_h\) being symplectic \cite[Section 10.3]{SaCa1994}. It is then standard practice to
insist on symplecticness and sacrifice conservation of energy. There are several reasons for this.
Symplecticness plays a key role in the Hamiltonian formalism (cf.~Theorem~\ref{th:symplectic}). In addition,
while, as we have seen, it is not difficult to find symplectic formulas, standard classes of integrators do
not include energy-preserving schemes except if the energy is assumed to have  particular forms. Finally, as
we shall discuss below, symplectic schemes have small energy errors even when the integration interval is very
long.

\subsection{Reversible problems}

For reversible systems we have, from Proposition~\ref{prop:comprever} and Theorem~\ref{th:reverflow}:

\begin{theorem}\label{th:revintegrator}
Assume that \eqref{eq:splitode}, and the split systems \eqref{eq:splitodeAB} are reversible with respect to
the same involution \(S\). If the system is integrated by means of a \emph{palindromic} splitting integrator
\eqref{eq:even} or \eqref{eq:odd}, then, for any \(h\), the mapping \(\psi_h\) will also be reversible.
\end{theorem}

The midpoint rule and Gauss Runge-Kutta formulas also generate reversible mappings \(\psi_h\) whenever they
are applied to a reversible system. No explicit Runge-Kutta method is reversible.

The \(n\)-fold composition \(\psi_h^n\) that advances the numerical solution over \(n\) time-steps is then
also reversible (Proposition~\ref{prop:comprever}). However note that, if variable time steps were taken, then
 the mapping \(\psi_{h_n}\circ\cdots\circ\psi_{h_1}\), that advances the solution from \(t_0\) to \(t_{n+1}\),
 would not be reversible. This is one of the reasons for not considering here variable time steps; see
 Remark~\ref{rem:variable_steps}.

The use of reversible integrators (with constant step sizes) ensures that the numerical solution inherits
relevant geometric properties of the true solution of the differential system \cite{CaSa1997,CaSa1998}.

\subsection{Modified equations}
 Modified equations are rather old (see references in \cite{GrSa1986}); however their use as a means to analyse
numerical integrators has only become popular  in the last twenty years, after the emergence of geometric
integration \cite{Sanz1996}.

\subsubsection{Motivation}
Let us first look at some examples:
\begin{example}Suppose that the system \eqref{eq:ode} is solved with Euler's rule \eqref{eq:euler} and
denote by \(\psi_h\) the corresponding map \(x+hf(x)\). Consider the new system, parameterized by \(h\),
\beq\label{eq:modeuler}
 \frac{d}{dt} x = \tilde{f}_h(x),\qquad \tilde{f}_h(x)  = f(x)-\frac{h}{2}f^\prime(x) f(x),
\eeq
 with flow \(\tilde{\varphi}_t\) (for simplicity, the dependence of this flow on the parameter \(h\) is not
incorporated in the notation). By proceeding as in the derivation of \eqref{eq:eulerexpansion}, we find that
the Taylor expansion of \(\tilde{\varphi}_h\) in powers of \(h\) is
\begin{eqnarray*}
\tilde{\varphi}_h(x)& = &
x + h \tilde{f}_h(x) +\frac{h^2}{2} \tilde{f}^\prime_h(x)\tilde{f}_h(x)+\mathcal{O}(h^3)\\
&=& x + h \left(f(x)-\frac{h}{2}f^\prime(x) f(x)\right)+\frac{h^2}{2} {f}^\prime_h(x){f}_h(x)+\mathcal{O}(h^3)\\
&=& x+hf(x) + \mathcal{O}(h^3).
\end{eqnarray*}
Thus, the Euler mapping \(\psi_h(x)\), which differs from the flow of the system \eqref{eq:ode} being solved
in \(\mathcal{O}(h^2)\) terms (first-order consistency),
 differs from the flow \(\tilde{\varphi}_h\) of the so-called
\emph{modified or shadow system} \eqref{eq:modeuler} in \(\mathcal{O}(h^3)\) terms (second-order consistency).
According to Theorem~\ref{th:convergence} (see \cite{GrSa1986} for details), over bounded time intervals, the
Euler solution differs from the corresponding solution of the modified system  in \(\mathcal{O}(h^2)\) terms.
Therefore we may expect that the properties of the Euler solutions for \eqref{eq:ode} will resemble the
properties of the solutions of \eqref{eq:modeuler} more than they resemble the properties of solutions of
\eqref{eq:ode} itself.
 When studying the properties of Euler's rule, working with
\eqref{eq:modeuler} rather than with \(\psi_h\) should be advantageous, because differential equations are
simpler to analyze than maps. An illustration is provided next.
\end{example}
\begin{example}Let us particularize the preceding example to the case of the harmonic oscillator
\eqref{eq:harmonic}. The modified system \eqref{eq:modeuler} is found to be
\[
\frac{dq}{dt} = p+\frac{h}{2} q,\qquad \frac{dp}{dt} = -q+\frac{h}{2} p,
\]
and from here a simple computation yields for the radius \(r = (q^2+p^2)^{1/2}\)
\[
\frac{dr}{dt} = \frac{h}{2} r,
\]
so that, over a time interval of length \(h\), \(r\) is multiplied by the factor \(\exp(h^2/2) =
1+h^2/2+\mathcal{O}(h^3)\).  This is precisely what we found in \eqref{eq:eulergrowth} for the Euler solution
of the harmonic oscillator (but  the \(\mathcal{O}(h^3)\) remainder here will not coincide with the one in
\eqref{eq:eulergrowth}, because there is an \(\mathcal{O}(h^3)\) difference between \(\psi_h\) and
\(\tilde{\varphi}_h\)).
\end{example}

\begin{example} In \eqref{eq:modeuler}, \(\tilde f_h(x)\) is a first-degree polynomial in \(h\). If \(\tilde f_h(x)\)  is
chosen to be quadratic in \(h\), \ie\ \(\tilde{f}_h = f-(h/2) f^\prime f+h^2 f_2\), it is possible to
determine \(f_2\) so as to achieve \(\psi_h-\tilde{\varphi}_h=\mathcal{O}(h^4)\) for the Euler map \(\psi_h(x)
= x+hf(x)\). Similarly, taking \(\tilde f_h(x)\)  as a suitable chosen polynomial in \(h\) with degree
\(\mu=3,4\dots\) it is possible to achieve \(\psi_h-\tilde{\varphi}_h=\mathcal{O}(h^{\mu+2})\).
\end{example}
\subsubsection{Definition}

Given a (consistent) integrator \(\psi_h\) for the system \eqref{eq:ode}, there exists a (unique) \emph{formal
series in powers of \(h\)}
\beq\label{eq:inf}
 \tilde{f}_h^\infty(x) = f(x) + h g^{[1]}(x)+ h^2g^{[2]}(x)+\cdots
\eeq
(the \(g^{[\mu]}\) map the space \(\Reals^D\) into itself) with the property that, for each \(\mu= 0,
1,\dots\), the flow \(\tilde{\varphi}_h^{[\mu]}\) of the modified system
\beq\label{eq:modsystrunc}
\frac{d}{dt} x = \tilde{f}_h^{[\mu]}(x),\quad\tilde{f}_h^{[\mu]}(x) = f(x) + h g^{[1]}(x)+\cdots+h^\mu g^{[\mu]}(x),
\eeq
satisfies
\beq\label{eq:modbound}
\psi_h(x) - \tilde{\varphi}_h^{[\mu]}(x)= \mathcal{O}(h^{\mu+2}).
\eeq
Furthermore, for a symmetric method, the odd-numbered \(g^{[\mu]}(x)\) are identically zero:
\beq\label{eq:gsymmetic}
g^{[1]}(x) = g^{[3]}(x) = g^{[5]}(x)= \cdots = 0.
\eeq

As \(\mu\) increases, the solutions of the modified system \eqref{eq:modsystrunc} provide better and better
approximations to \(\psi_h\). For some vector fields \(f\) and some integrators (see  an example in
Proposition~\ref{prop:modham} below), it may happen that the formal series \eqref{eq:inf} converges for each
\(x\) (so that  \(\tilde{f}_h^\infty\) is a well-defined mapping  \(\Reals^D:\rightarrow\Reals^D\)) and that
furthermore the flow \(\tilde{\varphi}_h^{[\infty]}\) of
\beq\label{eq:modsysinf}
\frac{d}{dt} x = \tilde{f}_h^\infty(x)
\eeq
coincides with \(\psi_h\). In those cases, \eqref{eq:modsysinf} provides an  \emph{exact} modified system to
study \(\psi_h\). However such situations are exceptional, because, in general, discrete dynamical systems
(such as the one generated by \(\psi_h\)) possess features that cannot appear in flows of autonomous
differential systems. It is the lack of convergence of   \eqref{eq:inf} that makes it necessary  to consider
the truncations  in \eqref{eq:modsystrunc} which do not exactly reproduce \(\psi_h\) but approximate it with
an error as in \eqref{eq:modbound}. It is possible, by increasing \(\mu\) as \(h\) becomes smaller and
smaller, to render the discrepancy \(\psi_h(x) - \tilde{\varphi}_h^{[\mu(h)]}(x)\) exponentially small with
respect to \(h\), as first proved by Neishtadt \cite[Section 10.1]{SaCa1994}.

\subsubsection{Finding explicitly the modified equations}

For splitting integrators, the terms of the series \eqref{eq:inf} may be found explicitly by using the
Baker-Hausdorff-Campbell formula. The next theorem provides a summary of some key points; a more detailed
description is given in Section~\ref{sec:bch}.

\begin{theorem}\label{th:modinfsplit}
Assume that system \eqref{eq:splitode} is integrated by means of a (consistent) splitting algorithm of the
general format \eqref{eq:generalsplit}. The series \eqref{eq:inf} is of the form
\begin{eqnarray*}
&&\big(f^{(A)}+f^{(B)}\big) + h C_{1,1} [f^{(A)},f^{(B)}]\\
&&
\qquad
{}+ h^2 \Big(C_{2,1}[f^{(A)},[f^{(A)},f^{(B)}]]+
 C_{2,2}[f^{(B)},[f^{(A)},f^{(B)}]]\Big) + \cdots
\end{eqnarray*}
where, for each \(\mu=1,2\dots\), the coefficient of \(h^\mu\) is a linear combination of linearly independent
 iterated commutators involving \(\mu+1\) fields \(f^{(A)}\), \(f^{(B)}\).
 The coefficients \(C_{i,j}\) that appear in the linear combinations are known polynomials on the coefficients
\(a_k\) and \(b_\ell\) that appear in the formula \eqref{eq:generalsplit}.

If the splitting is palindromic then all the coefficients \(C_{2i+1,j}\) corresponding to
 odd powers of \(h\) vanish.
\end{theorem}

For Runge-Kutta methods it is also possible to give expression for the terms in the series \eqref{eq:inf}, see
\cite{CaMuSa,Sanz1996}.

\subsubsection{Modified equations and order}
If in \eqref{eq:inf}, the functions \(g^{[1]}\), \dots , \(g^{[\mu-1]}\), \(\mu\geq 2\), vanish  so that
\[
\tilde{f}_h^{[\mu]}(x)-f(x) = h^\mu g^{[\mu]}(x)+ \mathcal{O}(h^{\mu+1})
\]
then,
the flow
\(\tilde{\varphi}_h^{[\mu]}\) of \eqref{eq:modsystrunc} and the flow \(\varphi_h\) of \eqref{eq:ode} will satisfy
\[
\tilde{\varphi}_h^{[\mu]}(x)- \varphi_h(x) = h^{\mu+1} g^{[\mu]}(x)+ \mathcal{O}(h^{\mu+2})
\]
 which, in view of \eqref{eq:modbound}
implies that
 \[\psi_h(x)-\varphi_h(x) = h^{\mu+1} g^{[\mu]}(x)+ \mathcal{O}(h^{\mu+2})\] \ie\ that the integrator
 \(\psi_h(x)\)
has order \(\mu\) (or higher). The converse is also true, because the argument may be reversed: order \(\geq
\mu\) implies that \(g^{[1]}\), \dots , \(g^{[\mu-1]}\) must vanish. From \eqref{eq:gsymmetic}, the order of a
symmetric methods must be an even integer. To sum up, we have the following result, which makes it possible to
use the series \eqref{eq:inf} rather than the mapping \(\psi_h\) to study the order of a given integrator.
\begin{theorem}\label{th:ordercond}
A (consistent) integrator has  order \(\geq \nu\), \(\nu\geq 2\), if and only if the functions
    \(g^{[1]}\), \dots , \(g^{[\nu-1]}\) appearing in \eqref{eq:inf} are identically zero.

When the order is exactly \(\nu\), \(\nu= 1, 2, \dots\), the leading term of the truncation error is
\(h^{\nu+1}
    g^{[\nu]}(x)\).

A (consistent) symmetric integrator has even order.
\end{theorem}

This theorem, in tandem with Theorem~\ref{th:modinfsplit}, provides the standard way to write down the order
conditions for splitting methods, \ie\ the  relations on the coefficients \(a_k\), \(b_\ell\) that are
necessary and sufficient for an integrator to have order \(\nu\). Specifically, order \(\geq \nu\) is
equivalent to \(C_{i,j}= 0\) whenever \(i< \nu\). Here are the simplest illustrations:
\begin{itemize}
\item The Lie-Trotter formula \eqref{eq:lt} has \(\nu=1\), and, according to Theorems~\ref{th:modinfsplit}
    and \ref{th:ordercond}, the leading term of the local error is a constant multiple of
    \(h^2[f^{(A)},f^{(B)}]\), which matches our earlier finding in \eqref{eq:ltlocal}.
\item The condition \(C_{1,1} = 0\) is necessary and  sufficient for an integrator to have order \(\nu\geq
    2\). It is automatically satisfied for palindromic formulas.
\item Strang's formula \eqref{eq:vv} has order exactly 2. The coefficient of \(h^3\) in the expansion of the
    local error is a combination of iterated commutators with three fields. This was already found in
    \eqref{eq:vvlocal}.
\item To have order \(\nu\geq 3\) we have
    to impose the order
conditions \(C_{1,1} = 0\), \(C_{2,1} = 0\), \(C_{2,2} = 0\).
\end{itemize}

\subsubsection{Modified equations and geometric integrators}

The geometric properties of integrators have a clear impact on their modified systems.

Let us begin with the Hamiltonian case. If the split fields in \eqref{eq:splitodeAB} are Hamiltonian with
Hamiltonian functions \(H^{(A)}\) and \(H^{(B)}\), then, by invoking Theorem~\ref{th:poisson}, it is feasible
to work with iterated Poisson brackets of \(H^{(A)}\) and \(H^{(B)}\) rather than with the iterated Lie
brackets of \(f^{(A)}\) and \(f^{(B)}\). Then the expansion in Theorem~\ref{th:modinfsplit} is replaced by
\begin{eqnarray*}
&&\big(H^{(A)}+H^{(B)}\big) - h C_{1,1} \{H^{(A)},H^{(B)}\}\\
&&
\quad
{}+ h^2 \Big(C_{2,1}\{H^{(A)},\{H^{(A)},H^{(B)}\}\}+
 C_{2,2}\{H^{(B)},\{H^{(A)},H^{(B)}\}\}\Big) + \cdots
\end{eqnarray*}
For each \(\mu\), \emph{the modified system \eqref{eq:modsystrunc} is a Hamiltonian system}. This is  a
reflection of the fact that, according to Theorem~\ref{th:sympintegrator}, splitting integrators give rise to
mappings \(\psi_h\) that are symplectic. In the class of Runge-Kutta methods it is also true that symplectic
integrators have modified systems that are Hamiltonian.

Speaking informally, we may say that all integrators change the system being integrated into a modified
system; in nonsymplectic methods the (perhaps small) change is such that the modified system is no longer
Hamiltonian. Symplectic methods are those that change Hamiltonian systems into Hamiltonian systems. This
heuristic description cannot be made entirely rigorous because, as pointed out above, the exact modified
system \eqref{eq:modsysinf} only exists in a formal sense due to the lack of convergence of the series in
\eqref{eq:inf}. The existence of Hamiltonian modified system is at the basis of many  favourable properties of
symplectic integrators.

\begin{remark}The considerations above only make sense if the step size \(h\) is held constant along the
integration interval. Since the modified system changes with \(h\), variable step size implementations of
symplectic integrators do not have  well-defined modified systems and in fact their behaviour is closer to
that of non-symplectic integrators than to that of symplectic integrators used with constant step sizes
\cite{CaSa1993}.
\end{remark}

For the reversible case, consider  the situation of Theorem~\ref{th:revintegrator}. From
Proposition~\ref{prop:revercomm} all iterated commutators involving an odd number of fields are themselves
\(S\)-reversible. Since, according to Theorem~\ref{th:modinfsplit}, for \emph{palindromic splitting
integrators} the iterated commutators with an even number of fields enter the expansion \eqref{eq:inf} with
null coefficients, then \emph{the modified systems \eqref{eq:modsystrunc} are reversible}.

\subsubsection{Conservation of energy
by symplectic integrators} As noted above, Neishtadt proved that, under suitable regularity assumptions, the
modified system may be chosen so as to ensure that its \(h\)-flow is exponentially close to the mapping
\(\psi_h\). For symplectic integrators, the modified system is Hamiltonian and therefore exactly preserves its
own Hamiltonian function that we denote by \(\tilde H_h\). It follows that, except for exponentially small
errors, \(\psi_h\) preserves \(\tilde H_h\). On the other hand for a symplectic  integrator of order \(\nu\),
the difference between \(\tilde H_h\) and \(H\) is \(\mathcal{O}(h^\nu)\) (Theorem~\ref{th:ordercond}). These
considerations make it possible to prove that symplectic integrators preserve the value of the Hamiltonian
\(H\) of the system being integrated with error \(\mathcal{O}(h^\nu)\) over time intervals \(0\leq t\leq T_h\)
whose length \(T_h\) increases exponentially as \(h\rightarrow 0\) \cite{HaLuWa2010}.

For linear problems, an exact modified system exists and using the same argument, we may conclude that the
error in energy of a symplectic integrator has an \(\mathcal{O}(h^\nu)\) bound over the infinite interval
\(0\leq t < \infty\), or in other words the energy error may be bounded independently of the number of step
taken. For the case of the harmonic oscillator, this will be illustrated presently.

\subsection{Geometric integrators and the harmonic model problem}
\label{ss:hmp}
 We take  again the integration of the harmonic oscillator as a model problem (see
\eqref{eq:harmonicintegrator}--\eqref{eq:harmonicintegrator2}); now our interest is in studying in detail the
behaviour of geometric integrators.

We focus on (consistent) integrators that are both  \emph{symplectic  and reversible.} In terms of the matrix
\(\tilde M_h\), the first of these properties corresponds to $A_hD_h-B_hC_h = 1$ and, when this condition
holds, reversibility is equivalent to \(A_h=D_h\). Our treatment follows \cite{BlCaSa2014}.

The characteristic polynomial of \(\tilde M_h\)
 is of the form \eqref{eq:characteristic} and there are four possibilities, the first two
correspond to \emph{unstable} simulations and the other two to \emph{stable} simulations:
\begin{itemize}
\item $h$ is such that $|A_h| > 1$. In that case $\tilde{M}_h$ has spectral radius $>1$ and therefore the
    powers $\tilde{M}_h^n$ grow exponentially with $n$.
\item $A_h = \pm 1$ and $|B_h|+|C_h| >0$. The powers  $\tilde{M}_h^n$ grow linearly with $n$.
\item  $A_h = \pm 1$, $B_h=C_h=0$, \ie\ $\tilde{M}_h =\pm I$, \(\tilde{M}_h^n=(\pm I)^n\).
\item  $h$ is such that $|A_h|< 1$. In that case, $\tilde{M}_h$  has complex conjugate eigenvalues of unit
    modulus and the powers $\tilde{M}_h^n$, $n = 0, 1,\dots$ remain bounded.
\end{itemize}

Comparing \eqref{eq:harmonicintegrator} with the result of setting \(t=h\) in \eqref{eq:rotation}, we see
that, by consistency, \(B_h = h+\mathcal{O}(h^2)\) and \(C_h = -h+\mathcal{O}(h^2)\), and therefore \(A_h =
(1+B_hC_h)^{1/2} = 1-h^2/2+\mathcal{O}(h^3)\). Thus, for \(h>0\) sufficiently small, \(A_h<1\) and the
integration will be stable. The \emph{stability interval} of the integrator is the longest interval
\((0,h_{max})\) such that integrations with \(h\in (0,h_{max})\) are stable. For reasons discussed in
Example~\ref{ex:stabilty2df} methods with long stability intervals are often to be favoured. From
Example~\ref{ex:midpoint}, the midpoint rule has stability interval \((0,\infty)\). Explicit integrators have
stability intervals of finite length.

For each $h$ such that $|A_h|\leq 1$, it is expedient to introduce $\theta_h\in\mathbb{R}$ such that $A_h =
D_h = \cos \theta_h$. For $|A_h|< 1$, we have $\sin\theta_h\neq 0$ and we may define
\begin{equation}\label{eq:chi}
 \chi_h = B_h/\sin \theta_h.
\end{equation}
In terms of $\theta_h$ and $\chi_h$, the matrices in (\ref{eq:harmonicintegrator}) and (\ref{eq:harmonicintegrator2}) are then
\begin{equation}\label{eq:tildemh}
\tilde{M}_h
=
\left[ \begin{matrix}\cos \theta_h & \chi_h\sin \theta_h\\ -\chi_h^{-1}\sin \theta_h & \cos \theta_h\end{matrix}\right]
\end{equation}
and
\begin{equation}\label{eq:tildemhdos}
\tilde{M}_h^n
=
\left[ \begin{matrix}\cos (n\theta_h) & \chi_h\sin (n\theta_h)\\ -\chi_h^{-1}\sin (n\theta_h) & \cos (n\theta_h)\end{matrix}\right]
.
\end{equation}
For a value of \(h\) in the (stable) case $A_h = \pm 1$, $B_h=C_h=0$, one has $\sin\theta_h=0$, so that
\eqref{eq:chi} does not make sense. However
 the matrix $\tilde M_h$ is of the form (\ref{eq:tildemh}) for any choice of $\chi_h$. (Typically,
 for  such a value of \(h\), one may avoid the indeterminacy in the value of $\chi_{h}$
by taking limits as $\epsilon\rightarrow 0$ in $\chi_{h+\epsilon} = B_{h+\epsilon}/\sin \theta_{h+\epsilon}$.)

For a method of order $\nu$, $\chi_h = 1+\mathcal{O}(h^\nu)$, $\theta_h = h + \mathcal{O}(h^{\nu+1})$ as
$h\rightarrow 0$. By comparing the numerical $\tilde{M}_h^n$ in (\ref{eq:tildemhdos}) with the true $M_{nh}$
in (\ref{eq:rotation}), one sees that a method with $\theta_h = h$ would have
 no phase error: the angular frequency of the rotation of the numerical solution would coincide
with the true angular frequency of the harmonic oscillator. More generally, the difference \(\theta_h-h\)
governs the phase error. According to \eqref{eq:tildemhdos}, this  phase error grows \emph{linearly} with
\(n\) (recall Table~\ref{tab:errorsvv}). On the other hand, a method with $\chi_h = 1$ would have no energy
error: the numerical solution would remain on the correct level curve of the Hamiltonian i.e.\ on the circle
$p^2+q^2 = p_0^2+q_0^2$. The discrepancy between \(\chi_h\) and \(1\) governs the energy errors. In
\eqref{eq:tildemhdos} we see that these are \emph{bounded} as \(n\) grows.

The preceding considerations may alternatively be understood by considering the modified Hamiltonian given in
the next result.

\begin{proposition}\label{prop:modham}
Consider the application to  the harmonic oscillator \eqref{eq:harmonic} of a (consistent) reversible,
volume-preserving integrator (\ref{eq:harmonicintegrator})   and assume that the step size $h$ is stable, so
that
 $\tilde M_h$ may be written in the form (\ref{eq:tildemh}). Then $\psi_h$ exactly coincides with the
\(h\) flow of the modified Hamiltonian
$$
\tilde{H}_h = \frac{\theta_h}{2h}\left(\chi_h p^2+ \frac{1}{\chi_h}q^2\right).
$$
\end{proposition}
\begin{figure}[t]
\begin{center}
\includegraphics[width=0.45\textwidth]{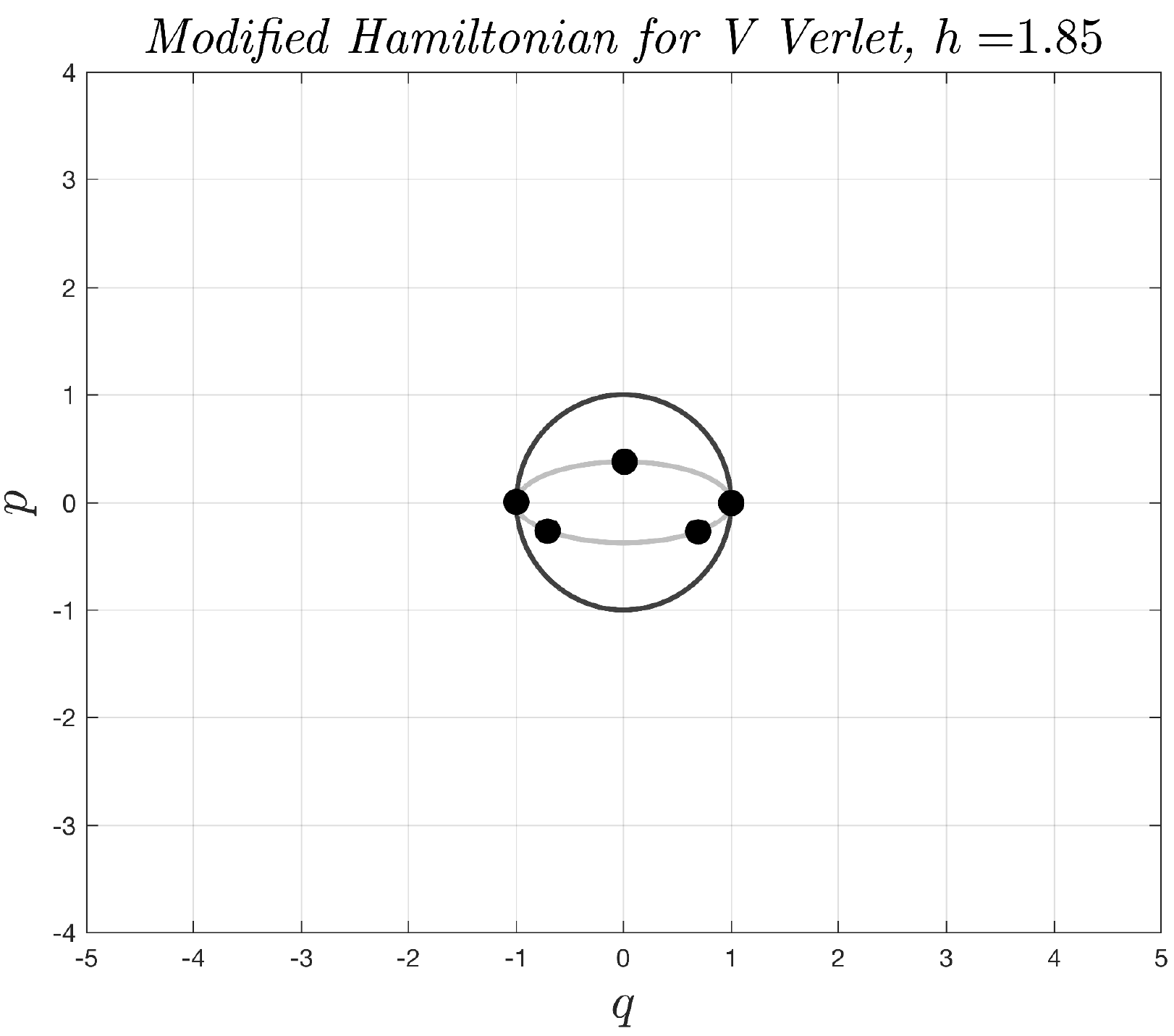}
\includegraphics[width=0.45\textwidth]{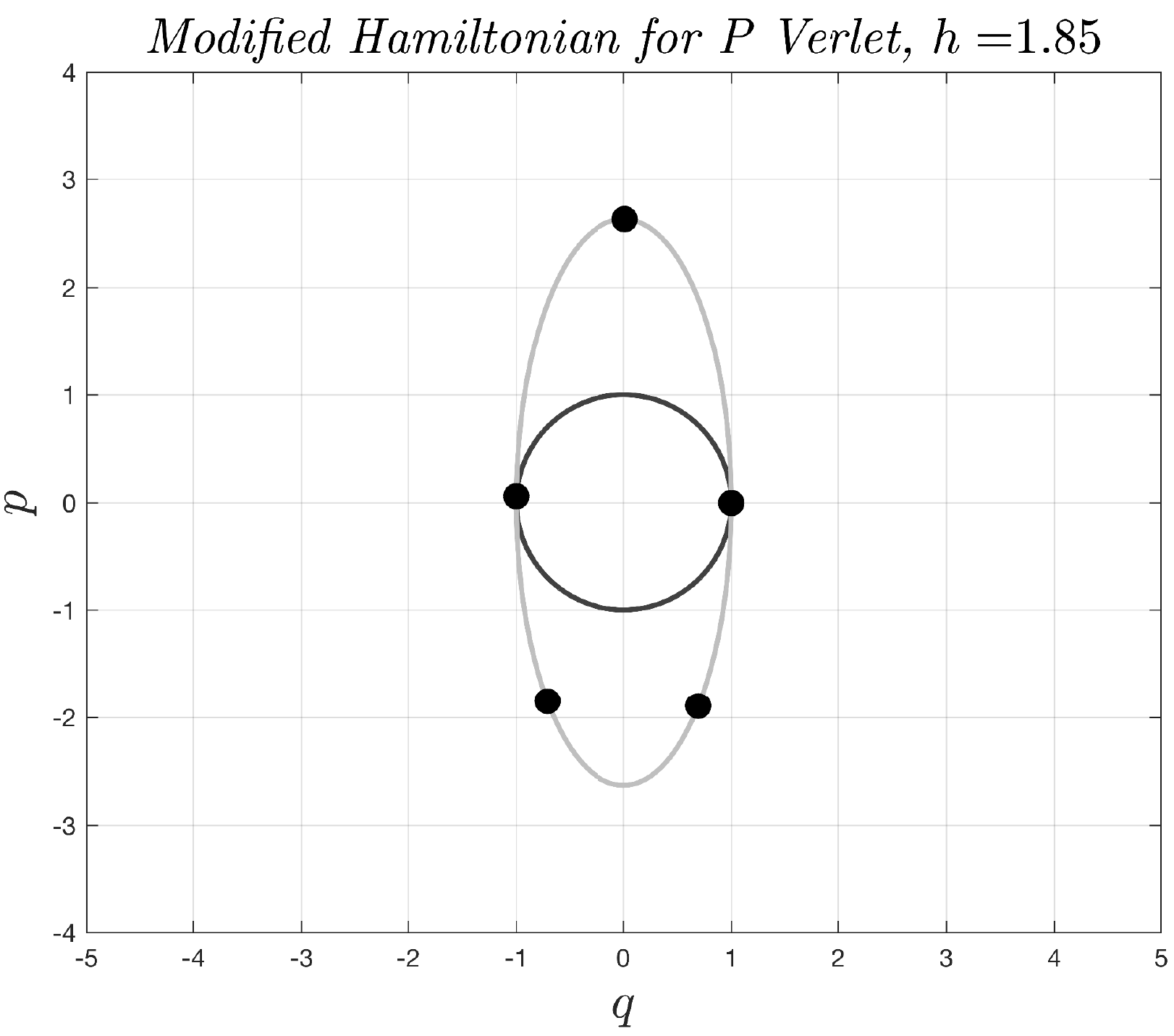}
\end{center}
\caption{ \small  Modified Hamiltonian for velocity (left panel)
 and position (right panel) Verlet.  The dots are the first five points along a
discrete orbit of the velocity/position Verlet integrator initiated at $(q,p) = (1,0)$ and
 the gray line provides the associated level set of the modified Hamiltonian.  For comparison,
 the black contour line shows the corresponding level set of the true Hamiltonian.
 }
 \label{fig:verlet_modified_hamiltonian}
\end{figure}

In particular numerical trajectories are contained in ellipses
 \beq\label{eq:ellipse}
\chi_hp^2+\frac{1}{\chi_h}q^2= \chi_hp_0^2+\frac{1}{\chi_h}q_0^2
\eeq
rather than in circles (Figure~\ref{fig:verlet_modified_hamiltonian}).

\begin{remark}\label{rem:swap}
A comparison of a given integrator (\ref{eq:harmonicintegrator}) with (\ref{eq:rotation}) shows that
\begin{equation}\label{eq:harmonicintegratorswap}
\left[ \begin{matrix}q_{n+1}\\p_{n+1}\end{matrix}\right] = \tilde{M}_h\left[ \begin{matrix}q_n\\p_n\end{matrix}\right],\qquad
\tilde{M}_h=
\left[ \begin{matrix}A_h& -C_h\\ -B_h & D_h\end{matrix}\right]
\end{equation}
is a second integrator of the same order of accuracy. We may think that \eqref{eq:harmonicintegratorswap}
arises from \eqref{eq:harmonicintegrator} by changing the roles of the variables \(q\) and \(p\). In the
particular case of splitting integrators, \eqref{eq:harmonicintegratorswap} arises from
\eqref{eq:harmonicintegrator} after swapping the roles of the split systems \(A\) and \(B\). The integrators
(\ref{eq:harmonicintegrator}) and (\ref{eq:harmonicintegratorswap}) share the same interval of stability and
the same $\theta_h$. The function $\chi_h$ of (\ref{eq:harmonicintegratorswap}) is obtained by changing the
sign of the reciprocal of the function  $\chi_h$ of (\ref{eq:harmonicintegrator}). The important function
$\rho(h)$ to be introduced in Proposition~\ref{prop:average} is also the same for
(\ref{eq:harmonicintegrator}) and (\ref{eq:harmonicintegratorswap}). The velocity Verlet algorithm and the
position Verlet algorithm provide an example of this kind of pair of integrators (see Examples~\ref{ex:vv} and
\ref{ex:pv}).
\end{remark}

The results we have just presented may be extended  to quadratic Hamiltonians with \(d\) degrees of freedom
\eqref{eq:MKmodel}: it is sufficient to use diagonalization as in Proposition~\ref{prop:KM}. In particular,
for stability we require that the stability interval of the integrator contains  all products \(h\omega_i\),
where the frequencies \(\omega_i\) are the square roots of the eigenvalues of \(L^{-1}KL^{-T}\), with
\(LL^T=M\).

\subsection{Optimal stability of Strang's method}
\label{ss:optimal}
 Let us fix an integer \(N\) and consider consistent
palindromic splitting integrators \eqref{eq:even}--\eqref{eq:odd} with \(s=N\); these use \(N\) evaluations of
\(F\) per step when applied to problems of the form \eqref{eq:A}--\eqref{eq:B}. The corresponding coefficient
\(A_h\) in \eqref{eq:harmonicintegrator} is a polynomial of degree \(N\) in the variable \(z = h^2\) (for obvious reasons \(A_h\) is often called the \emph{stability polynomial} of the integrator).
We pointed out above that consistency imposes the relation \(A_h = 1-z/2+\mathcal{O}(h^2)\). Our aim
is to identify, among the class just described, the polynomial \(A_h(z)\) that satisfies \(|A_h(z)|<1\) for
\(0<h<h_{\rm max}\) with \(h_{\rm max}\) as large as possible.
 As we shall show presently, the
\(A_h(z)\) sought corresponds to the integrator
\beq\label{eq:opt}
 \psi_h = \overbrace{\psi^{V}_{h/N} \circ \cdots\circ \psi^{V}_{h/N}}^{N{\rm times}},
\eeq
where \(\psi_h^{V}\) is the mapping associated with the Strang/Verlet formula \eqref{eq:vv}.\footnote{Of
course if, rather than in integrators of the format \eqref{eq:even}--\eqref{eq:odd}, one is interested in the
corresponding palindromic integrators that start and end with an A flow, then, in the right-hand side of
\eqref{eq:opt}, one has to use \eqref{eq:pv} rather than \eqref{eq:vv}. } Note that to carry out a step of length \(h\) with the method in \eqref{eq:opt} one just has to take \(N\) consecutive steps of length \(h/N\) of standard velocity Verlet.
In other words, \emph{subject to
stability, if one wishes to take as long a step as possible with a budget of \(N\) evaluations of the force
per step, the best choice is to concatenate \(N\) steps of Strang/Verlet.}\footnote{More precisely, if one is
interested in methods that start with a kick (resp.\ drift) one has to concatenate the velocity (resp.\
position) version of Verlet.}

To see the optimality of \(\psi_h\) in \eqref{eq:opt}, we first note that, after expressing \(\psi_{h/N}^V\)
in terms of the \(A\) and \(B\) flows and merging consecutive \(B\) flows, the mapping \eqref{eq:opt}
corresponds indeed to a palindromic splitting with \(N\) stages. Then, from Example~\ref{ex:vv} we know that
Verlet is stable for \(0<h<2\) and this implies that \eqref{eq:opt} is stable for \(0<h/N<2\) (for
\eqref{eq:opt} the powers of \(\tilde{M}_h\) are powers of the Verlet matrix \(\tilde{M}^V_{h/N}\)). In this
way \eqref{eq:opt} has stability interval \((0,2N)\) and we shall prove next that this is the longest
possible. From Example~\eqref{ex:vv}, Verlet with step size \(h/N\)  has  \(A_{h/N}^V= 1-h^2/(2N^2)\), which,
in view of \eqref{eq:tildemh}--\eqref{eq:tildemhdos}, implies that for \eqref{eq:opt} the coefficient \(A_h\)
has the expression
\[
 \cos \theta_h = \cos(N\theta_{h/N}^V).
\]
Recalling the definition \(T_N(\cos \alpha) = \cos(N\alpha)\) of the Chebyshev polynomial \(T_N\), we observe
that for \eqref{eq:opt}
\[A_h(z)=
T_N\big(1-\frac{z}{2N^2}\big).
\]
 Well-known properties of \(T_N\), imply that no other  polynomial \(A_h(z)\) of degree \(\leq N\) with
\(A_h = 1-z/2+\mathcal{O}(h^2)\) has modulus \(\leq 1\) in the interval \(-1 < 1-z/(2N^2)<1\), \ie\ when
 \(0<h<2N\).
\vspace{2mm}

\section{Monte Carlo methods}
\label{sec:montecarlomethods}

In this section we review some basic concepts and principles of Monte Carlo methods aimed at computing
integrals with respect to a given probability distribution \cite{So1997,AsGl2007,Li2008,Di2009}. We also
describe
 the
Hamiltonian or Hybrid Monte Carlo (HMC) method and some of its variants.

\subsection{Simple Monte
Carlo Methods}

Given a probability distribution $\mu$ in $\mathbb{R}^D$ (the \textit{target distribution}) and a function $F:
\mathbb{R}^D \to \mathbb{R}$, the  problem addressed by the algorithms  considered here is to numerically
estimate the following $D$-dimensional integral with respect to $\mu$,
\begin{equation}\label{eq:mu_of_F}
\mu(F) = \int_{\mathbb{R}^D} F(x) \mu(d x) \;.
\end{equation}
In general,  $\mu(F)$ cannot be determined analytically.   Moreover, since the dimension $D$ might not be
small, conventional numerical quadrature is likely not to be practical or even feasible.

The simple Monte Carlo method approximately computes $\mu(F)$ in \eqref{eq:mu_of_F} by generating $N$
independent and identically distributed (i.i.d.) samples $X_1, \dots, X_N$ from $\mu$, evaluating the function
$F$ at these samples and using the \textit{estimator}
\beq\label{eq:estimator}
 \bar F_N = \frac{1}{N} \sum_{i=1}^N F(X_i) \;.
\eeq
Assuming that $\mu(F) < \infty$, the \emph{law of large numbers} states that
\[
\lim_{N \to \infty} \bar F_N = \mu(F) \quad \text{as $N \to \infty$},
\]
almost surely. If, in addition, the standard deviation \(\sigma_0(F)\) of the random variable \(F(X)\),
\(X\sim \mu\), defined by
\begin{equation}\label{eq:sigma0_of_F}
\sigma_0(F)^2 = \int_{\mathbb{R}^D} (F(x) - \mu(F))^2 \mu(dx) \;
\end{equation}
is finite, the \emph{central limit theorem} ensures the following distributional limit
 \[
N^{1/2} ( \bar F_N - \mu(F) ) \overset{d}{\to} \mathcal{N}(0, \sigma_0(F)^2 )
 \quad \text{as} \quad N \to \infty.
\]
 Loosely speaking, this may be interpreted as stating
that the distribution of $\bar F_N$ is approximately \(\mathcal{N}(0, \sigma_0(F)^2/N) \).  Hence, the
standard deviation of the error $ \bar F_N - \mu(F)$ decreases like the inverse square root of the number of
samples. Often this standard deviation is referred to as \emph{Monte Carlo error}. Thus, to halve the Monte
Carlo error the number of i.i.d. samples needs to be quadrupled.

In most cases of practical interest, one cannot directly generate i.i.d. samples \(X_i\) from $\mu$ and
resorts to Markov Chain Monte Carlo methods.

\subsection{Markov Chain Monte Carlo Methods}
\label{ss:markov}
 Recall that a \textit{Markov chain} with state space $\mathbb{R}^D$ is a sequence of random
$D$-vectors
 $\{ X_i \}_{i \in \mathbb{N}}$ that satisfies the Markov property
 \[
\mathbb{P}( X_{i+1} \in A \mid X_1, \dots, X_i ) = \mathbb{P}( X_{i+1} \in A \mid X_i)
\]
for all measurable sets $A$.   In other words, given the past history of the chain \(X_1,\dots, X_i\), the
only information required to update the state of the chain is the current state \(X_i\). Here our interest is
restricted to (time--)homogenuous chains, \ie\ to cases where \(\mathbb{P}( X_{i+1} \in A \mid X_i)\) is
independent of \(i\).

Typically one constructs a homogeneous Markov chain in terms of its \textit{transition probabilities} $\Pi_x$,
$x \in \mathbb{R}^D$. These are the probabilities
\[ \Pi_{x}(A)
= \mathbb{P}(X_{i+1} \in A \mid X_i = x) \;,
\]
with $i \in \mathbb{N}$ and  $A$ measurable. Often \(\Pi_x\) may be computed as
\[
\Pi_x(A) = \int_A \Pi_x(dx^\prime),
\]
for a suitable kernel \(\Pi_x(dx')\).
 Clearly the chain is determined once the transition probabilities
and the distribution of \(X_1\) are known. In practice, the term chain is used in a wide sense to refer to the
transition probabilities without specifying the distribution of \(X_1\).

A probability distribution $\nu$ is an \textit{invariant or stationary distribution} of a Markov chain with
transition probabilities $\Pi_x$ if \[ \int_{\mathbb{R}^D} \Pi_{x}(A) \nu(dx) = \int_{A} \nu(dx)
\]
holds for all measurable sets $A$. We also say that $\Pi_x$ preserves \(\nu\). In the situations we are
interested in, a Markov chain will have a unique invariant distribution. If \(\nu\) is the invariant
distribution of the chain and
 in addition \(X_1\sim\nu\), one says that the chain is \emph{at
stationarity}.

Markov Chain Monte Carlo (MCMC) methods generate a Markov chain $\{ X_i \}_{i \in \mathbb{N}}$ that has the
target $\mu$ as an invariant distribution and estimate \(\mu(F)\) by the average \eqref{eq:estimator}. By
analogy to the simple i.i.d.~situation described above, one would like to have MCMC methods that meet two
basic requirements.
\begin{itemize}
\item For each $F: \mathbb{R}^D \to \mathbb{R}$ such that $\mu(F) < \infty$,
\begin{equation}
    \label{eq:mcmc_lln} \bar F_N \overset{a.s.}{\to} \mu(F) \quad \text{as $N \to \infty$} \;.
\end{equation} This is the MCMC analog of the law of large numbers.

\medskip

\item For each $F: \mathbb{R}^D \to \mathbb{R}$ such that $\mu(F) < \infty$ and $\sigma_0(F)<\infty$,
    \begin{equation} \label{eq:mcmc_clt} N^{1/2} ( \bar F_N - \mu(F) ) \overset{d}{\to} \mathcal{N}(0,
    \sigma(F)^2 ) \quad \text{as $N \to \infty$},
\end{equation}
for some \(\sigma(F)\).
 This is the MCMC analog of the central limit theorem.
\end{itemize}

 For each fixed function $F$ and Markov chain $\{ X_i \}_{i \in \mathbb{N}}$,
 the constant $\sigma(F)^2$ appearing in \eqref{eq:mcmc_clt} is called the
\textit{asymptotic variance} of the MCMC estimator $\bar{F}_N$.
 A straightforward calculation shows that this asymptotic variance satisfies
 \[
 \sigma(F)^2 = \sigma_0(F)^2
 + 2 \sum_{i > 1}
 \cov_{\mu}( F(X_1), F(X_i) )
 \]
 where $\sigma_0(F)$ is defined in \eqref{eq:sigma0_of_F} and the covariances are computed assuming
that the chain is at stationarity. If the \(X_i\) were independent, all the covariances \(\cov_{\mu}( F(X_1),
F(X_i))\) would vanish and we would recover the standard central limit theorem.
 Since in most interesting cases the $X_i$'s in the Markov chain
 are not mutually independent, often $\sigma(F)$ is larger than $\sigma_0(F)$.
Generally speaking it is desirable to have low values of \(\cov_{\mu}( X_1, X_i)\) so that \(\sigma(F)\) is
not far away from $\sigma_0(F)$ for each \(F\).

In practice, the inputs that the user has to supply to an MCMC algorithm include, at least:
\begin{itemize}
\item A sample of the initial state \(X_1\). Ideally,  this sample should be taken in a domain of state
    space of high probability. Otherwise the chain may need many steps to start generating useful samples. A
    discussion of this issue is out of the scope of this paper.
\item A, not necessarily normalized, density function \(\rho\) of the target \(\mu\) (\ie\ the probability
    density function is \(Z^{-1}\rho(x)\), where \(Z = \int_{\Reals^D}\rho(x)dx\) is not assumed to be
    \(1\);   the value of \(Z\) is not required to run the algorithms).
\end{itemize}

\subsection{Metropolis Method for Reversible Maps}

The replacement of the i.i.d.\ variables that simple Monte Carlo uses in the estimator \eqref{eq:estimator}
with variables of a Markov chain is of interest because it is not difficult to construct a chain that has a
given target \(\mu\) as an invariant distribution. The key of this construction is the Metropolis-Hastings
accept/reject mechanism, that turns a given \emph{proposal} chain (for which \(\mu\) is not invariant) into a
\emph{Metropolized chain}, which leaves \(\mu\) invariant.  The simplest Metropolis rule was introduced in
1953 in a landmark paper \cite{MeRoRoTeTe1953};  later Hastings provided an important generalization
    \cite{Ha1970}.

A review of the Metropolis-Hastings rule is not required for our purposes here. However we shall
 present a trimmed down variant of Metropolis-Hastings that we will use to define HMC.
This variant works in the special case where the target \(\mu\) is invariant with respect to a linear
involution, takes as an input a reversible deterministic map and manufactures a Markov chain that preserves
\(\mu\).
 While the chain that we construct is not expected to satisfy a law of
large numbers and therefore has no practical merit, Proposition~\ref{prop:metropolized_reversible_map} will be
used later to analyse HMC methods.

 The technique, patterned after \cite{FaSaSk2014}, requires a non-normalized density function
$\rho: \mathbb{R}^D \to \mathbb{R}$ of the target, and, as pointed out above, a map $\Phi: \mathbb{R}^D \to
\mathbb{R}^D$ that is reversible with respect to a linear involution $S$ that preserves probability, \ie\
\(\rho\circ S = \rho\). The accept/reject mechanism is based on the \emph{acceptance probability} $\alpha:
\mathbb{R}^D \to [0,1]$ defined as
\begin{equation}  \label{eq:acceptance_probability}
\alpha(x) = \min\left\{ 1, \frac{\rho(\Phi(x))}{\rho(x)}  \left| \det \Phi'(x) \right| \right\}.
\end{equation}
(\(\Phi^\prime\) is the Jacobian matrix of \(\Phi\).)

We  consider the following algorithm:
\begin{algorithm}[Metropolized Reversible Map] \label{algo:metropolized_reversible_map}
Given $X_0 \in \mathbb{R}^D$ (the input state),  the method outputs a state $X_1$ as follows.
\begin{description}
\item[Step 1] Generate a proposal move $\tilde X_1 = \Phi(X_0)$.
\item[Step 2] Output $X_1 = \gamma \tilde X_1 + (1-\gamma) S(X_0)$ where $\gamma$ is a Bernoulli random
    variable with parameter $\alpha(X_0)$ (\ie\ \(\gamma\) is \(1\) with probabilty $\alpha(X_0)$ and \(0\)
    with probability $1-\alpha(X_0)$).
\end{description}
\end{algorithm}
\medskip
Step 2 contains the accept/reject mechanism. In case of \emph{acceptance} the updated state coincides with the
state proposed $\tilde X_1$ from Step 1; in case of \emph{rejection} the updated state is $S(X_0)$.
 Note that, in case of rejection,   conventional Metropolis mechanisms set the updated state of the chain
 to be $X_0$.

\begin{proposition} \label{prop:metropolized_reversible_map}
In the situation described above,
 let $\{X_i
\}_{i \in \mathbb{N}}$ be the Markov chain  defined by iterating
Algorithm~\ref{algo:metropolized_reversible_map}. Then the target distribution is an
 invariant distribution of
this chain.
\end{proposition}

\begin{proof}
The  transition kernel of $\{ X_i \}_{i \in \mathbb{N}}$ is given by
\[
\Pi_x(d x') = \alpha(x) \delta (x' - \Phi(x)) dx' +  (1-\alpha(x)) \delta (x' - S(x)) dx'
\]
where $\delta(\cdot)$ is the  Dirac-delta function.   Hence, for any measurable set $A$, if  $1_A(x)$ denotes
the corresponding indicator function,
 \begin{align*}
\int_{\mathbb{R}^D} & \Pi_x(A) \rho(x) dx = \int_A \rho(x) dx  \\
& + \int_{\mathbb{R}^D} \alpha(x)  \rho(x) 1_{A} (\Phi(x)) dx
 - \int_{\mathbb{R}^D} \alpha(x)  \rho(x) 1_{A}
(S(x)) dx \;.
\end{align*}

 By change of variables in the third integral,
\begin{align*}
\int_{\mathbb{R}^D} & \Pi_x(A) \rho(x) dx = \int_A \rho(x) dx \\
& + \int_{\mathbb{R}^D} \alpha(x)  \rho(x) 1_{A} (\Phi(x)) dx \\
& - \int_{\mathbb{R}^D} \alpha\big(S(\Phi(x))\big)  \left| \det \Phi'(x) \right| \rho(\Phi(x)) 1_{A}
(\Phi(x)) dx,
\end{align*}
where we used the hypothesis that $\rho \circ S = \rho$. The last two terms on the right-hand side of this
equation cancel because
\begin{align*}
\alpha\big(S(\Phi(x))\big) &= \min\left\{ 1, \frac{\rho(x)}{\rho(\Phi(x)) }   \left| \det \Phi'(x) \right|^{-1} \right\} \\
&= \alpha(x) \rho(x)  \frac{1}{\rho(\Phi(x)) }   \left| \det \Phi'(x) \right|^{-1},
\end{align*}
which follows from the reversibility of \(\Phi\), the hypothesis $\rho \circ S = \rho$ and
Proposition~\ref{prop:jacob}.
\end{proof}

\subsection{The HMC method: basic idea}
We consider a target distribution \(\Pi\) in \(\Reals^d\). If $\mathcal{U}: \mathbb{R}^d \to \mathbb{R}$
denotes the negative logarithm of the (not necessarily normalized) probability density function of the target,
then
\[ \Pi(dq) = Z_q^{-1} \exp(-\mathcal{U}(q)) dq \;,
\quad \text{where} \quad Z_q = \int_{\mathbb{R}^{d}} \exp(-\mathcal{U}(q))\, dq \;.
\]
The Monte Carlo algorithms studied here use \(\mathcal{U}\) but do not require the knowledge of the
normalization factor \(Z_q\). In HMC, regardless of the application in mind, \(\mathcal{U}\) is seen as the
potential energy of a mechanical system with coordinates \(q\). Then auxiliary momenta \(p\in\Reals^d\) and a
quadratic kinetic energy  function \(\mathcal{T}(p)=(1/2)p^TM^{-1}p\) are introduced as in \eqref{eq:sepham}
(\(M\) is a positive-definite, symmetric matrix chosen by the user).\footnote{ Often \(M\) is just taken to be
the unit matrix; however \(M\) may be advantageously chosen to precondition the dynamics, see
Remark~\ref{rem:massmatrix}. }
The total energy of this fictitious mechanical system
is \(H=\mathcal{T}+\mathcal{U}\) and the equations of motion are given in \eqref{eq:newton2}.

\begin{example}\label{ex:gaussianHMC}
In the particular case where the target \(\mu\) is Gaussian with un-normalized density \( \exp (-(1/2)q^TKq)\),
\(q\in\Reals^d\), we have \(\mathcal{U}(q) = (1/2)q^TKq\), which leads to the Hamiltonian \eqref{eq:MKmodel}
we discussed before. For a univariate standard normal target, \(H = (1/2)(p^2+q^2)\) and Hamilton's equations
reduce to those of harmonic oscillator \eqref{eq:harmonic}.
\end{example}

The \emph{Boltzmann-Gibbs} distribution in \(\Reals^{2d}\) corresponding to \(H\) was discussed earlier in
connection with Theorem~\ref{th:bg}. This distribution is defined as (for simplicity the inverse temperature
is taken here to be \(\beta=1\)):
\begin{align} \label{eq:boltzmann_gibbs}
\nonumber
\Pi_{BG}(dq,dp) &=
(2 \pi)^{-\frac{1}{2} d} \left| \det{M} \right|^{-\frac{1}{2}}\exp\left(-\frac{1}{2} p^TM^{-1}p\right)\\
&\qquad\qquad\qquad\qquad\qquad \times
 Z_q^{-1} \exp(-\mathcal{U}(q))dq\,dp.
\end{align}
Clearly the target \(\Pi\) is the \(q\)--marginal of \(\Pi_{BG}\). The \(p\)--marginal is Gaussian with zero
mean and covariance matrix \(M\); therefore samples from this marginal are easily available (and will be put
to use in the algorithms below). A key fact for our purposes:  Hamilton's equations of motion
\eqref{eq:newton2} preserve \(\Pi_{BG}\) (Theorem~\ref{th:bg}).

HMC  generates (correlated) samples \((q_i,p_i)\in\Reals^{2d}\)  by means of a Markov chain that leaves
\(\Pi_{BG}\) invariant; the corresponding marginal \(q_i\in\Reals^d\) chain then leaves invariant the target
distribution \(\Pi\). The basic idea of HMC is encapsulated in the following algorithm (the duration
$\lambda>0$ is a
---deterministic--- parameter, whose value is specified by the user).
\begin{algorithm}[Exact HMC] \label{algo:exact_hmc}
Let  $\lambda>0$ denote the duration parameter.

Given  the current state of the chain $(q_0, p_0)\in\mathbb{R}^{2d}$, the method outputs a state
$(q_1,p_1)\in\mathbb{R}^{2d}$ as follows.
\begin{description}
\item[Step 1] Generate a $d$-dimensional random vector  $\xi_0 \sim \mathcal{N}(0,M)$.
\item[Step 2] Evolve over the time interval $[0,\lambda]$ Hamilton's equations \eqref{eq:newton2} with
    initial condition $(q(0), p(0)) = (q_0,\xi_0)$.
\item[Step 3] Output $(q_1, p_1) = (q(\lambda), p(\lambda))$.
\end{description}
\end{algorithm}

Note that \(p_0\) plays no role, since the initial condition starts from \(\xi_0\). Step 1 is referred to as
\emph{momentum refreshment} or \emph{momentum randomization}.

It is easy to see that this algorithm succeeds in preserving the distribution \(\Pi_{BG}\):

\begin{theorem} \label{thm:exact_rhmc}
Consider the Markov chain $\{(q_i,p_i) \}_{i \in \mathbb{N}}$ defined by iterating
Algorithm~\ref{algo:exact_hmc}. The probability distribution $\PiBG$ in \eqref{eq:boltzmann_gibbs} is an
invariant distribution of this chain.
\end{theorem}
\begin{proof}The transformation \((q_0,p_0)\mapsto (q_0,\xi_0)\) obviously preserves the Boltzmann-Gibbs
distribution. The same is true for the transformation \((q_0,\xi_0)\mapsto(q_1,p_1)\) as we saw in
Theorem~\ref{th:bg}.
\end{proof}

The transition kernel of this chain is given by
\[
\Pi_{(q,p)}(dq', dp') = \mathbb{E} \left\{  \delta( (q',p') - \varphi_{\lambda}(q,\xi) ) \right\}  dq' dp',
\]
where the expected value is over $\xi \sim \mathcal{N}(0,M)$.

The most appealing feature of the algorithm is that, if \(\lambda\) is sufficiently large, we may hope that
the Markov transitions \(i\rightarrow i+1\) produce values \(q_{i+1}\) far away from \(q_i\), thus reducing
the correlations in the chain and facilitating the exploration of the target distribution.

\subsection{Numerical HMC}
Algorithm~\ref{algo:exact_hmc} cannot be used in  practice because in the cases of interest the exact solution
flow \(\varphi_\lambda\) of Hamilton's equations is not available. It is then necessary to resort to numerical
approximations to \(\varphi_\lambda\), but, as pointed out in Section~\ref{ss:geometricintegration}, numerical
methods cannot preserve volume in phase space \emph{and} energy and therefore do not preserve exactly the
Boltzman-Gibbs distribution. To correct the bias introduced by the time discretization error, the numerical
solution is Metropolized using Algorithm~\ref{algo:metropolized_reversible_map}. However, this requires  that
the numerical integrator be \emph{reversible}.

Let $\Psi_{\lambda}$ denote a numerical approximation of $\varphi_{\lambda}$ (more precisely, if the step size
is \(h\) and \(n=\lfloor \lambda/h \rfloor\) is the number of steps required to integrate up to \(t=\lambda\),
then \(\Psi_\lambda =\psi_h^n\)). In order to use Algorithm~\ref{algo:metropolized_reversible_map} with
\(\Psi_\lambda\) playing the role of \(\Phi\) and the momentum flip involution \eqref{eq:mflip} playing the
role of \(S\),  we first note (Proposition~\ref{prop:mflip}) that the
 momentum flip involution preserves \(H\) in \eqref{eq:sepham} and, as a consequence, it preserves the Boltzmann-Gibbs
distribution. In addition Theorem~\ref{th:reverflow} ensures that the Hamiltonian flow is reversible with
respect to this involution;  it then makes sense (Theorem~\ref{th:revintegrator}) to assume that the
integrator chosen is such that  $\Psi_{\lambda}$ is also reversible. The acceptance probability in
\eqref{eq:acceptance_probability} now reads
\begin{equation}\label{eq:alphadelta}
\alpha(q,p) = \min\left\{ 1, e^{-\Delta H(q,p)} \left| \det \Psi_{\lambda}'(q,p) \right|
\right\}  \;,
\end{equation}
where
\[
\Delta H(q,p) = H(\Psi_{\lambda}(q,p)) - H(q,p)
\]
is the energy error (recall that if the integrator were exact \(H(\Psi_{\lambda}(q,p))\) would coincide with
\(H(q,p)\) by conservation of energy, Theorem~\ref{th:consenergy}).

\begin{algorithm}[Numerical HMC] \label{algo:numerical_hmc}
 Denote by
\(\lambda>0\) the duration parameter and  let \(\Psi_\lambda\) be a reversible numerical approximation to the
Hamiltonian flow \(\varphi_\lambda\).

Given  the current state of the chain $(q_0, p_0)\in\mathbb{R}^{2d}$; the method outputs a state
$(q_1,p_1)\in\mathbb{R}^{2d}$ as follows.
\begin{description}
\item[Step 1] Generate a $d$-dimensional random vector  $\xi_0 \sim \mathcal{N}(0,M)$.
\item[Step 2] Find \(\Psi_{\lambda}(q_0,\xi_0)\) by  evolving Hamilton's equations \eqref{eq:newton2} with a
    reversible integrator over the time interval $[0,\lambda]$  with initial condition $(q_0,\xi_0)$.
\item[Step 3] Output $(q_1, p_1) = \gamma \Psi_{\lambda}(q_0,\xi_0) + (1-\gamma)  (q_0, - \xi_0) $ where
    $\gamma$ is a Bernoulli random variable with parameter $\alpha(q_0, \xi_0)$ with \(\alpha\) as in
    \eqref{eq:alphadelta}.
\end{description}
\end{algorithm}

\begin{theorem}\label{thm:numerical_hmc}
Consider the Markov chain $\{(q_i,p_i) \}_{i \in \mathbb{N}}$ defined by iterating
Algorithm~\ref{algo:numerical_hmc}. The probability distribution $\PiBG$ in \eqref{eq:boltzmann_gibbs} is an
invariant distribution of this chain.
\end{theorem}
\begin{proof} As in the preceding theorem, the transformation \((q_0,p_0)\mapsto (q_0,\xi_0)\)  preserves the Boltzmann-Gibbs
distribution. The same is true for the transformation \((q_0,\xi_0)\mapsto(q_1,p_1)\) according to
Proposition~\ref{prop:metropolized_reversible_map}.
\end{proof}

The transition kernel of the chain is given by
\begin{align*}
\Pi_{(q,p)}(dq', dp') &=  \mathbb{E} \left\{ \alpha(q,\xi) \delta( (q',p') - \Psi_{\lambda}(q,\xi) )  \right\} dq' dp'  \\
& \qquad\qquad + \mathbb{E} \left\{  (1-\alpha(q,\xi))  \delta( (q',p') - S(q,\xi) ) \right\} dq' dp'.
\end{align*}

In practice, the acceptance probability \eqref{eq:alphadelta} may not be readily available due to the need to
compute \(\det \Psi_{\lambda}'(q,p)\). If the numerical approximation \(\Psi_\lambda\) in addition to being
assumed reversible is also \emph{volume preserving} (as it would be for splitting integrators according to
Theorem~\ref{th:sympintegrator}), then the determinant drops from the formula, and then the acceptance
probability
\begin{equation}\label{eq:alphadelta2}
\alpha(q,p) = \min\left\{ 1, e^{-\Delta H(q,p)}
\right\}  \;,
\end{equation}
becomes easily computable. Variants where preservation of volume does not take place are studied by
\cite{FaSaSk2014}.

\begin{remark}The states of the Markov chain $\{(q_i,p_i) \}_{i \in \mathbb{N}}$ are not to be confused with
the intermediate values of \(q\) and \(p\) that the numerical integrator generates while transitioning the
chain from one state of the chain to the next. Those intermediate values were denoted by \((q_n,p_n)\) in the
preceding sections and we have preferred not to introduce additional notation to describe the Markov chain.
\end{remark}

\begin{remark}Theorems~\ref{thm:exact_rhmc} and \ref{thm:numerical_hmc} show that \(\Pi_{\rm BG}\) is an
invariant distribution for the chains generated by Algorithms~\ref{algo:exact_hmc} and
\ref{algo:numerical_hmc} respectively. However they do not guarantee that those chains meet the two basic
requirements in \eqref{eq:mcmc_lln} and \eqref{eq:mcmc_clt} and indeed a simple example will be presented
below where the sequence of values of \(q\) generated by those algorithms is \(q_0\), \(-q_0\), \(q_0\),
\(-q_0\), \dots so that the requirements are \emph{not} met. A detailed study of the convergence properties of
HMC is outside the scope of this paper and we limit ourselves to some remarks in
Section~\ref{sec:convergence}.
\end{remark}

\subsection{Exact Randomized HMC}

The Hamiltonian flow in Step 2 of Algorithm~\ref{algo:exact_hmc} is what, in principle, enables HMC to make
large moves in state space that reduce correlations in the Markov chain $\{ q_i \}_{i \in \mathbb{N}}$.
Roughly speaking, one may hope that,  by increasing the duration $\lambda$, $q_1$ moves away from $q_0$, thus
reducing  correlation.  However, simple examples show that this outcome is far from assured.

Indeed, for the univariate standard normal target distribution in Example~\ref{ex:gaussianHMC},
 the Hamiltonian flow is a rotation in the $(q,p)$-plane with period $2\pi$. It is easy to see  that, if $q_0$ is taken from the target distribution, as $\lambda$ increases from $0$ to $\pi/2$, the correlation between $q_1$ and $q_0$ decreases and for $\lambda = \pi/2$, $q_1$ and $q_0$ are independent. However increasing $\lambda$ beyond $\pi/2$ will
 cause an increase in the correlation and for $\lambda = \pi$, $q_1=-q_0$ and the chain is not ergodic.  For general distributions, it is likely that a small $\lambda$ will lead to a highly correlated chain, while choosing $\lambda$ too large  may cause the Hamiltonian trajectory to make a U-turn and fold back on itself, thus increasing correlation \cite{HoGe2014}. Generally speaking the performance of HMC may be very sensitive to changes in \(\lambda\) as first noted by \cite{Ma1989}. In order to increase the robustness of the algorithm,
Mackenzie suggested to vary randomly \(\lambda\) from one Markov transition to the next and for that purpose
he used a uniform distribution in an interval \( [\lambda_{\rm min},\lambda_{\rm max}]\).

Recently \cite{BoSa2016} have studied an algorithm where
the lengths of the time intervals of integration of the Hamiltonian dynamics at the different transitions of the Markov chain are independent and identically distributed exponential random variables with mean $\lambda$; these durations are of course taken to be independent of the state of the chain. The algorithm is then as follows:

\begin{algorithm}[Exact RHMC] \label{algo:exact_rhmc}
Given the current state of the chain $(q_0, p_0) \in \mathbb{R}^{2d}$, the algorithm outputs the state $(q_1, p_1)\in \mathbb{R}^{2d}$ as follows.
\begin{description}
\item[Step 1] Generate a $d$-dimensional random vector $\xi_0 \sim \mathcal{N}(0,M)$.
\item[Step 2] Generate a random duration $t \sim \Exp(1/\lambda)$.
\item[Step 3] Evolve over the time interval $[0, t]$ Hamilton's equations \eqref{eq:newton} with initial
    condition $(q(0), p(0))= (q_0, \xi_0)$.
\item[Step 4] Output $(q_1, p_1) = (q(t), p(t))$.
\end{description}
\end{algorithm}

Analogous to Theorem 5.1, the probability distribution $\PiBG$ in
 \eqref{eq:boltzmann_gibbs} is an invariant distribution of the
  Markov chain $\{ (q_i, p_i) \}_{i \in \mathbb{N}}$ defined by iterating Algorithm~\ref{algo:exact_rhmc}.
   Analytical results and numerical experiments \cite[\S 4-5]{BoSa2016} show that
the dependence of the performance of the RHMC Algorithm~\ref{algo:exact_rhmc}
on the mean duration parameter \(\lambda\) is simpler than the dependence of
the performance of Algorithm~\ref{algo:exact_hmc}  on its constant duration parameter.

\subsection{Numerical Randomized HMC}

Unfortunately, the  complex dependence of correlation on the duration parameter $\lambda$ of Algorithm~\ref{algo:exact_hmc} is not removed by time discretization and is therefore inherited by Algorithm~\ref{algo:numerical_hmc}.
For instance, for the univariate standard normal target, it is easy to check that if $\lambda$ is close to an integer multiple of $\pi$ and $h>0$ is suitably chosen, then a Verlet numerical integration will result, for each $q_0$, in $q_1 = -q_0$ (a move that will be accepted by the Metropolis-Hasting step).

To improve its performance, Algorithm~\ref{algo:numerical_hmc} is typically operated with values of $h$ that
are randomized \cite{Ne2011}.
Since, due to stability restrictions, explicit integrators cannot be used with
arbitrarily large
 values of the time step, \(h\) is typically chosen  from a uniform distribution in an (often narrow)
  interval $(\Delta t_{\rm min}, \Delta t_{\rm max})$. The number of time steps \(N\) in each integration
  leg is kept constant
  and therefore the length of the integration intervals is random with a uniform distribution in
  $(N\Delta t_{\rm min}, N\Delta t_{\rm max})$.
   Even after such a randomization, the fact remains that increasing
  the duration parameter will increase the computational cost of each integration leg and may impair the quality of the sampling.

Several modifications of Algorithm~\ref{algo:exact_rhmc} that use numerical integration are suggested in
\cite{BoSa2016}. The most obvious of them approximates the Hamiltonian flow in Step 3 by a volume-preserving,
reversible integrator (such as Verlet) operated with a fixed step size $h$ and with the number of integration
steps $m$ at the different transitions of the Markov chain being independent and identically distributed
geometric random variables with mean $\lambda /h$.  These random numbers are of course taken to be independent
of the state of the chain.  As in Algorithm~\ref{algo:numerical_hmc}, one needs an accept-reject mechanism to
remove the bias due to the energy errors introduced by this integrator.

\begin{algorithm}[Numerical RHMC] \label{algo:numerical_rhmc}
Denote by $\lambda>0$ the duration parameter and let $\psi_{h}$ be a reversible numerical approximation of the Hamiltonian flow $\varphi_h$.

Given the current state of the chain $(q_0, p_0) \in \mathbb{R}^{2d}$, the algorithm outputs the state $(q_1, p_1)\in \mathbb{R}^{2d}$ as follows.
\begin{description}
\item[Step 1] Generate a $d$-dimensional random vector $\xi_0 \sim \mathcal{N}(0,M)$.
\item[Step 2] Generate a geometric random variable $m$ supported on the set $\{1, 2, 3, ... \}$ and with mean $\lambda/h$.
\item[Step 3] Output $(q_1, p_1) = \gamma \psi_h^m(q_0, \xi_0) + (1-\gamma) (q_0,-\xi_0)$ where $\gamma$ is
    a Bernoulli random variable with parameter $\alpha$ defined as in \eqref{eq:alphadelta2}.
\end{description}
\end{algorithm}

Analogous to Theorem 5.2, the probability distribution $\PiBG$ in \eqref{eq:boltzmann_gibbs} is an invariant distribution of the Markov chain
$\{ (q_i, p_i) \}_{i \in \mathbb{N}}$ defined by iterating Algorithm~\ref{algo:numerical_rhmc}.

In the remainder of this paper the attention is focused on Algorithm~\ref{algo:numerical_hmc}. Experiments
based on Algorithm~\ref{algo:numerical_rhmc} are reported in Section~\ref{sec:preconditioned} and an example
that illustrates the effects of randomization of \(h\) and \(\lambda\) is presented in
Section~\ref{sec:randomizing}. \vspace{2mm}

\section{Numerical integration and HMC}
\label{secc:numintHMC}

The computational work in HMC mainly stems from the cost of the evaluations of the  force $-\nabla
\mathcal{U}$ that are required in the numerical integration  to be carried out at each transition of the
Markov chain (Step 2 of Algorithm~\ref{algo:numerical_hmc}). If the dimension \(d\) of \(q\) is high, those
evaluations are likely to be expensive; for instance, in a molecular dynamics study of a macromolecule with
\(N\gg 1\) atoms, \(d=3N\) and each atom typically interacts with all others, so that the complexity of
evaluating $-\nabla \mathcal{U}$ grows like \(N^2\). As in any other numerical integration, the aim is to
reach a target accuracy with the minimum possible complexity. Specific to the HMC scenario is the fact that,
in the event of rejection in Step~3, the algorithm wastes all of the force evaluations used to compute
\(\Psi_\lambda(q,\xi)\). In addition, when a rejection occurs at a transition \(i\rightarrow i+1\) of the
Markov chain, the new value \(q_{i+1}\)  coincides with the old \(q_i\), and this contributes to an increase
in the correlations along the chain, which degrades the quality of sampling as pointed out when discussing the
central limit theorem in \eqref{eq:mcmc_clt}. Since low acceptance rates are unwelcome and the acceptance
probability \eqref{eq:alphadelta2} is a function of the \emph{energy} error \(\Delta H(q,p) =
H(\Psi_{\lambda}(q,p)) - H(q,p)\), it is important to perform the integration so as to have \emph{small energy
errors}. In this connection we recall from Section~\ref{ss:geometricintegration} that symplectic integrators
conserve energy with errors that are \(\mathcal{O}(h^\nu)\)  over exponentially long time intervals. In this
way, the symplecticness of the integrator plays a dual role in HMC. On the one hand, it ensures conservation
of volume, thereby making it possible to have the simple expression \eqref{eq:alphadelta2} for the acceptance
probability. On the other hand, it ensures favourable energy errors even if the integration legs are very
long.

An additional point: the \emph{sign} of the error matters. Formula \eqref{eq:alphadelta2} shows that, if the
integration starts from a point \((q,p)\) for which \(\Delta<0\), then \(\alpha =1\) leading to acceptance.

This section begins with a key result, Theorem~\ref{thm:mean_energy_error}, that shows that for reversible,
volume preserving integrators the energy error is on average much smaller than one may have anticipated. After
that we study in detail the model case of Gaussian targets and discuss the construction of integrators more
efficient than Verlet.

\subsection{Mean energy error}

Step 2 in Algorithm~\ref{algo:numerical_hmc} requires a volume-preserving, reversible integrator if
   \eqref{eq:alphadelta2} is to be used. As we shall discuss now, those geometric
properties have a direct impact on the mean energy error. We begin with an auxiliary result \cite[Lemma
3.3]{BePiRoSaSt2013} that holds for any volume-preserving, reversible map.
\begin{proposition}\label{prop:natesh}
Let $\Psi: \mathbb{R}^{2d} \to \mathbb{R}^{2d}$ be a bijection that is volume-preserving and reversible with
respect
 to the momentum flip involution \eqref{eq:mflip} and set $\Delta(q,p) = H(\Psi(q,p)) - H(q,p)$. If \(g\)
 is an odd real function of a real variable, then
 \[
 \int_{\mathbb{R}^{2d}} g(\Delta(q,p))\, e^{-H(q,p)} dq
dp =- \int_{\mathbb{R}^{2d}} g(\Delta(q,p))\, e^{- H(\Psi(q,p))} dq dp,
 \]
 provided that one of the integrals exists. If \(g\) is even, then
 \[
 \int_{\mathbb{R}^{2d}} g(\Delta(q,p))\, e^{-H(q,p)} dqdp
 = \int_{\mathbb{R}^{2d}} g(\Delta(q,p))\, e^{- H(\Psi(q,p))} dq dp,
 \]
  provided that one of the integrals exists.
\end{proposition}
\begin{proof}
By a change of variables under the map $(q,p) \mapsto S(\Psi(q,p))$,
\begin{eqnarray*}
\int_{\mathbb{R}^{2d}} g(\Delta(q,p)) e^{-H(q,p)} dqdp &=&
\int_{\mathbb{R}^{2d}} g(\Delta(S(\Psi(q,p))))\, e^{-H(S(\Psi(q,p)))} dqdp\\
&=&\int_{\mathbb{R}^{2d}} g(-\Delta(q,p))\,e^{- H(\Psi(q,p))} dq dp,
\end{eqnarray*}
where in the first step we used that both \(S\) and \(\Psi\) preserve volume and in the second we took into
account that \(S\) leaves \(H\) invariant and that the reversibility of \(\Psi\) implies
\begin{equation}\label{eq:nicefigure}
\Delta(q,p) =- \Delta(S(\Psi(q,p))).
\end{equation}
\end{proof}

\begin{figure}[t]
%\vspace{-10cm}
\begin{center}\includegraphics[scale=0.45]{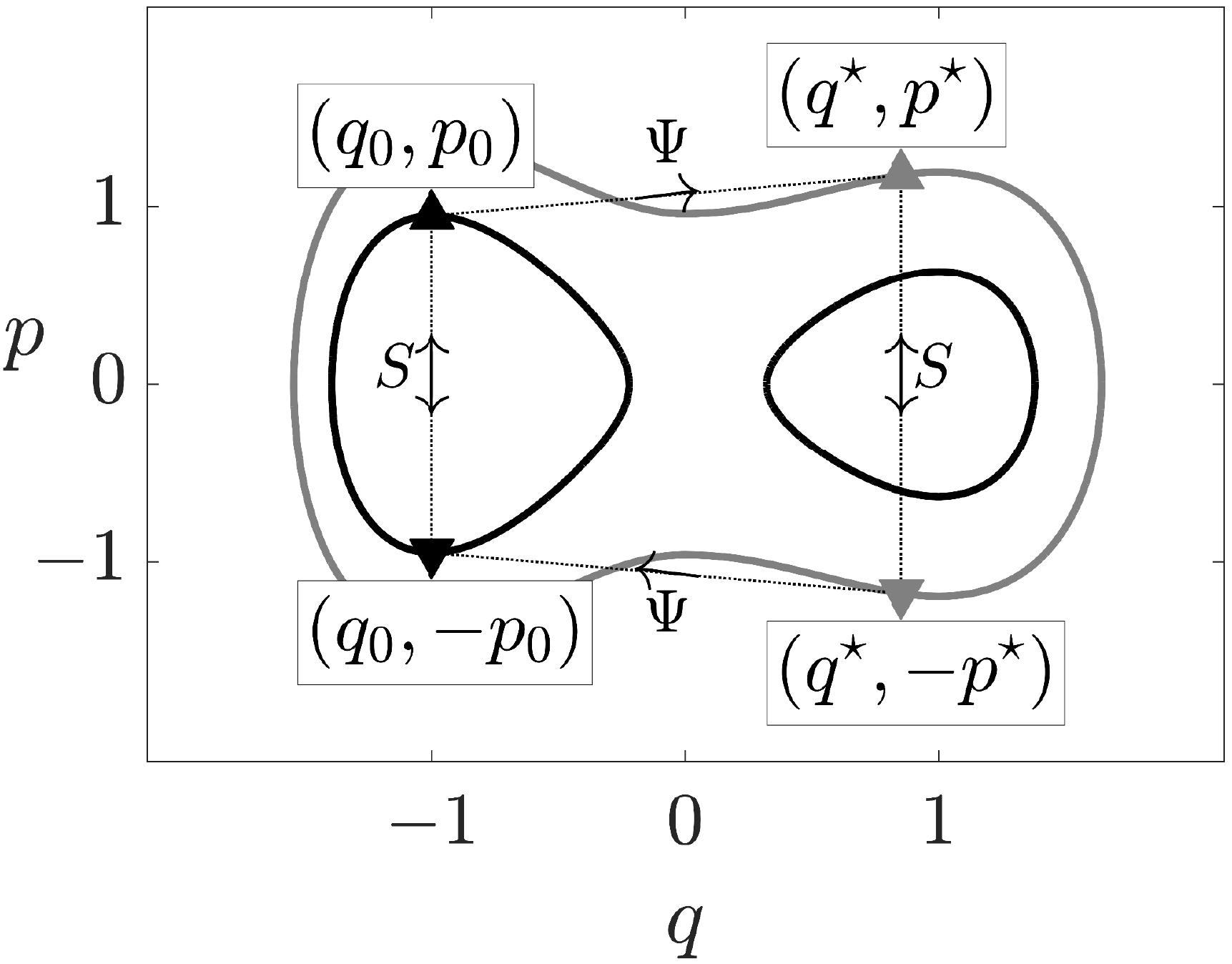}\end{center}
%\vspace{-15cm}
\caption{
A univariate target with probability modes at \(q=\pm 1\), leading to a double-well potential \(\mathcal{U}\).
The continuous lines are contours of constant \(H\). The symmetry of the contours with respect to the  axis \(p=0\)
is a consequence of the reversibility of the Hamiltonian flow. The solutions of Hamilton's equations
move from left to right when \(p>0\) and from right to left when \(p<0\), so when contours are reflected
over the horizontal axis
the arrow of time is reversed. If a reversible \(\Psi\) maps
\((q_0,p_0)\) into  \((q^*,p^*)\),  it has to map \((q^*,-p^*)\) into \((q_0,p_0)\), so as to preserve the
symmetry of the figure. The transition \((q_0,p_0)\mapsto(q^*,p^*)\) has an increase in energy and
\((q^*,-p^*)\mapsto(q_0,p_0)\) decreases energy in \emph{exactly} the same amount; this is the content of
formula \eqref{eq:nicefigure}. } \label{fig:rever_energy}
\end{figure}

We now use the proposition to bound the average of \(\Delta\).

\begin{theorem} \label{thm:mean_energy_error}
Let $\Psi: \mathbb{R}^{2d} \to \mathbb{R}^{2d}$ be a bijection that is volume-preserving and reversible with
respect
 to the momentum flip involution \eqref{eq:mflip}. If  the  integral
\beq\label{eq:mdelta1}
 m_{\Delta} = \int_{\mathbb{R}^{2d}} \Delta(q,p) e^{-H(q,p)} dq dp,
\eeq
 with $\Delta(q,p) = H(\Psi(q,p)) - H(q,p)$, exists, then
\begin{equation} \label{eq:mean_energy_error} 0 \le
m_\Delta \le \int_{\mathbb{R}^{2d}} \Delta(q,p)^2 e^{-H(q,p)} dq dp.
\end{equation}
Furthermore the first inequality is strict except in the trivial case where \(\Delta(q,p)\) vanishes for each \((q,p)\).
\end{theorem}
\begin{proof}
 By using the preceding proposition with \(g(x)=x\), we get
\begin{equation}\label{eq:aux20jul}
m_\Delta = \frac{1}{2} \int_{\mathbb{R}^{2d}}   \Delta(q,p)
\left( 1- e^{-\Delta(q,p)} \right) e^{-H(q,p)} dq
dp \;.
\end{equation}
The inequality $x (1-e^{-x}) > 0$, valid for all real $x\neq 0$, yields the lower bound in
\eqref{eq:mean_energy_error}.

For the upper bound, we apply the inequality $|e^x - 1| \le |x| (e^x + 1)$,  valid for all $x \in
\mathbb{R}$, to obtain
\begin{align*}
m_{\Delta} &\le \frac{1}{2} \int_{\mathbb{R}^{2d}}   \left| \Delta(q,p) \right| \left| 1- e^{-\Delta(q,p)} \right| e^{-H(q,p)} dq dp \\
&\le \frac{1}{2} \int_{\mathbb{R}^{2d}} \Delta(q,p)^2 \left( 1 + e^{-\Delta(q,p)} \right) e^{-H(q,p)} dq dp \\
&\le \int_{\mathbb{R}^{2d}} \Delta(q,p)^2 e^{-H(q,p)}\,dq dp.
\end{align*}
In the last step we used the Proposition~\ref{prop:natesh} with \(g(x)= x^2\).
\end{proof}

Figure~\ref{fig:rever_energy} illustrates the geometry behind formula \eqref{eq:nicefigure}. The figure makes
clear that to each initial condition \((q_0,p_0)\) with an energy \emph{increase} \(\Delta\geq 0\) there
corresponds an initial condition \((q^*,-p^*)\) with an energy \emph{decrease} of the same magnitude. In the
integrand in \eqref{eq:mdelta1} the energy increase at \((q_0,p_0)\) is weighed by \(\exp(-H(q_0,p_0))\), a
larger factor than the weight \(\exp(-H(q^*,-p^*))\) of the corresponding energy decrease. In addition, by
conservation of volume, if \((q_0,p_0)\) ranges in a small domain, then the corresponding points
\((q^*,-p^*)\) range in a small domain of the same measure. This explains why the integral \eqref{eq:mdelta1}
is positive.

We are of course interested in applying Theorem~\ref{thm:mean_energy_error} to the case where \(\Psi\) is the
map \(\Psi_\lambda\) in Step 2 of the HMC Algorithm~\ref{algo:numerical_hmc}. Assume for simplicity that
\(\lambda/h\) is an integer, so that \(t=\lambda\) coincides with one of the step points of the numerical
integration. Then the  energy error satisfies pointwise, \ie\  at each fixed \((q,p)\),
\[
\Delta(p,q) = H(\Psi_\lambda(q,p))-H(q,p) = \big[H(\varphi_\lambda(q,p))+\mathcal{O}(h^\nu)\big]-H(q,p)
= \mathcal{O}(h^\nu),
\]
where \(\nu\) is the order of the integrator and we have successively used the convergence result in
Theorem~\ref{th:convergence} and conservation of energy (Theorem~\ref{th:consenergy}). The bounds in
Theorem~\ref{thm:mean_energy_error} have important implications.

\begin{itemize}
\item \emph{The upper bound.} Even though pointwise, the energy error \(\Delta\) is of size
    \(\mathcal{O}(h^\nu)\), \emph{on average} (with respect to the Boltzmann-Gibbs distribution)
    \(m_\Delta\) is, at least formally, of size  \(\mathcal{O}(\Delta^2)=\mathcal{O}(h^{2\nu})\). The order
    of the average energy error is automatically \emph{twice} what we would have expected. This is clearly a
    pro of using a volume-reserving, reversible integrator.
\item \emph{The lower bound.} This is a con. Imagine a case where two uncoupled systems with Hamiltonian
    functions \(H_1\) and \(H_2\) are juxtaposed. The aggregate is a new Hamiltonian system with Hamiltonian
\(H_1+H_2\). The mean energy error for the aggregate is  \(\E(\Delta_1+\Delta_2) =
\E(\Delta_1)+\E(\Delta_2)\) and because both terms being added are \(\geq 0\) there is no room for
cancelation. In general,
the value of the energy and therefore the value of the energy error may be expected
to increase as the number of degrees of freedom increases, in agreement with the fact that in Physics
energy is an extensive quantity.
\end{itemize}

This discussion will be continued in Section~\ref{sec:highD_hmc}.

\subsection{Energy error in the standard Gaussian target}
\label{ss:energy_error_standard}

Sections~\ref{ss:fixedh} and \ref{ss:hmp} were devoted to investigating the behaviour of different integrators
when applied to the harmonic oscillator. We take up this theme once more, this time in the HMC context. The
aim is then to study what happens when HMC is used to sample  from the univariate standard normal, hoping that
any findings will be relevant to more complex distributions. (Of course in practice there is no interest in
using a MCMC algorithm to sample from a normal distribution. See the discussion on model problems at the
beginning of Section~\ref{ss:fixedh}.)

\subsubsection{Pointwise energy error bounds}
We return to  the situation in Proposition~\ref{prop:modham} and now study the energy error after \(n\)
time-steps \(\Delta(q_0,p_0)= H(p_n,q_n)-H(q_0,p_0)\), with \(H = (1/2)(p^2+q^2)\). The following result
\cite[Proposition 4.2]{BlCaSa2014} provides an upper bound for \(\Delta(q_0,p_0)\) \emph{uniform} in \(n\).

\begin{proposition}\label{prop:pointwise}
 The energy error, \(\Delta(q_0,p_0)\), may be bounded as
 \[
\Delta(q_0,p_0) \leq \frac{1}{2}(\chi_h^2-1) p_0^2,
\]
 if \(\chi_h^2 \geq 1\) or as
\[
\Delta(q_0,p_0) \leq \frac{1}{2}\left(\frac{1}{\chi_h^2}-1\right) q_0^2,
\]
if \(\chi_h^2 \leq 1\).
\end{proposition}
\begin{proof} We only prove with the first item; the other is similar. The ellipse \eqref{eq:ellipse} has its major
axis along the co-ordinate axis $p=0$ of the $(q,p)$ plane, as in the left panel of
Figure~\ref{fig:verlet_modified_hamiltonian}. Hence $2 H(q,p) = p^2+q^2$ attains its maximum on that ellipse
if $p=0$ which implies $q^2 = q_0^2+\chi_h^2 p_0^2$. If the   \emph{final} point \((q_n,p_n)\) of the
numerical trajectory happens to be at that maximum, $2\Delta(q_0,p_0) = (q_0^2+\chi_h^2p_0^2) -
(q_0^2+p_0^2)$.
\end{proof}

\subsubsection{Average energy error bounds}

We estimate next the \emph{average}  energy error \cite[Proposition 4.3]{BlCaSa2014}. Note that the bound
provided is once more \emph{uniform} in \(n\).
\begin{proposition}\label{prop:average}
In the situation described above, assume that $(q_0,p_0)$ is a Gaussian random vector with non-normalized
probability density function
 \( \exp(-H(q,p))\), with \(H = (1/2)(p^2+q^2)\). Then the expectation of the random variable
\(\Delta(q_0,p_0)\in\Reals^2\) is given by
\[
\E(\Delta) =  \sin^2(n\theta_h)\:\rho(h),
 \]
 where
 \[
\rho(h) =
\frac{1}{2}\left(\chi_h^2+\frac{1}{\chi_h^2}-2\right)
= \frac{1}{2}\left(\chi_h -\frac{1}{\chi_h}\right)^2\geq 0,
\]
and accordingly
\[
 0\leq \mathbb{E}(\Delta)  \leq \rho(h).
\]
.
\end{proposition}
\begin{proof}With the shorthand $c = \cos (n\theta_h)$, $s = \sin (n\theta_h)$, we may write
$$
2\Delta(q_0,p_0)  = \left(-\frac{1}{\chi_h}sq_0+cp_0\right)^2+\big(cq_0+\chi_hsp_0\big)^2-\big(p_0^2+q_0^2\big)
$$
or
$$
2\Delta(q_0,p_0)  = s^2\left(\frac{1}{\chi_h^2}-1\right)  q_0^2+ 2cs\left(\chi_h-\frac{1}{\chi_h}\right) q_0p_0
+s^2\big(\chi_h^2-1\big) p_0^2.
$$
Since $\mathbb{E}(q_0^2) = \mathbb{E}(p_0^2) = 1$ and $\mathbb{E}(q_0p_0) = 0$, the proof is ready.
\end{proof}

\subsubsection{Energy error bounds for Verlet}

Let us illustrate the preceding results in the case of the Verlet integrator. For the velocity version, we
find from Example~\ref{ex:vv} and \eqref{eq:chi}, for \(0<h<2\),
$$
\chi_h^2 = \frac{h^2}{1-\left(1-\frac{h^2}{2}\right)^2} = \frac{1}{1-\frac{h^2}{4}}>1.
$$
The bound in Proposition~\ref{prop:pointwise} reads
\begin{equation}\label{eq:boundverlet}
\Delta(q_0,p_0) \leq \frac{h^2}{8 (1-\frac{h^2}{4})}\:p_0^2.
\end{equation}
For $h = 1$, $\Delta(q_0,p_0) \leq p_0^2/6$; therefore, if $-2 < p_0 <2$ (an event that for a standard normal
distribution has probability $>95\%$), then $\Delta(q_0,p_0) < 2/3$ which results in a probability of
acceptance $\geq 51\%$, regardless of the number $n$ of time-steps.

The position Verlet integrator (Example~\ref{ex:pv} and Remark~\ref{rem:swap}) has $\chi_h^2 = 1-h^2/4 < 1$
provided that $0<h<2$. Proposition ~\ref{prop:pointwise} yields
$$
\Delta(q_0,p_0) \leq \frac{h^2}{8 (1-\frac{h^2}{4})}\:q_0^2
$$
(as one may have guessed from \eqref{eq:boundverlet} by symmetry).

From Proposition~\ref{prop:average}, for both the velocity and the position versions,
\begin{equation}\label{eq:rhoverlet}
0\leq \mathbb{E}(\Delta)\leq\rho(h)= \frac{h^4}{32(1-\frac{h^2}{4})}.
\end{equation}
We draw the attention to the exponent of \(h\) in the numerator: \emph{even though, pointwise, for the
harmonic oscillator, energy errors for the Verlet integrator are \(\mathcal{O}(h^2)\), they are
\(\mathcal{O}(h^4)\) on average}, which of course matches our earlier finding in
Theorem~\ref{thm:mean_energy_error}. For $h=1$ the expected energy error is $\leq 1/24$. Halving $h$ to
$h=1/2$, leads to an expected energy error $\leq 1/480$. The conclusion of the examples above is that
\emph{acceptance rates for Verlet} are likely to be high even if the step size \(h\) is not small.

\subsection{Velocity Verlet or position Verlet?}

\begin{figure}[t]
\begin{center}
\includegraphics[width=0.45\textwidth]{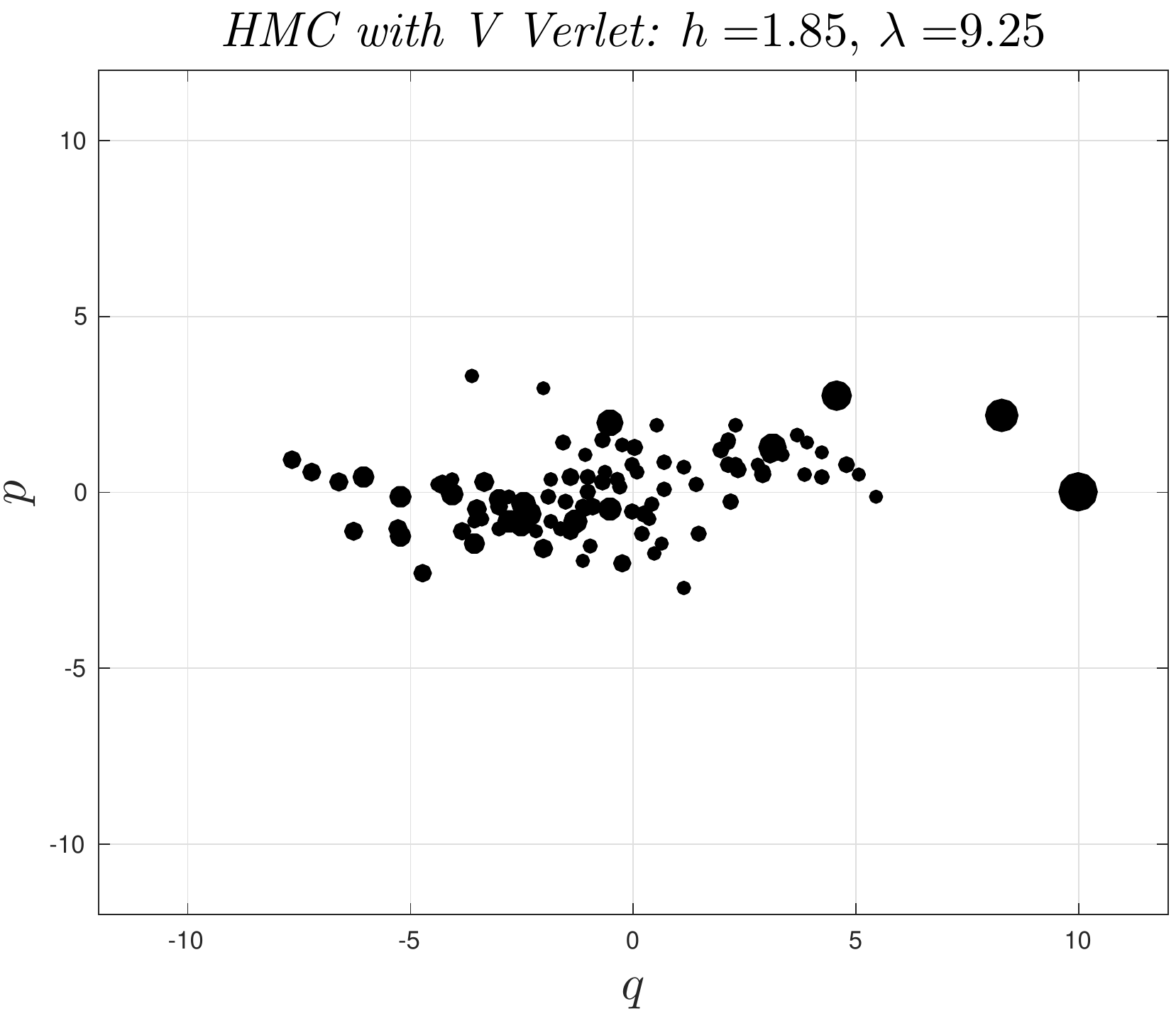}
\includegraphics[width=0.45\textwidth]{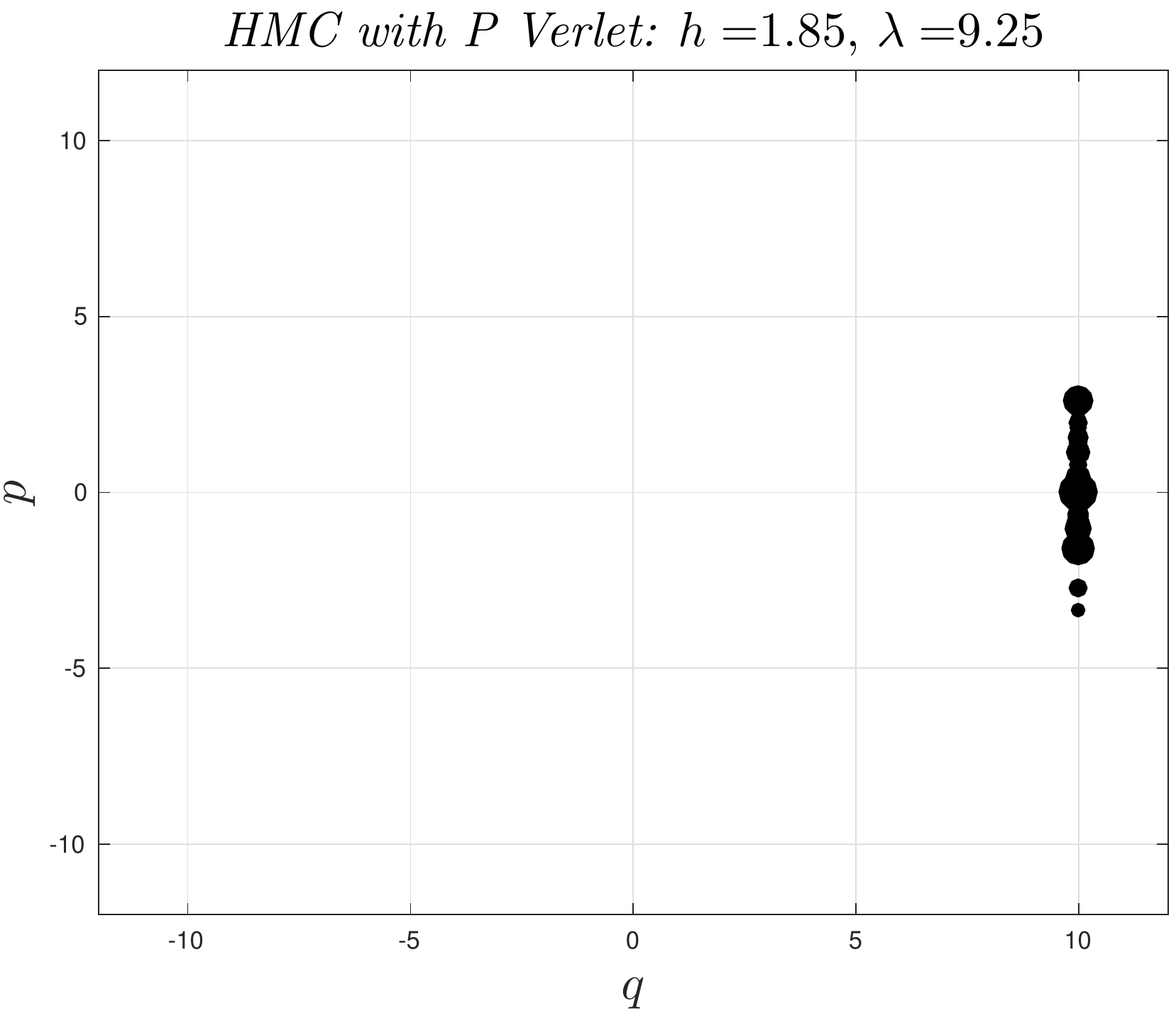}
\end{center}
\caption{ \small
 Discrete trajectories of the Markov chain when HMC for the
 standard Gaussian target is operated with velocity (left panel)
and position (right panel) Verlet.
Each marker is a state of the Markov chain (intermediate values of \((q,p)\) along the numerical
trajectories  are not depicted).
The size of the markers is related to the index $i$ in the chain: states
corresponding to larger values of $i$  have smaller markers.  The chains start at \(q=10\). On the right,
 all chain transitions result in rejection,
so that \(q_i=10\) for all values of \(i\); position Verlet  gets stuck. On the left, the velocity chain
quickly identifies the interval around \(q=0\) where the target distribution is concentrated.
 }
 \label{fig:hmc_numerical_linear_1d}
\end{figure}

The symmetry between the velocity and position Verlet integrators may lead to the conclusion that their
performances in the HMC scenario are equivalent. This is not the case, because in
Algorithm~\ref{algo:numerical_hmc} (for a general target distribution) the \(q\) and \(p\) variables do not
play a symmetric role at all. In particular, samples of \(p\) from the correct marginal distribution (Step 1)
are used as initial conditions for the integration legs in Step 2. On the contrary, when the Markov chain is
initialized, the user has to \emph{guess} a suitable starting value for \(q\) (see the end of
Section~\ref{ss:markov}). In some applications, the chosen  starting value may actually correspond to a
location of low probability, and we now study the difficulties that may arise in that case.

Figure~\ref{fig:hmc_numerical_linear_1d}  corresponds to the standard Gaussian/harmonic oscillator target and
compares the performance of the velocity and position algorithms
 when the chain is initialized with \(q\) at a location of very  low probability, \(q_0=10\), ten standard
deviations away from the mean (the initialization of \(p\) is of no consequence since momenta are discarded in
the momentum--refreshment Step 1 of Algorithm~\ref{algo:numerical_hmc}). The position integrator gets stuck
 where its velocity counterpart succeeds without problems.

Insight into the different behaviors of velocity and position Verlet may be gained from
Figure~\ref{fig:verlet_deltaH}, that corresponds to integrations with the large value \(h=1.85\)  carried out
in the interval \(0\leq t \leq  \lambda=9.25\) (five time-steps). Regardless of the choice of \(|q_0|\),
velocity Verlet will result in an energy decrease (and therefore in acceptance) if the initial sample
\(\xi_0\) of the momentum is such that \((q_0,\xi_0)\) lands in the grey area of the left panel, which will
happen with large probability because small values of \(|\xi_0|\) are likely to occur. On the contrary, in the
position algorithm energy decreases only occur if very large values of \(\xi_0\) are drawn. Certainly, there
is symmetry between the gray areas of both panels; however, in HMC, having energy decreases near the \(q\)
axis is helpful and having them near the \(p\) axis is not. It may  help to go back to the bound
\eqref{eq:boundverlet} that shows that, for velociy Verlet, \(p_0=0\) ensures an energy decrease; compare both
panels in Figure~\ref{fig:verlet_modified_hamiltonian} as well.

\begin{figure}[t]
\begin{center}
\includegraphics[width=0.45\textwidth]{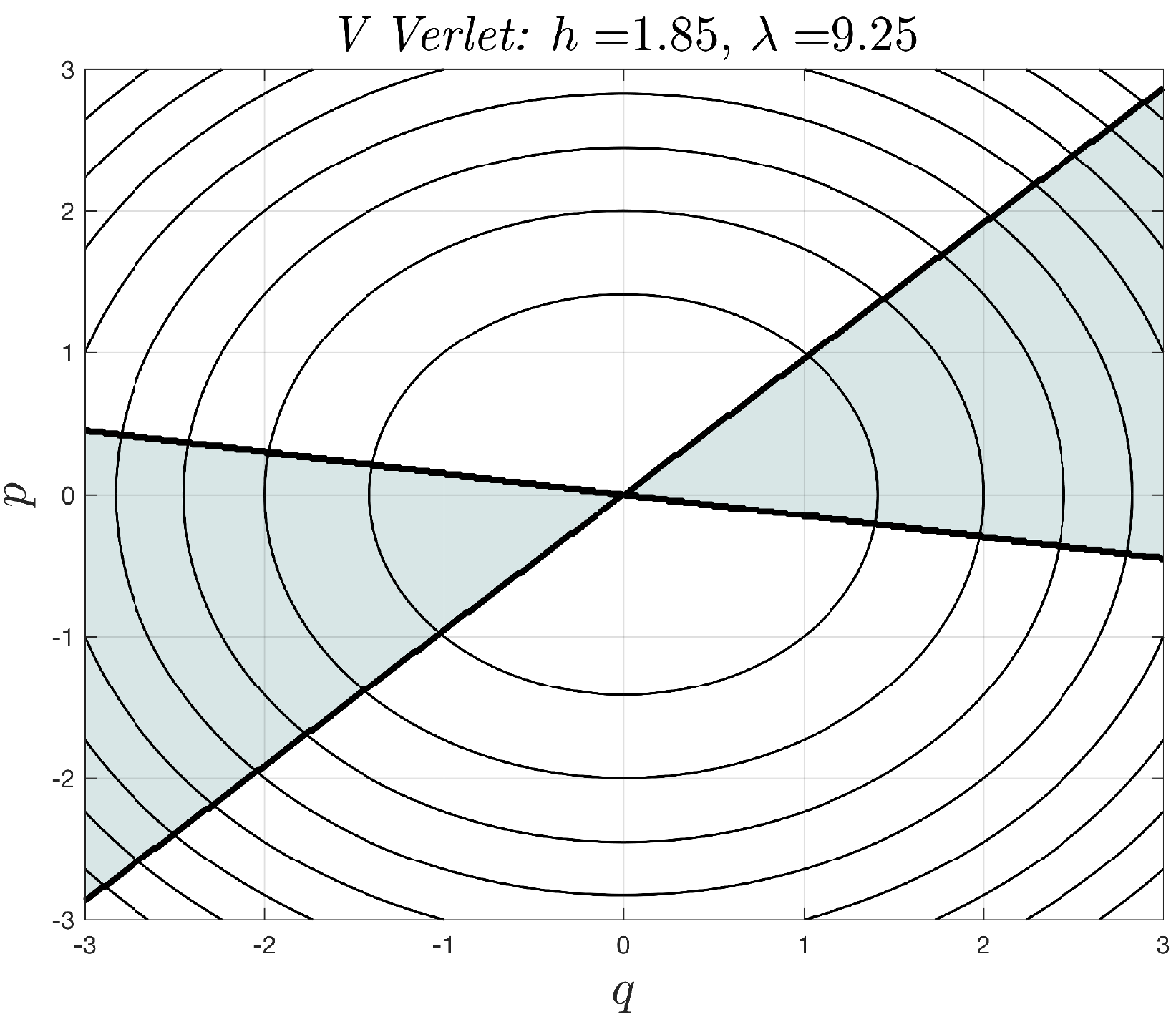}
\includegraphics[width=0.45\textwidth]{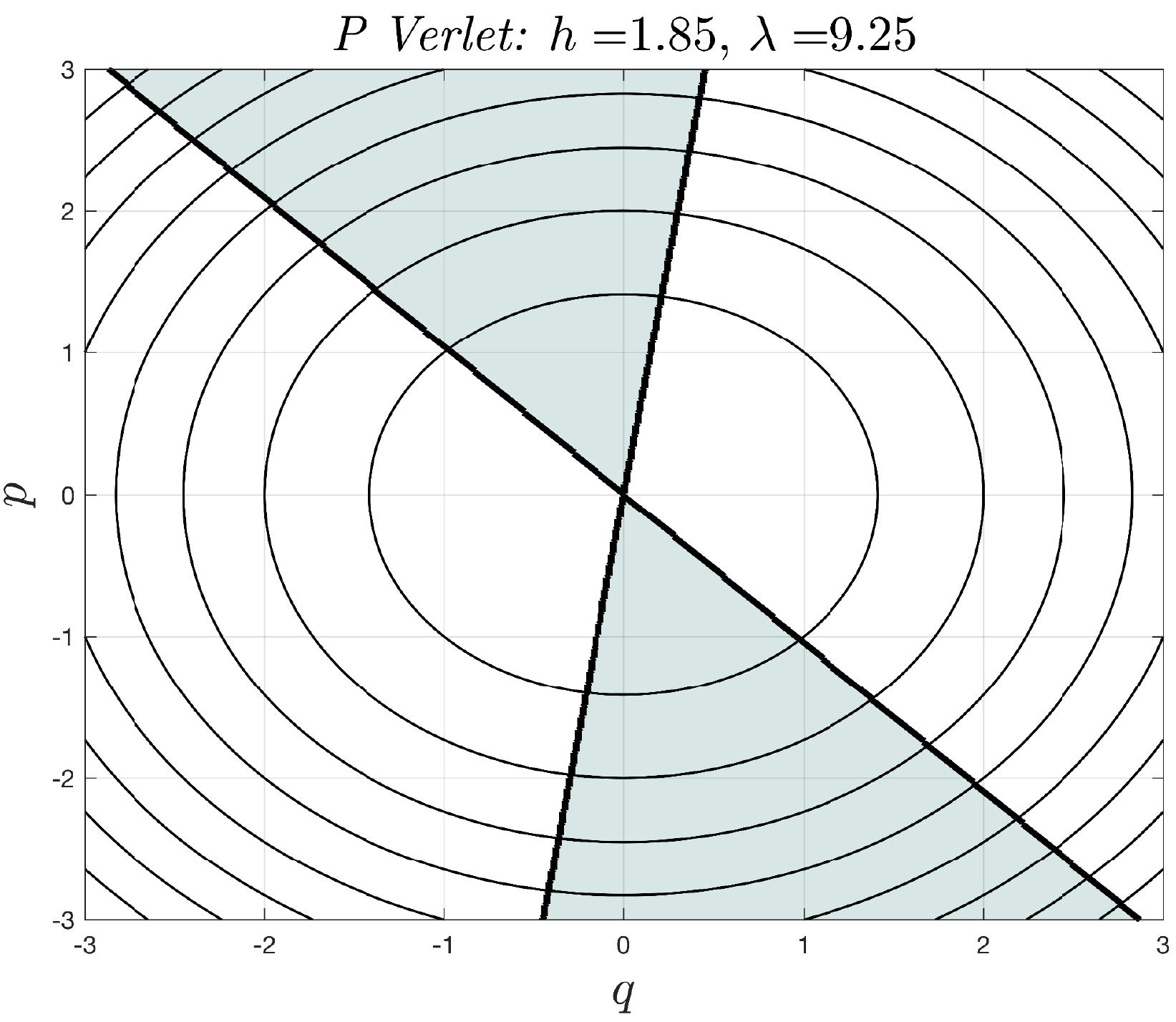}
\end{center}
\caption{ \small  Regions where the energy error $\Delta (q,p)$
is positive/negative for velocity (left panel) and position (right panel) Verlet.
In the gray-shaded regions $\Delta(q,p)$
is negative, ensuring acceptance of the Markov transition.  The contour lines in the
 background shows level sets of $H(q,p)$.
 }
 \label{fig:verlet_deltaH}
\end{figure}

To conclude this discussion we emphasize that if the chains were initialized by drawing samples from the
marginal distribution of \(q\) (\ie\ started at stationarity), then the behaviour of velocity and position
Verlet would be the same due to the \(q/p\) symmetry.

\subsection{Multivariate Gaussian model}
\label{ss:multivariate}
 We study once more the quadratic
Hamiltonian \eqref{eq:MKmodel} and note that the covariance matrix \(\Sigma\) of the target is the inverse of
the stiffness matrix \(K\). As discussed in Proposition~\ref{prop:KM}, in the particular case where \(M=I\),
the square roots of the eigenvalues of \(K\) are the angular frequencies of the dynamics; the eigenvalues of
\(\Sigma\) are of course the variances of the target along the directions of the eigenvectors. Thus
\emph{small variance implies high frequency}.

\subsubsection{Average energy error in the multivariate case}

Proposition~\ref{prop:average} may be extended to the multivariate Gaussian case. We begin by providing a
variant of Proposition~\ref{prop:KM}.

\begin{proposition}\label{prop:KMsecond}
The change of dependent variables, \(q= L^{-T}U\Omega^{-1} Q\), \(p =LU P\), decouples the system into a
collection of \(d\) harmonic oscillators (superscripts denote components):
\[
\frac{d Q^i}{dt} = \omega_i P^i,\qquad
\frac{d P^i}{dt} = -\omega_i  Q^i, \qquad i = 1,\dots, d.
\]

In terms of the new variables, the function in \eqref{eq:MKmodel} is given by
\[
\frac{1}{2}P^TP + \frac{1}{2}Q^TQ = \frac{1}{2}\sum_i \big((P^i)^2+(Q^i)^2\big).
\]
\end{proposition}
Therefore the \(2d\) variables \(Q^i\) and \(P^i\) are independent standard Gaussian.\footnote{Obviously, the
last display  does not give the Hamiltonian function for the Hamiltonian dynamics for \((Q,P)\).
Proposition~\ref{prop:canonical change} does not apply because the change of variables \((q,p) = \Phi(Q,P)\)
is not symplectic.}

As discussed in connection with Proposition~\ref{prop:KM}, we may assume that the numerical method is applied
to each of the uncoupled oscillators. After recalling Remark~\ref{rem:omega},  the following result is easily
proved.

\begin{theorem} Consider a stable integration of \eqref{eq:MKmodel} with a (consistent), reversible, volume
preserving integrator applied with a stable value of \(h\). If \((q_0,p_0)\) are random variables with
probability density function (proportional to) \(\exp(-H)\), then for each \(n\) the expectation of the energy
error \(\Delta(q_0,p_0) = H(q_n,p_n)-H(q_0,p_0)\) may be bounded as
\begin{equation}\label{eq:multivariate}
0\leq  \mathbb{E}(\Delta) \leq \sum_{j=1}^d \rho(\omega_j h),
 \end{equation}
where \(\rho\) is the function defined in Proposition~\ref{prop:average}.
\end{theorem}

\subsubsection{An illustration}
\begin{figure}
\begin{center}
\includegraphics[width=0.45\textwidth]{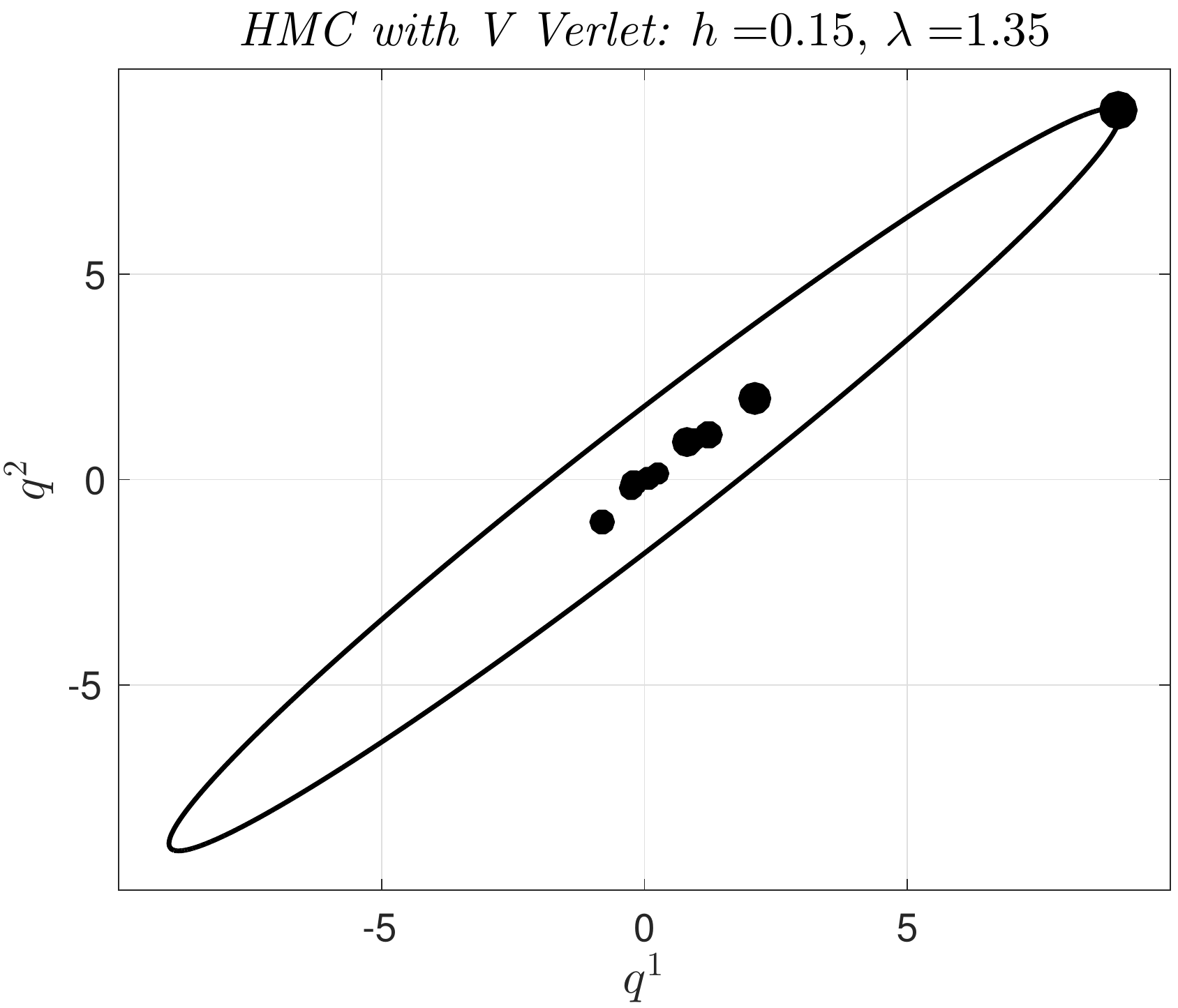}
\includegraphics[width=0.45\textwidth]{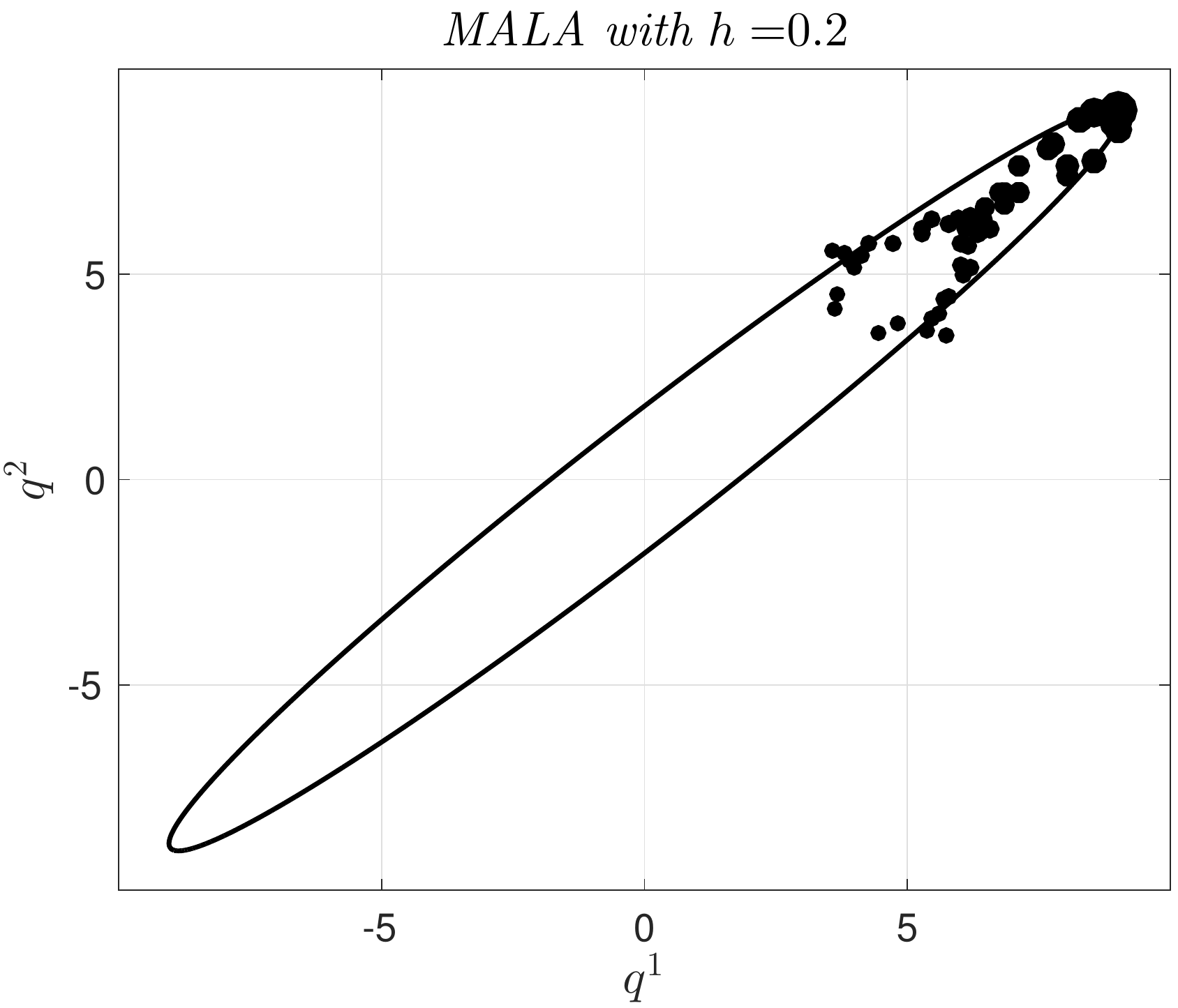}
\end{center}
\caption{
Bivariate Gaussian distribution, \((q^1,q^2)\) plane; the ellipse is a contour of constant probability density. HMC operated with velocity Verlet (left panel) and MALA (right panel).
The size of the markers is related to $i$: points along the Markov trajectory corresponding
to larger values of $i$  have smaller markers.  The computational
budget for both algorithms is fixed at 100 force evaluations.
Along these trajectories, the average acceptance probability for HMC and MALA is 93\% and 82\%,
respectively.  The stability requirement for Verlet is $h < 0.2$, which is set
by the component of the target distribution with the smallest variance.  In relation to MALA,
note that HMC relaxes faster to a region of high probability.
}
\label{fig:hmc_numerical_linear_2d}
\end{figure}

As an example, we consider a two-dimensional target where  \[ K = \Sigma^{-1} = \frac{1}{2}
\begin{bmatrix} 101 & -99 \\ -99 & 101 \end{bmatrix}
\] This matrix has eigenvalues $\omega_1^2 = 1$ (with eigenspace \(q^1=q^2\)) and $\omega_2^2=100$
(with eigenspace \(q^1= -q^2\)). The mass matrix is chosen to be \(M = I\). We use HMC with velocity Verlet;
stability requires $h \omega_2 < 2$ or $h < 0.2$ and
 we set $h=0.15$.  The chain is started at
$(q^1, q^2)=(9,9)$ where $\mathcal{U}(9,9)=81$, \ie\ about 9 standard deviations in the direction of largest
variance. The duration parameter \(\lambda\) is set to be equal to 1.35 (\ie\ there are \(9\) time-steps in
each numerical integration leg) and the number of transitions in the chain was the maximum attainable with a
computational budget of 100 force evaluations. The left panel in Figure~\ref{fig:hmc_numerical_linear_2d}
shows that HMC rapidly relaxes to a region of high probability.

\begin{remark}It may be of interest to relate this experiment to the discussion
of stiffness in Example~\ref{ex:stabilty2df}. Here it is the oscillation with frequency  \(\omega_1\) (\ie\
the evolution along the diagonal \(q^1=q^2\)) that matters most; however \(h\) has to be chosen in terms of
the oscillation with frequency \(\omega^2\).\footnote{Example~\ref{ex:stabilty2df} had \(\omega_2 = 100\);
here we have chosen a smaller value of \(\omega_2\) so as not to blur the figure.} Nevertheless HMC succeeds
because, at each Markov transition, several time-steps are taken.
\end{remark}

As a comparison we have also implemented the well-known algorithm MALA, with the same initial state and same
computational budget.  MALA may be described as the algorithm that results when, in HMC operated with velocity
Verlet, the duration parameter \(\lambda\) is chosen to coincide with \(h\), \ie\ each integration leg only
takes one time-step. In this way the accept/reject mechanism operates after each individual time-step in MALA
(but only after \(\lambda/h\) successive time-steps in HMC). The right panel shows MALA zigzags in a region of
low probability.

\begin{remark}In connection with the preceding remark, we note that for MALA the progress of the Markov
chain along the \(q^1=q^2\) diagonal is slow. Each time-step moves the state by a small amount, because \(h\)
was chosen by taking into account the behaviour of the problem along the \(q^1=-q^2\) diagonal. The momentum
refreshments, which in MALA take place after every individual time-step, change the direction of motion in the
\((q^1,q^2) \) plane, thus inducing a random-walk behaviour.  \emph{HMC avoids this random walk behaviour by
taking sufficiently many time-steps between consecutive accept/reject steps so that integration legs make
substantial strides in the solution components with high variance.}
\end{remark}

\subsection{May Verlet be beaten?}

Currently, velocity Verlet is the integrator of choice within HMC algorithms. Is this the best possibility? In
this subsection we will discuss whether it is possible to construct a palindromic splitting formula with
\(s>1\) stages \eqref{eq:even}--\eqref{eq:odd} that improves on the velocity Verlet algorithm in the sense
that, when operated with a step size \(h\), leads to smaller energy errors ---higher acceptance
probabilities--- than velocity Verlet with step size \(h/s\). Because, as illustrated above, Verlet is often
very successful for large values of the step size, such a construction has to be based on investigating the
behaviour of the integrators for finite \(h\) and therefore cannot be guided by information (such as the
expansion in Theorem~\ref{th:modinfsplit}) that corresponds to the behaviour as \(h\to 0\). We therefore
resort to Gaussian models with Hamiltonian of the form \eqref{eq:MKmodel}. In the remainder of this subsection
we shall use the symbol \(\tau\) to refer to the actual value of the step size implemented in the algorithm
and keep the letter \(h\) for the (nondimensional) product \(\omega \tau\), where \(\omega\) represents one of
the frequencies present in the problem. With this notation, Verlet is stable for \(0<h<2\) or
\(0<\tau<2/\omega_{\rm max}\).

When the number \(d\) of degrees of freedom equals 1, velocity Verlet is indeed the best choice. We know, from
Section~\ref{ss:energy_error_standard}, that Verlet will deliver high acceptance rates for \(h\) just below
its stability limit and, from Section~\ref{ss:optimal}, that Verlet has the longest stability interval. So
Verlet peforms well for HMC for values of \(\tau\) where other integrators are not even stable. Essentially
the same is true for \(d\) small, as one sees by decoupling the problem into \(d\) one-degree-of-freedom
oscillators as we did in Section~\ref{ss:multivariate}. However, as \(d\) increases, Verlet  is not likely to
keep high acceptance rates if \(\tau\) is just below \(2/\omega_{\rm max}\): in fact the additive character of
the energy, noted in the discussion of Theorem~\ref{thm:mean_energy_error}, implies that, even if the energy
error in each of the one-dimensional oscillators in the problem is small, the energy error for the overall
system will likely be large. In that scenario, where, on \emph{accuracy} grounds, Verlet has to be operated
with \(\tau\) significantly smaller than the stability limit \(2/\omega_{\rm max}\), there is room for
improvements in efficiency by resorting to more sophisticated integrators, as we discuss next.

\subsubsection{Two stages}
\begin{figure}[t]
\begin{center}
\includegraphics[width=0.45\textwidth]{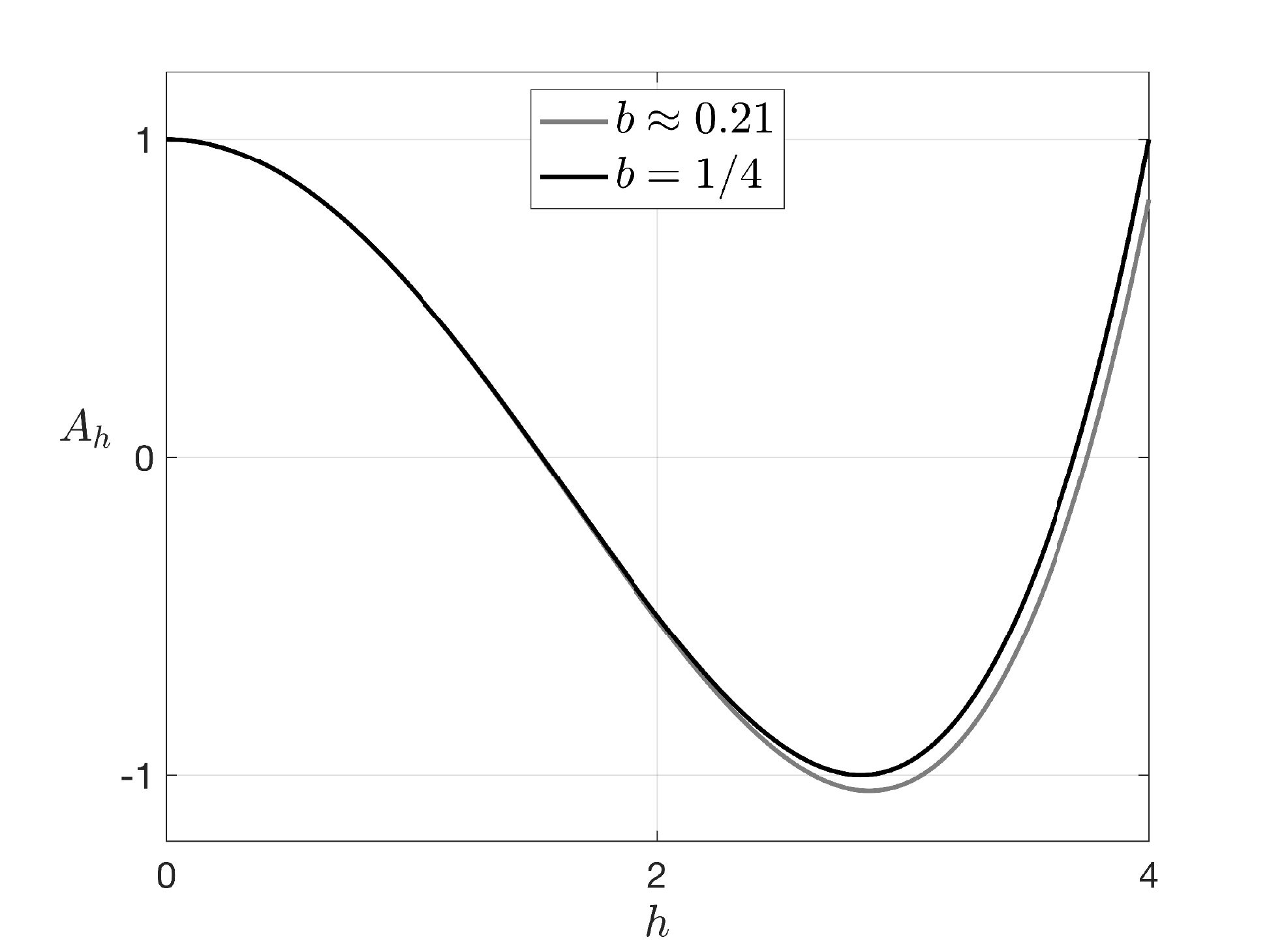}
\includegraphics[width=0.45\textwidth]{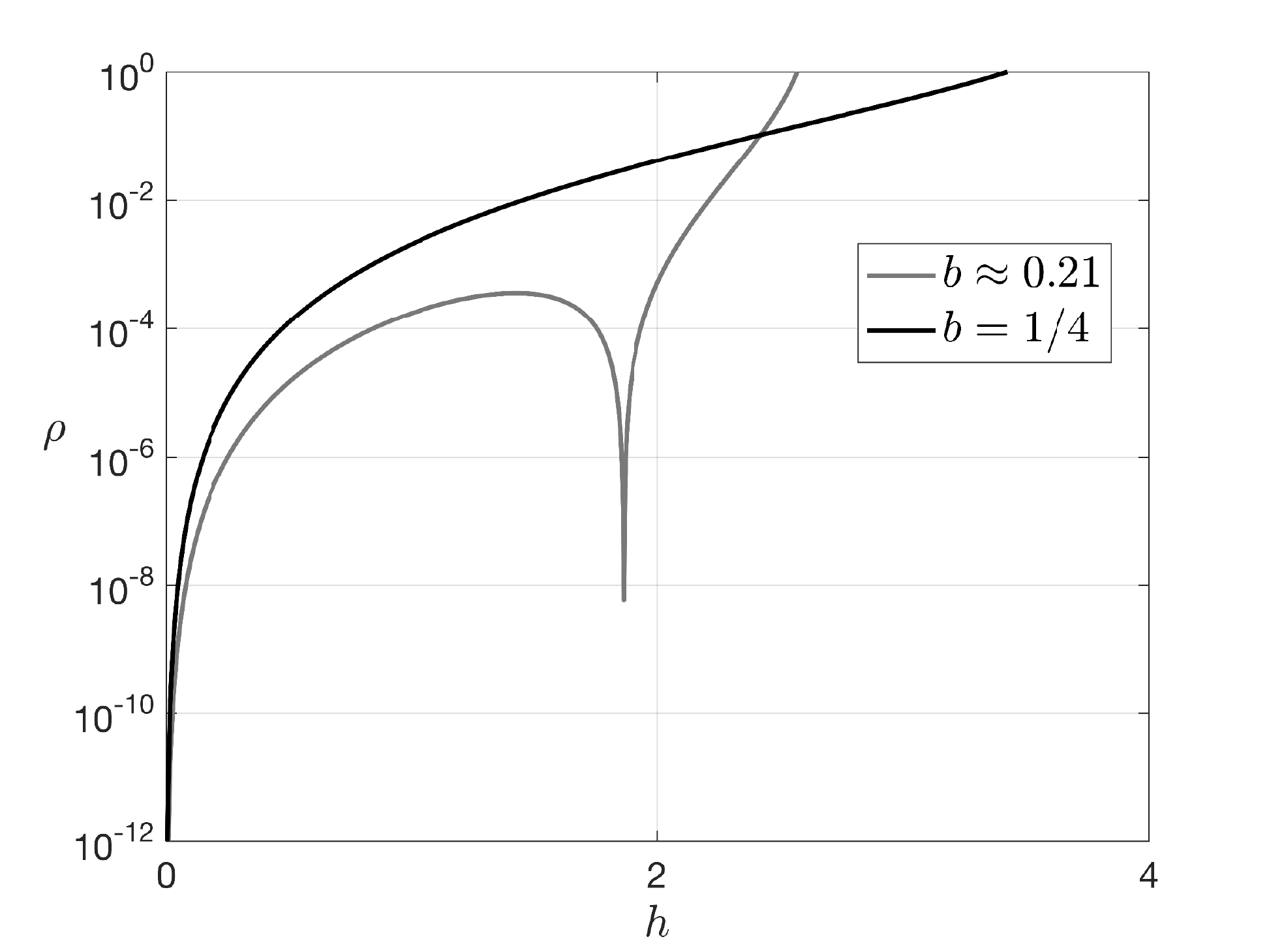}
\end{center}
\caption{ \small Two-stage  palindromic integrators
 \eqref{eq:twostagefamily} corresponding to  \(b=1/4\)   or \(b\) given by \eqref{eq:nuestro}. Left panel:
stability polynomial \(A_h\). For \(b=1/4\) (which provides two steps of velocity Verlet of step length \(\tau/2\)) stabilty is lost
at \(h=4\) and there is a double root of the equation \(A_h=-1\). The method \eqref{eq:nuestro} has a
significantly shorter stability interval because under perturbation the double root of  \(A_h=-1\) splits
 into two real roots. Right panel: the quantity \(\rho=\rho(h)\) that governs the expected energy error for
 Gaussian problems (see \eqref{eq:multivariate}). When the methods are operated with \(0<h<2\), \ie\ \(0<\tau<2/\omega_{\rm max}\), the choice
\eqref{eq:nuestro} yields values of \(\rho\) much smaller than those corresponding to \(b=1/4\).
 }
 \label{fig:two_stage}
\end{figure}

We begin by studying the one-parameter family \eqref{eq:twostagefamily} of methods with two stages (\ie\
essentially two force evaluations per step). For the choice \(b= 1/4\), one step of length \(\tau\) of the
two-stage method is equivalent to two steps of length \(\tau/2\) of velocity Verlet (see equation
\eqref{eq:opt} with \(N=2\)). The stability polynomial \(A_h\)  was described in terms of the Chebyshev
polynomial \(T_2\) in Section~\ref{ss:optimal} and its graph may be seen in the left panel of
Figure~\ref{fig:two_stage}. Stability is lost at \(h =4\) when \(A_h = 1\); at \(h=2\sqrt{2}\approx 2.82\),
there is a double root of the equation \(A_h=-1\)  (corresponding to the double root at \(\zeta = 0\) of the
equation \(T_2(\zeta) = 2\zeta^2-1 = -1\)). Small perturbations of \(b=1/4\) turn the double root of the
equation \(A_h=-1\) into two real roots in the neigbourhood of \(h=2\sqrt{2}\); after such perturbation the
length of the stability interval drops from 4 to \(\approx 2\sqrt{2}\). Accordingly, in problems where \(d\)
is small enough for standard velocity Verlet  to work well with \(\tau\omega_{\rm max}\) in the interval
\((\sqrt{2},2)\), no two-stage formula  with \(b\neq 1/4\) can improve upon Verlet. However if, when applying
Verlet, \(\tau\omega_{\rm max}\) has to be chosen below \(\sqrt{2}\) on accuracy grounds, then, as we will see
now, efficiency may be improved by choosing \(b\) in \eqref{eq:twostagefamily} different from \(1/4\).

The paper \cite{BlCaSa2014} suggests the following procedure to find the value of \(b\). Numerical experiments
in that reference show that,  for Gaussian problems with \(d\leq 1000\), (one-stage) velocity Verlet achieves
acceptance rates \(\geq 20\:\%\) when \(\tau\omega_{\rm max}\leq 1\) (\ie\ when \(\tau\) is less than a half
of the maximum allowed by stability). It is then assumed that two-stage integrators should be demanded to
perform well when \(\tau\omega_{\rm max}\leq 2\). On the other hand, in view of \eqref{eq:multivariate}, for
Gaussian models performance is governed by the quantity \(\rho\)  defined in Proposition~\ref{prop:average},
which for the one-parameter family \eqref{eq:twostagefamily} is found to be \[ \rho(h;b) = \frac{h^4
\big(2b^2(1/2-b)h^2+4b^2-6
   b+1\big)^2}
   {8 \big(2-b h^2\big)
   \big(2-(1/2-b) h^2\big)
   \big(1-b(1/2-b) h^2\big)}.
\]
Then \(b\) is chosen by minimizing
\[
\| \rho\|_{(2)} =
\max_{0\leq h\leq 2} \rho(h;b);
\]
which yields \(b = 0.21178\dots\) To avoid cumbersome decimal expressions, \cite{BlCaSa2014} instead uses the
approximate value
 \begin{equation}\label{eq:nuestro}
   b = \frac{3-\sqrt{3}}{6} \approx 0.21132,
   \end{equation}
which gives \(\|\rho\|_{(2)}\approx 5\times 10^{-4}\). For comparison \(b=1/4\) has a substantially larger
\(\|\rho\|_{(2)}\approx 4\times 10^{-2}\). Numerical experiments reported in \cite{BlCaSa2014} confirm that,
when \(d\) is not small the method \eqref{eq:nuestro} is a clear improvement on Verlet in an example where the
target is not Gaussian. See also \cite{MaKlSk2017}.

\subsubsection{Three stages}
\begin{figure}[t]
\begin{center}
\includegraphics[width=0.45\textwidth]{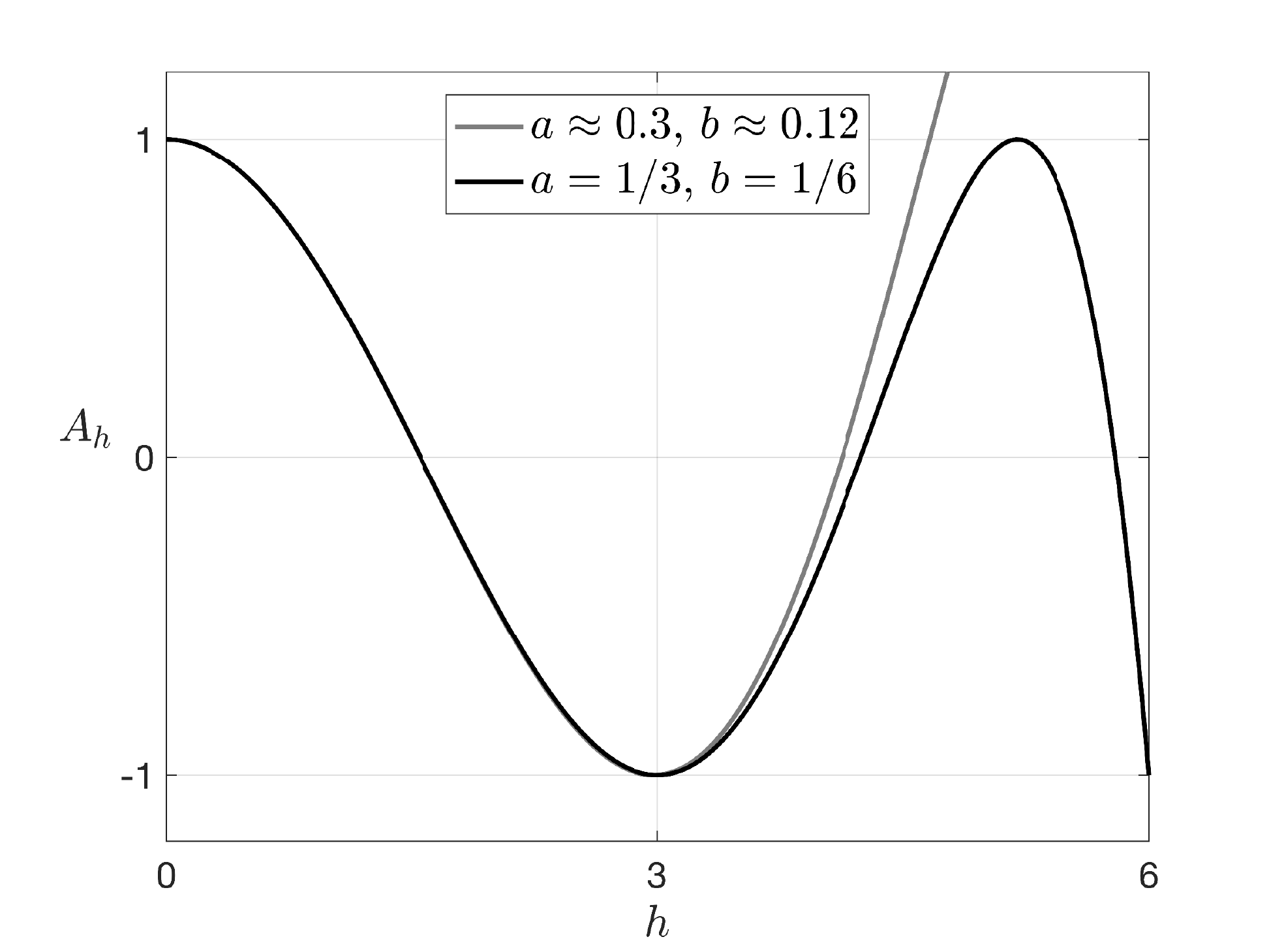}
\includegraphics[width=0.45\textwidth]{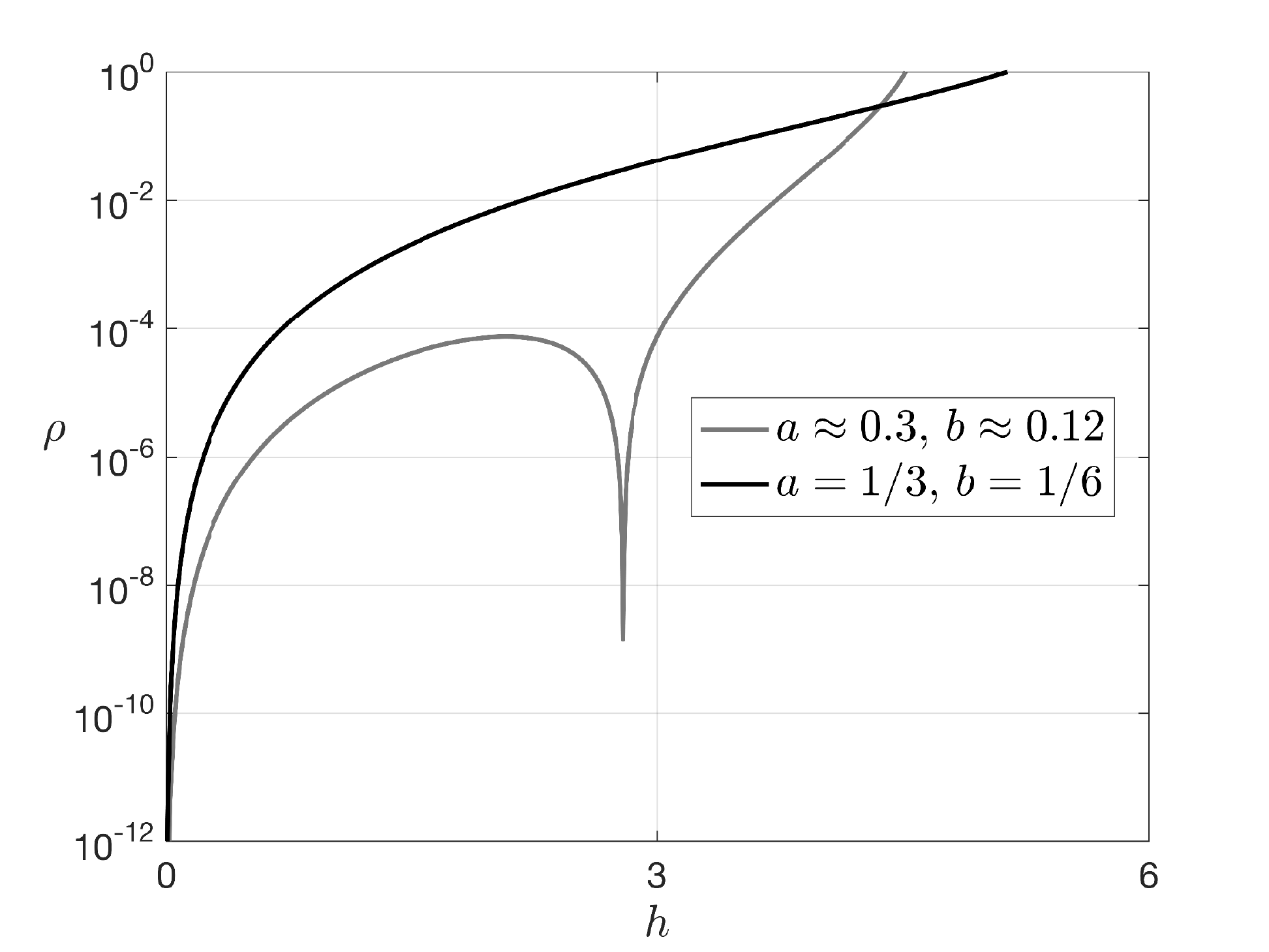}
\end{center}
\caption{
\small Three-stage  palindromic integrators
 \eqref{eq:threestagefamily} for  \(a=1/3\),
  \(b=1/6\)   or \(a\), \(b\) given by \eqref{eq:minrhotres}.
Left panel:
stability polynomial \(A_h\). Right panel: the quantity \(\rho=\rho(h)\) that governs the
 expected energy error for
 Gaussian problems (see \eqref{eq:multivariate}). When the methods are operated with \(0<h<3\), the choice
\eqref{eq:minrhotres} yields values of \(\rho\) much smaller than those corresponding to
 \(a=1/3\), \(b=1/6\).
 }
 \label{fig:three_stage}
\end{figure}
A similar study may be undertaken for the three-stage family \eqref{eq:threestagefamily} with two parameters
\(a\) and \(b\). The choice \(a =1/3\), \(b=1/6\) yields a method consisting of the concatenation of three
successive steps of length \(\tau/3\) of velocity Verlet (see equation \eqref{eq:opt} with \(N=3\)). The
graphs of the stability polynomial \(A_h\) may be seen in the left panel of Figure~\ref{fig:three_stage}.
Stability is lost at \(h =6\) when \(A_h = -1\). The double root of \(T_3(\zeta) = 4\zeta^3-4\zeta = 1\) at
\(\zeta=-1/2\) gives rise to a double root of \(A_h = 1\) at \(h = 3\sqrt{3}\approx 5.19\). The double root of
\(T_3(\zeta) = - 1\) gives rise to a double root of \(A_h = -1\) at \(h=3\). Perturbing the values \(a =1/3\),
\(b=1/6\) leads in general to methods with (short) stability intervals of length \(\approx 3\), because
typically the double root of \(A_h = -1\) splits into two single roots. However it is possible to perturb
\(a\) and \(b\) in such a way that, after perturbation, \(A_h=-1\) has a double root close to \(h=3\) as
analyzed in \cite{CaSa2017}. For such particular perturbations, stability is lost by crossing the line
\(A_h=1\). This is what happens for the choice
\begin{equation}\label{eq:minrhotres}
a = 0.29619504261126,\quad b =  0.11888010966548,
\end{equation}
found in \cite{BlCaSa2014} by minimizing
\[
\| \rho\|_{(3)} =
\max_{0\leq h\leq 3} \rho(h;b).
\]
Note that  the range of the maximum is now \({0\leq h\leq 3}\) to account  for the fact that, when using three
force evaluations per step, we aim at integrations with \(\tau\omega_{\rm max}\leq 3\) to be competitive with
Verlet. The choice \eqref{eq:minrhotres} results in \(\|\rho\|_{(3)} \approx 7\times 10^{-5}\) and a stability
interval of length \(\approx 4.67\dots\) The paper \cite{BlCaSa2014} reports experiments that show the
superiority of \eqref{eq:minrhotres} over standard velocity Verlet and also describes the construction of
optimized four-stage integrators.

\subsubsection{The AIA approach}
Let us go back to  the one-parameter family \eqref{eq:twostagefamily} of methods with two stages. The choice
\eqref{eq:nuestro} is based on the assumption that with \(b=1/4\) the integrator (equivalent to velocity
Verlet) would be used with \(\tau\omega_{max}\) in the interval \((0,2)\). There is of course a degree of
arbitrariness in the choice of this interval. When an interval \((0,c)\) with \(c\) slightly above \(2\) is
considered, then minimization of
\begin{equation}\label{eq:max}
\max_{0\leq h\leq c} \rho(h;b)
\end{equation}
results in a value of \(b\) slightly above that corresponding to \((0,2)\). This improves the length of the
stability interval, but, with the notation in Theorem~\ref{th:modinfsplit},  increases the quantity
\(C_{2,1}^2+ C_{2,2}^2\), so that the integrator becomes less accurate  in the limit \(h\to 0\) (note that
\(C_{1,1}=0\) because we are dealing with methods of order 2). As \(c\) is increased further, the value of
\(b\) that minimizes \eqref{eq:max} increases towards  \(b=1/4\); the stability interval improves and accuracy
worsens.
 For \(c\geq 2\sqrt 2\) the optimization procedure leads to \(b=1/4\) because, as pointed out earlier, then
  \(b=1/4\) is the only value for which the maximum in \eqref{eq:max} is finite. On the other hand,
  as \(c\) decreases from \(2\), the optimal value of \(b\) decreases, accuracy in the limit  \(h\to 0\)
   improves and the stability interval becomes shorter.
 In the limit \(c\downarrow 0\), \(b\) approaches
\(b \approx 0.1932\), the value that minimizes \(C_{2,1}^2+ C_{2,2}^2\) \cite{Mc1995}.

In the Adaptive Integration Approach (AIA) suggested by \cite{FeAkSa16,AkFeRaSa2017}
 for molecular dynamics problems, the value of \(b\) in \eqref{eq:twostagefamily} is chosen by minimizing
 \eqref{eq:max} with the parameter \(c\) adapted
  to the problem
 being solved, rather than being kept at the somehow arbitrary value \(c=2\). The procedure is the following.
 In molecular dynamics packages such as GROMACS \cite{Heetal2008}, the user
 is prompted to specify her choice of the value of \(\tau\); of course  smaller values of \(\tau\)  lead to
  more expensive simulations. The package estimates
the values of the frequencies of  all harmonic forces present in the problem and will not run if \(\tau
\tilde\omega_{\rm max}\) is close to the Verlet stability limit (\(\tilde\omega_{\rm max}\) is the maximum of
the estimated frequencies).  AIA sets \(c = \sqrt{2}\tilde\omega_{max} \tau\) (\(\sqrt{2}\) is a safety
factor) so that in \eqref{eq:max} \(h\) ranges in the shortest interval that contains all products
\(\sqrt{2}\tau\tilde \omega_i\), where \(\tilde \omega_i\) are the estimated frequencies. (If \(c\geq 4\), the
user is demanded to reduce \(\tau\).) In this way whenever the user is prepared to operate with a value of
\(\tau\) that is \lq\lq small\rq\rq\ for the problem being tackled, AIA chooses a more accurate integrator; as
\(\tau\) increases, AIA will pick up a value of \(b\) leading to a more stable, less accurate integrator. This
strategy has been successfully implemented in an in-house version of GROMACS. The  computational overheads are
negligible as they only stem from finding the value of \(b\) for the problem at hand and the value of \(\tau\)
chosen by the user. Numerical experiments show that AIA is a clear improvement on Verlet.

\vspace{2mm}

\section{HMC in high dimension. Tuning the step size}
\label{sec:highD_hmc} We now follow \cite{BePiRoSaSt2013} and study the behaviour of HMC as the dimensionality
increases. As a byproduct we obtain a general rule for tuning the value of \(h\) when running HMC with a
chosen value of the duration \(\lambda\) of the integration legs. The analysis uses a model situation similar
to that employed in \cite{RoGeGi1997,RoRo1998} to analyze other sampling techniques. High dimensionality in a
different setting is the object of the next section.

\subsection{The set up}

We consider a high-dimensional Hamiltonian system in \((\mathbb{R}^{2d})^m\), \(m\gg 1\), obtained by
juxtaposing without coupling \(m\) copies of a fixed Hamiltonian system in \(\mathbb{R}^{2d}\).  If we write
 a
point $(q,p) \in (\mathbb{R}^{2d})^m$ in components as \[ q= (q_1, \dots, q_m) \in (\mathbb{R}^d)^m \quad
\text{and} \quad p= (p_1, \dots, p_m) \in (\mathbb{R}^d)^m,
\]
then the Hamiltonian of the system in \((\mathbb{R}^{2d})^m\) is
\[ H_m(q,p) = \sum_{j=1}^m H(q_j,p_j),
\]
with
\[
H(q_j,p_j) = \frac{1}{2}  p_j^T M^{-1} p_j + \mathcal{U}(q_j) - \log(Z).
\]
Here $\mathcal{U}: \mathbb{R}^d \to \mathbb{R}$ is the potential energy function for each component and \(Z\)
is the normalizing factor for the one-component Boltzmann-Gibbs distribution, so that \(\exp(-H(q_j,p_j))\) is
a (normalized) probability density function in  \(\mathbb{R}^{2d}\).\footnote{The presence of \(Z\) in the
expression for \(H\) simplifies several formulas below but has no consequence in the HMC algorithm itself (the
constant \(\log(Z)\) does not change the Hamiltonian dynamics and drops from the expression of the acceptance
probability).} In this way \(\exp(-H_m(q,p))\) is the (normalized) probability density function of the
Boltzmann-Gibbs distribution in \((\mathbb{R}^{2d})^m\); clearly the \(2m\) random vectors \(q_j\), \(p_j\)
are stochastically independent for this distribution.

From a sampling viewpoint, the target \(\Pi\) is the distribution in \((\mathbb{R}^d)^m\) with non-normalized
density \( \exp(-\sum_{j=1}^m \mathcal{U}(q_j))\); under this target the \(q_j\) are \emph{independent and
identically distributed.} We sample from \(\Pi\) by using the HMC Algorithm \ref{algo:numerical_hmc} evolving
\((q,p)\) by integrating numerically the dynamics associated with \(H_m\). Since there is no coupling between
components, in Step 2 of the algorithm each pair \((q_j,p_j)\) moves independently of all others. In
particular, the step size stability restriction for the whole system will coincide with the restriction for
the one-component Hamiltonian and therefore will be independent of \(m\). Similarly, in Step 1 randomizing the
momentum \(p\) in \( (\mathbb{R}^{d})^m \) is equivalent to randomizing each \(p_j\). However, in Step 3, the
different components come together: acceptance depends on the error in \(H_m\) and this is obtained by
\emph{adding} the errors in the energies of the individual components. As a consequence of
Theorem~\ref{thm:mean_energy_error}, for a fixed value of \(h\) and a fixed integration interval \(0\leq t\leq
\lambda\), the mean energy error in one component is positive and therefore the mean energy error for \(H_m\)
will grow  linearly with \(m\); as a result, the acceptance rate will decrease as the dimensionality
increases. Therefore, for \(m\) large with \(\lambda\) fixed, the value of \(h\) will have to be decreased to
ensure a satisfactory acceptance rate; note that this restriction on \(h\) is due to accuracy considerations
and has no relation with the stability limit of the integrator being used, which, as we just pointed out, is
independent of \(m\) in our setting.

\subsection{Decreasing \(h\) as the dimensionality increases}
In what follows we fix the duration \(\lambda\) of the integration leg, assume that \(\lambda/h\) is an
integer and study the way in which \(h\) has to be decreased if we wish to ensure that the acceptance rate
remains bounded away from \(0\) as \(m\uparrow \infty\). We denote by \(\Delta(q_j,p_j;h)\) the one component
energy error at \(t=\lambda\) for an integration started from the initial state \((q_j,p_j)\); this initial
state is regarded as a random variable distributed according to the Boltzmann-Gibbs associated with \(H\).
Thus \((q,p)\) will be distributed according to the Boltzmann-Gibbs distribution with density \(\exp(-H_m)\);
in other words the chain is assumed to be at stationarity.
\subsubsection{The mean energy error for one component}
\label{sec:onecomponent}
We
shall   be concerned with the first and second moments
\begin{eqnarray*}
\mu(h) &=& \int_{\mathbb{R}^{2d}} \Delta(q_j,p_j;h) e^{-H(q_j,p_j)} dq_j
dp_j,\\
s^2(h) &=& \int_{\mathbb{R}^{2d}} \Delta(q_j,p_j;h)^2 e^{-H(q_j,p_j)} dq_j
dp_j,
\end{eqnarray*}
and the corresponding variance \(\sigma^2(h) = s^2(h)-\mu(h)^2\).

The following important result rounds out  Theorem~\ref{thm:mean_energy_error}. We now assume that
\(\Delta(q_j,p_j;h)\)  behaves asymptotically as \(h^\nu a(p_j,q_j)\) for a suitable function \(a\) (cf.\
Remark \ref{rem:asymptotic} where global errors rather than energy errors were considered). In this way,
smaller values of \(|a|\) correspond to more accurate integrators and/or \lq\lq easier\rq\rq\  Hamiltonians.
The theorem below ensures that the expectation \(\mu(h)\) of \(\Delta\) is \(\approx (\Sigma/2) h^{2\nu}\),
where we emphasize that the exponent of \(h\) is twice the order of the method and that the proportionality
constant \(\Sigma\) is the average of \(a^2\). The hypotheses in the theorem will hold in practice under
reasonable hypotheses on the integrator and the target distribution; the reader is referred to
\cite{BePiRoSaSt2013} for a complete study in the case of the velocity Verlet scheme.

\begin{theorem}\label{th:onecomponent}Assume that:
\begin{itemize}
\item There exist functions \(a(p_j,q_j)\) and \(r(q_j,p_j;h)\) such that
\[
\Delta(q_j,p_j;h) = h^\nu a(p_j,q_j)+ h^\nu r(q_j,p_j;h),
\]
with \(\lim_{h\rightarrow 0} r(q_j,p_j;h) = 0\) for each \((q_j,p_j)\).
\item There exists a real-valued function \(D(q_j,p_j)\), integrable with respect to the Boltzmann-Gibbs
    distribution, such that, for a suitable \(h_0>0\),
\[
\sup_{0<h<h_0} \frac{\Delta(q_j,p_j;h)^2}{h^{2\nu}}\leq D(q_j,p_j).
\]
\end{itemize}

Then
\[
\lim_{h\rightarrow 0} \frac{\mu(h)}{h^{2\nu}} = \frac{\Sigma}{2},\qquad
\lim_{h\rightarrow 0} \frac{\sigma^2(h)}{h^{2\nu}} = \Sigma,
\]
with
\[
\Sigma = \int_{\mathbb{R}^{2d}} a(q_j,p_j)^2 e^{-H(q_j,p_j)} dq_jdp_j.
\]
\end{theorem}
\begin{proof}
We first establish the second limit. For fixed \((q_j,p_j)\), the first hypothesis implies that
\(\Delta(q_j,p_j;h)/h^{2\nu}\rightarrow a(q_j,p_j)\) and then, by dominated convergence,
\[
\lim_{h\rightarrow 0} \frac{s^2(h)}{h^{2\nu}} = \int_{\mathbb{R}^{2d}} a(q_j,p_j)^2 e^{-H(q_j,p_j)} dq_jdp_j =\Sigma.
\]
In addition, the bound \eqref{eq:mean_energy_error} shows that
\[
\lim_{h\rightarrow 0} \frac{\mu(h)^2}{h^{2\nu}}=0
\]
and the limit for \(\sigma^2(h)/h^{2\nu}\) follows.

From \eqref{eq:aux20jul}, with the shorthand \(\Delta = \Delta(q_j,p_j;h)\),
\begin{eqnarray*}
&&\frac{2\mu(h)-\sigma^2(h)}{h^{2\nu}}- \frac{\mu(h)^2}{h^{2\nu}}\\&&\qquad\qquad =
 - \int_{\mathbb{R}^{2d}} \frac{\Delta}{h^\nu}\,
 \frac{\exp(-\Delta)-1+\Delta}{h^\nu}\,
 e^{-H(q_j,p_j)} dq_jdp_j .
\end{eqnarray*}
We note that the second fraction in the integral approaches \(0\) as \(h\rightarrow 0\) for fixed
\((q_j,p_j)\), because the numerator is \(\mathcal{O}(\Delta^2)\).
 A dominated convergence argument
(see \cite{BePiRoSaSt2013} for details) shows that the integral also approaches 0 and the proof is complete.
\end{proof}

\subsubsection{The mean energy error for \(m\) components}
For the Hamiltonian \(H_m\) the energy error \(\Delta_m(q,p;h)\) is given by the sum
\(\sum_{j=1}^m\Delta(q_j,p_j;h)\). Under our hypotheses, the random variables being added are independent and
identically distributed and therefore
 \(\Delta_m(q,p;h)\) has expectation \(m\mu(h)\approx m \Sigma h^{2\nu}/2\) and
 variance \(m\sigma^2(h) \approx m \Sigma h^{2\nu}\); thus to ensure that \(\Delta_m(q,p;h)\)
 has a distributional limit as \(m\uparrow \infty\), it is reasonable
 to impose a relation
 \begin{equation}\label{eq:hn}
 h = \ell m^{-1/(2\nu)},
 \end{equation}
 with \(\ell>0\) a constant. Under this relation,
 a central limit theorem (see \cite{BePiRoSaSt2013} for details) ensures that,
 as \(m\uparrow \infty\), the distribution of the random
 variable
  \(\Delta_m(q,p;h)\) converges   to the distribution of a random variable
  \(\Delta_\infty \sim N(\ell^{2\nu}\Sigma/2,\ell^{2\nu}\Sigma)\).
  It then follows that,   the expectation of the acceptance probability \( \min\{1,\exp(-\Delta_m)\}\)
  converges to
 \[
 \E \Big(\min \left\{1, e^{-\Delta_\infty}\right\}\Big).
 \]
 The last expectation may be
 found analytically and turns out to be:
 \begin{equation}\label{eq:Phi}
 A(\ell) = 2\Phi(-\ell^\nu \sqrt{\Sigma}/2),
 \end{equation}
where \(\Phi\) is the standard normal cumulative distribution function
\[
\Phi(x) = \frac{1}{\sqrt{2\pi}}\int_{-\infty}^x \exp(- \frac{\xi^2}{2})\,d\xi.
\]

We then conclude:
\begin{theorem}\label{thm:optimaltuning}
Assume that the hypotheses of Theorem~\ref{th:onecomponent} are fulfilled. If  in the scenario above the
number of copies \(m\) and the stepsize \(h\) are related as in \eqref{eq:hn}, then, at stationarity, the
expectation of the acceptance probability of the HMC algorithm converges to \eqref{eq:Phi} as \(m\uparrow
\infty\).
\end{theorem}

For Verlet, with \(\nu=2\), the scaling \eqref{eq:hn} entails that halving \(h\) compensates for a
multiplication of  \(m\) by a factor of \(16\). This scaling in high dimension is very favourable when
compared with the situation for other sampling techniques \cite{RoGeGi1997,RoRo1998}.
\subsection{Optimal tuning}
 In Theorem~\ref{th:onecomponent}, a lower value of
\(\Sigma\), i.e.\ a lower mean value of the function \(a^2\), indicates a more accurate integrator and/or an
\lq\lq easier\rq\rq\ Hamiltonian. Correspondingly, formula \eqref{eq:Phi} is such that lowering the value of
\(\Sigma\) increases the expected acceptance probability \(A\) (\(\Phi(x)\) is an increasing function of
\(x\)). In practice the function \(a\) and the constant \(\Sigma\) are unknown and this would seem to imply
that \eqref{eq:Phi} is of no practical value. We shall show now that, on the contrary, that relation provides
a basis for tuning the parameter \(h\) in the HMC algorithm.

 Increasing  the value of \(\ell\) in \eqref{eq:hn} leads to a larger step size and therefore lowers the computational
 cost of each integration leg over the fixed interval \([0,\lambda]\) but typically increases the energy error. Note in
 this connection that in \eqref{eq:Phi}, \(A\) decreases as \(\ell\) increases. What is the best choice of \(\ell\)? It is argued in \cite{BePiRoSaSt2013} that the function \(E(\ell) =
\ell A(\ell)\) is a sensible indicator of the efficiency of the algorithm, as its reciprocal measures the
amount of work necessary to generate an accepted proposal. With this metric, \(\ell\) should be determined to
maximize \(\ell A(\ell)\): a direct maximization is not feasible because, as we just noted, we cannot compute
\(A(\ell)\). This difficulty may be circumvented by treating \(A\) as an independent variable and expressing
\(\ell\) as a function of \(A\): then
\[
E(A) = \frac{2^{1/\nu}}{\Sigma^{1/(2\nu)}} A \left(\Phi^{-1}(1-A/2)\right)^{1/\nu}.
\]
Clearly the value of \(A\) that maximizes \(E\) is independent of \(\Sigma\) and is therefore
\emph{independent  of the target}. In addition the optimal \(A\) is the same for all integrators of sharing
the same value of \(\nu\). For the case \(\nu=2\), it is found numerically that
\[
A \left(\Phi^{-1}(1-A/2)\right)^{1/2}
\]
is maximized when
\[ A \approx 0.651
\]

This analysis suggests that, for an integrator of order \(\nu =2\) and once the integration interval has been
chosen, if \(m\) is large, \(h\) should be selected in such a way that, in the simulations, the acceptance
rate is observed to be close to \(65\%\). If the observed acceptance rate is larger, one would do better with
a larger value of \(h\). That would imply more rejections, but the waste caused by the rejections would be
offset by the larger number of proposals that may be generated with a fixed budget of force evaluations.
Conversely, acceptance rates significantly below \(65\%\) indicate that one would do better by working more to
generate each single proposal. While our analysis has been performed in the case where the target is a product
of \(m\) independent identical distributions, practical experience shows that this rule works more generally
\cite{BePiRoSaSt2013}. \vspace{2mm}

\section{HMC for path sampling. Sampling from a perturbed Gaussian distribution}
\label{sec:path}

 So far the samples generated by the HMC algorithm have been vectors \(q\in\mathbb{R}^d\). In
some applications the samples needed are \emph{paths}. In this section we discuss the use of
 HMC in those situations. As it will be clear, the material is directly relevant for the problem of sampling
 from  targets that are perturbations of Gaussian distributions.

 \subsection{A model problem}

In order to keep the presentation as simple as possible, we shall initially limit our attention to a model
situation; the general case is discussed at the end of the section.

We consider paths \(u(s)\), where \(u\) is a real-valued function of the variable \(s\in[0,S]\). The paths are
constrained by the homogeneous Dirichlet conditions \(u(0)=u(S) = 0\) and, formally, their distribution is governed by the \lq\lq potential energy\rq\rq\
functional
\begin{equation}\label{eq:pathpotential}
\mathcal{U}(u)  = \int_0^S \left[ \frac{1}{2}  (\partial_s u(s))^2 + g(s,u(s)) \right] ds.
\end{equation}
If the function \(u(s)\) is smooth, then after integration by parts,
\begin{equation}\label{eq:pathpotentialbis}
\mathcal{U}(u) = \int_0^S \left[ -\frac{1}{2} u(s)\partial_{ss}u(s)  + g(s,u(s)) \right] ds.
\end{equation}

Distributions of this form arise when studying \textit{diffusion bridge} problems
\cite{ReVa2005,BeRoStVo2008,HaStVo2009}.

\begin{example}\label{ex:ohbridge}
Fix a time\footnote{In applications, the variable \(s\) typically corresponds to \emph{physical} time, as distinct from the
\emph{artificial} time \(t\) to be used later in the Hamiltonian dynamics that evolves the paths when
obtaining samples.} horizon $S>0$ and consider the process $\mathsf{X}: [0, S] \to \mathbb{R}$ that solves the
Ornstein-Uhlenbeck equation
\[
d \mathsf{X}(s) = -  \mathsf{X}(s) ds +  d \mathsf{B}(s),
\]
conditioned on both initial and final conditions \(\mathsf{X}(0) = 0\),   \(\mathsf{X}(S) = 0\).
Here $ \mathsf{B}$ is a standard Brownian motion.
The law of this diffusion bridge is a probability measure \(\Pi\) on paths \(u\) satisfying the boundary conditions that is
associated to the  functional:
\begin{equation} \label{eq:potential_of_diffusion_bridge}  \frac{1}{2}  \int_0^S \left[ (\partial_s
u(s))^2 +u(s)^2 \right] ds.
\end{equation}
\end{example}
\begin{remark}
The case where the process is conditioned on
\(\mathsf{X}(0) = x^-\), \(\mathsf{X}(S) = x^+\)
may be reduced to the case with homogeneous boundary conditions,
by writing \(u(s) = \bar u(s) +\ell(s)\),
where \(\ell(s)\) is  a smooth function of \(s\) with \(\ell(0) = x^-\), \(\ell(S) = x^+\). After this transformation, the new paths \(\bar
u(s)\) satisfy homogeneous boundary conditions and their distribution corresponds to a functional of the form
\eqref{eq:pathpotential}.
\end{remark}

A precise mathematical description of the meaning of \eqref{eq:pathpotential} and of the associated
probability distribution on paths \(\Pi\) will be given later. For the time being, we note that to study the
infinite-dimensional problem  on a computer it is necessary to introduce a finite-dimensional discretized
version, and we  turn to presenting a way of performing the discretization.

We use  a uniform grid consisting of $d+2$ grid points
 \[
\{ s_j = j \Delta s~ \mid~  j=0, ... , d+1 \},\qquad \Delta s = S/(d+1).
\]
The space of paths is then replaced by the finite-dimensional state space $\mathbb{R}^d$; the \(j\)-th
component \(\boldsymbol{u}_j\) of an element $\boldsymbol{u} \in \mathbb{R}^d$ is seen as an approximation to \(u(s_j)\),
\(j=1, ... , d\). The functional \eqref{eq:pathpotential} is discretized as (cf.\ \eqref{eq:pathpotentialbis})
\begin{equation} \label{eq:truncated_potential_energy}
\mathcal{U}_d(\boldsymbol{u}) = \Delta s  \left( - \frac{1}{2} \boldsymbol{u}^T \boldsymbol{L} \boldsymbol{u} + G_d(\boldsymbol{u}) \right) \;, \quad G_d(\boldsymbol{u}) =   \sum_{j=1}^d g( u_j) \;,
\end{equation}
where the matrix $\boldsymbol{L}$ corresponds the standard central difference approximation to \(\partial_{ss}\) with
homogeneous
Dirichlet boundary conditions:
\begin{equation} \label{eq:discrete_laplacian}
\boldsymbol{L} = \frac{1}{\Delta s^2} \begin{bmatrix}
-2 & 1 &  & \\
1 & \ddots & \ddots &    \\
& \ddots & \ddots &  1 \\
&  & 1 &  -2 \end{bmatrix}.
\end{equation}
Note that, if the vector \(\boldsymbol{u}\) contains the grid values of a smooth path \(u\), then
\(\mathcal{U}_d(\boldsymbol{u})\rightarrow \mathcal{U}(u)\) as \(\Delta s\rightarrow 0\).

Our task is then to sample from the target \(\Pi_d\) in \(\mathbb{R}^d\) with non-normalized density \(
\exp(-\mathcal{U}_d(\boldsymbol{u}))\). Special attention has to be paid to the increase in dimension \(d\)
 as the discretization becomes more accurate.

\subsection{Preconditioned HMC for path sampling}
\label{sec:preconditioned}
We now set up an HMC algorithm, to be called \emph{Preconditioned HMC} (PHMC), to sample from our target
\(\Pi_d\).

\subsubsection{Algorithm description}
 As expected, we introduce an auxiliary variable \(\boldsymbol{p}\in\mathbb{R}^d\) and a
Boltz\-mann-Gibbs distribution \(\Pi_{BGd}\) in \(\mathbb{R}^d\times\mathbb{R}^d\) with non-normalized density
 \(
\exp(-\mathcal{H}_d(\boldsymbol{u},\boldsymbol{p}))\). The Hamiltonian \(\mathcal{H}_d\) is given by \( \Delta
s H_d(\boldsymbol{u},\boldsymbol{p})\), with ($\boldsymbol{M}$ is a mass matrix)
\[
H_d( \boldsymbol{u}, \boldsymbol{p} ) = \frac{1}{2} \boldsymbol{p}^T \boldsymbol{M}^{-1}
\boldsymbol{p}  - \frac{1}{2} \boldsymbol{u}^T \boldsymbol{L} \boldsymbol{u} + G_d(\boldsymbol{u}).
\]
Of course, the target \(\Pi_d\) is
the \(\boldsymbol{u}\)-marginal of  \(\Pi_{BGd}\), while the
\(\boldsymbol{p}\)-marginal is Gaussian  \(N(\boldsymbol{0}, (\Delta s)^{-1} \boldsymbol{M})\).

The pair \(( \boldsymbol{u}, \boldsymbol{p} )\) is evolved by means of the Hamiltonian dynamics asso\-ciated to
\(H_d\):
\begin{equation} \label{eq:semidiscrete_hamiltonian_dynamics}
\begin{bmatrix}
\boldsymbol{\dot u}(t) \\
 \boldsymbol{\dot p}(t)  \end{bmatrix} = \begin{bmatrix} \boldsymbol{M}^{-1} \boldsymbol{p}(t)  \\
  \boldsymbol{L} \boldsymbol{u}(t) - \nabla G_d(  \boldsymbol{u}(t)  ) \end{bmatrix}.
\end{equation}Clearly, this dynamics preserves the value of \(\mathcal{H}_d\) and therefore
 the  distribution \(\Pi_{BGd}\).
\begin{remark}Instead of \eqref{eq:semidiscrete_hamiltonian_dynamics}, one could use the dynamics corresponding
to the Hamiltonian \(\mathcal{H}_d\) that features in  the non-normalized density
\(\exp(-\mathcal{H}_d(\boldsymbol{u},\boldsymbol{p}))\) of \(\Pi_{BGd}\); that alternative dynamics differs from that of \eqref{eq:semidiscrete_hamiltonian_dynamics} by a change in the scale of \(t\). However
\eqref{eq:semidiscrete_hamiltonian_dynamics} is more natural in our context, where there is an
infinite-dimensional problem in the background. We illustrate this as follows. Assume that \(\boldsymbol{M}\)
is taken to be the identity and \(g=0\). Then \eqref{eq:semidiscrete_hamiltonian_dynamics}, after eliminating
\(\boldsymbol{p}\), implies \( \boldsymbol{\ddot u}= \boldsymbol{L}\boldsymbol{u}\), so that \(\boldsymbol{u}(t)\) satisfies the
standard semidiscrete wave equation. This matches the fact that as \(\Delta s \rightarrow 0\),
\(\mathcal{H}_d\) approaches
\[
\int_0^S  \frac{1}{2}  \Big(p(s)^2+(\partial_s u(s))^2 \Big) ds,
\]
which provides the Hamiltonian functional (total energy) for the Hamiltonian partial differential
 equations \(\partial_t u = p\), \(\partial_t
p =
\partial_{ss}u\), that, after elimination of \(p\), yield the wave equation \(\partial_{tt} u = \partial_{ss}u\).
\end{remark}

A challenge with numerically solving \eqref{eq:semidiscrete_hamiltonian_dynamics} using an explicit
 symplectic
integrator (like Verlet) is that the spectral radius of $\boldsymbol{L}$ grows with $d$. Consequently, as the
number of grid points increases, the dynamics may become highly oscillatory due to the presence of fast
frequencies, and stably approximating this type of dynamics may be  difficult \cite{CaSa2009}. For example,
numerical stability of a Verlet integrator applied to \eqref{eq:semidiscrete_hamiltonian_dynamics} with
$G_d=0$ and $\boldsymbol{M}=\boldsymbol{I}$ requires that its time step size $h$ be inversely proportional to
$d$.

To avoid this type of restrictive dependence, we \emph{precondition} the dynamics  by choosing the mass matrix
$\boldsymbol{M} = -\boldsymbol{L}$:
\begin{equation} \label{eq:preconditioned_semidiscrete}
\begin{bmatrix}
\boldsymbol{\dot u}(t) \\
 \boldsymbol{\dot p}(t)  \end{bmatrix} = \begin{bmatrix} -\boldsymbol{L}^{-1} \boldsymbol{p}(t)  \\
  \boldsymbol{L} \boldsymbol{u}(t) - \nabla G_d(  \boldsymbol{u}(t)  ) \end{bmatrix}.
\end{equation}
In the particular case where \(g\) vanishes,  we have \(\boldsymbol{\ddot u}= -\boldsymbol{ u}\); all the
\(d\) frequencies of the preconditioned problem are \(1\) and Verlet has a stepsize restriction \(h<2\),
\emph{independently} of \(d\).

\begin{figure}
\begin{center}
\includegraphics[width=0.33\textwidth]{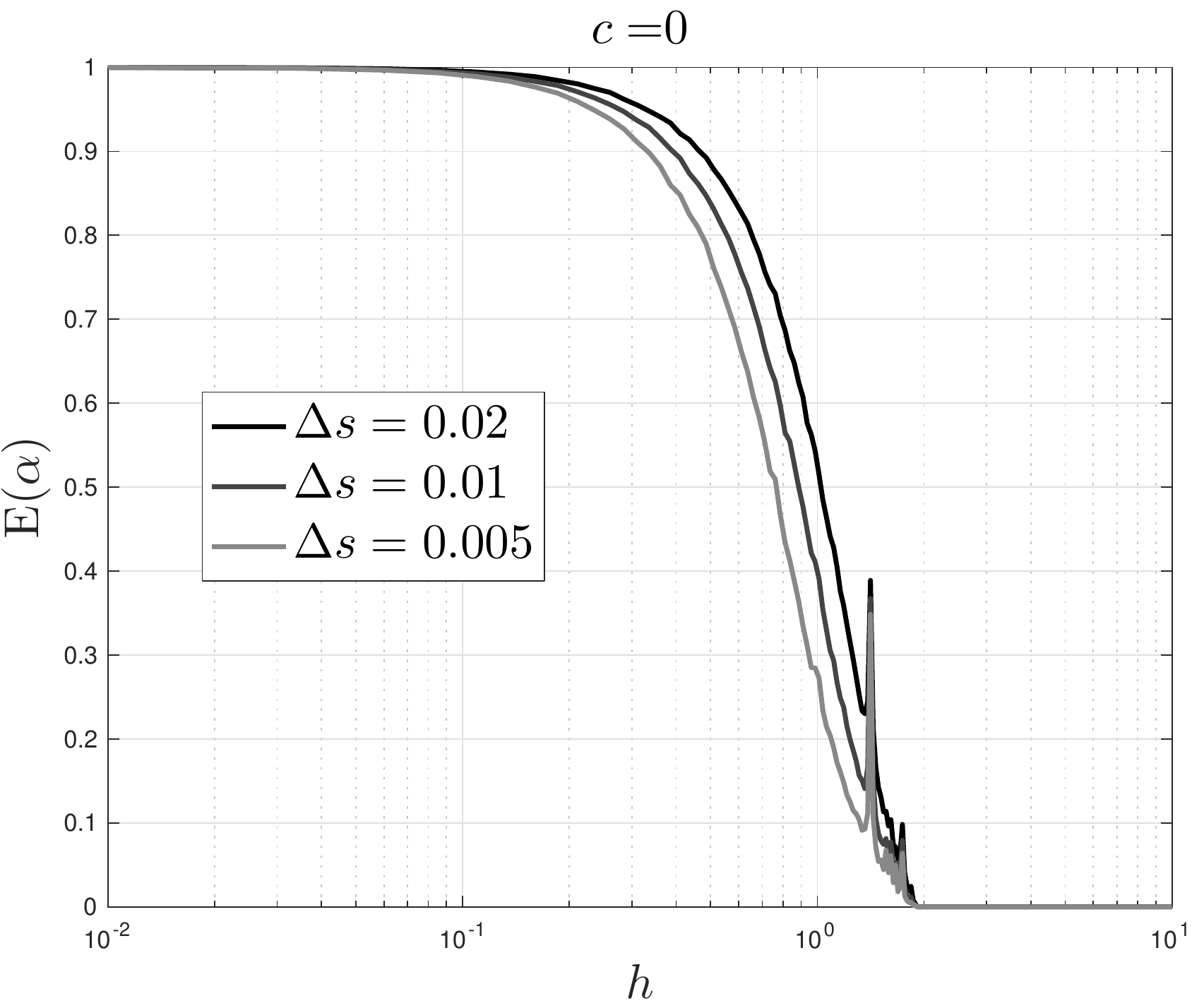}   \hspace{0.1in}
\includegraphics[width=0.33\textwidth]{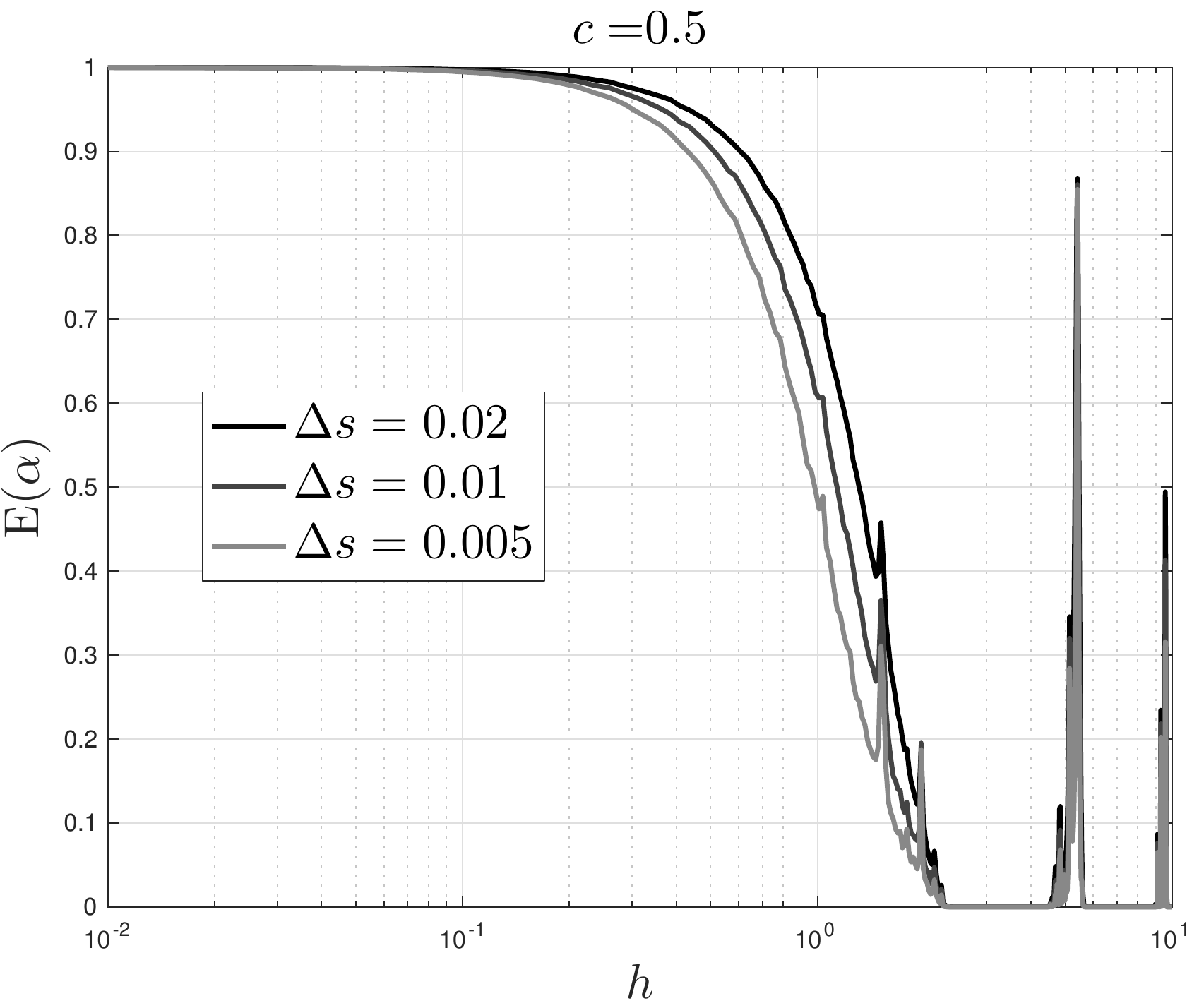} \hspace{0.1in}
\includegraphics[width=0.33\textwidth]{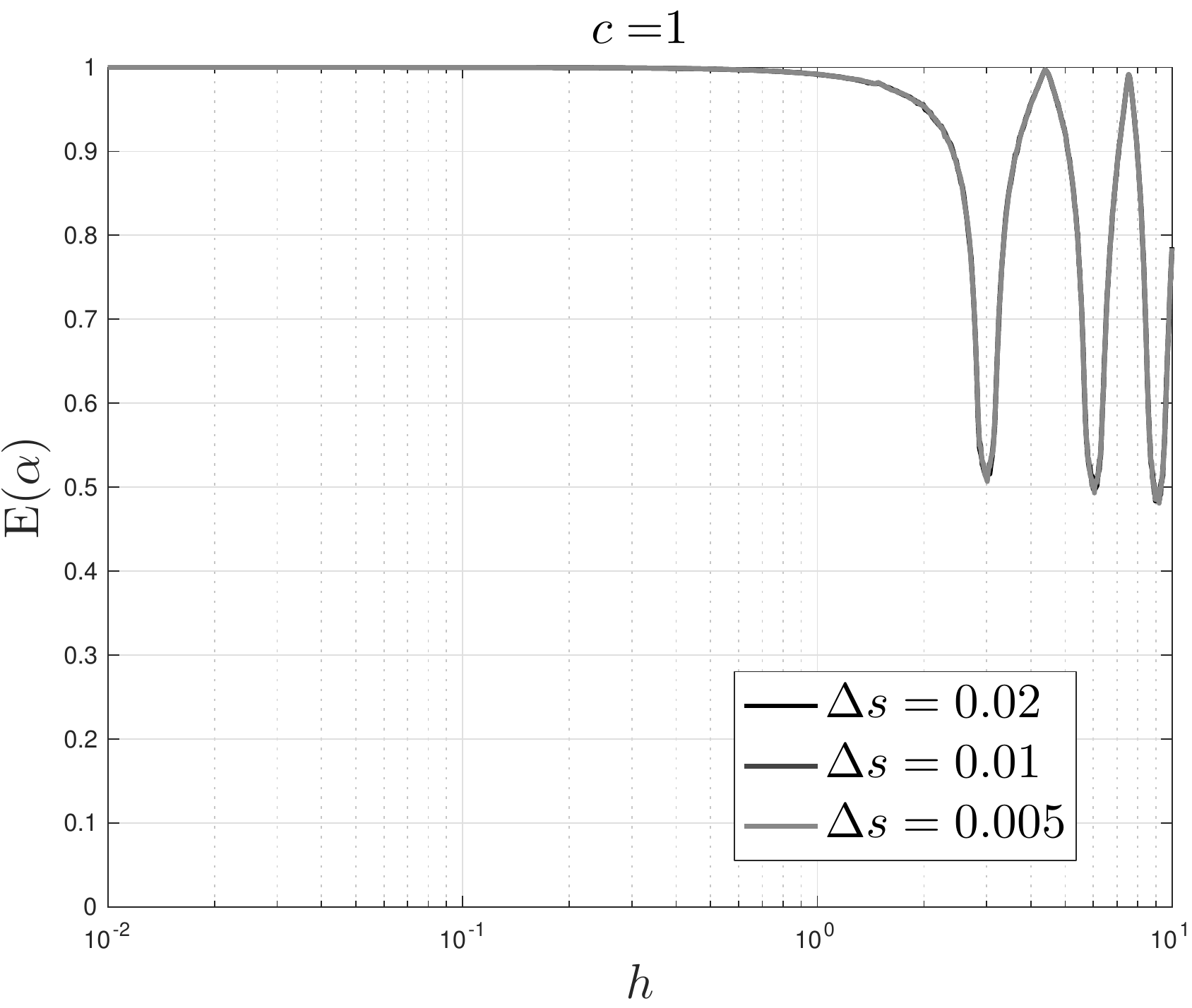}
\end{center}
\caption{\small   Ornstein-Uhlenbeck Bridge Example:
Mean acceptance probability
as a function of $h$ for the PRHMC algorithm
 for three choices of the splitting parameter
$c$
and
the three values of $\Delta s$.  The number of time-steps \(n\) at each integration leg is random with a geometric distribution chosen
in such a way that the average length \(\lambda=nh\) of the integration interval is 20. The number of samples is $10^4$.
For \(c=1\) the acceptance probability does not decrease as the grid is refined.
}
  \label{fig:ou_phmc_mean_ap}
\end{figure}

\begin{remark}\label{rem:massmatrix}Consider for a moment,
a dynamics of the form \eqref{eq:semidiscrete_hamiltonian_dynamics} where \(\boldsymbol{L}\) is an arbitrary
negative-definite matrix and \(\nabla G_d(\boldsymbol{u})\) is small with respect to
\(\boldsymbol{L}\boldsymbol{u}\). From a sampling viewpoint, this situation would arise if standard HMC is
applied
 to sampling from a perturbation of the centered Gaussian distribution with covariance
  matrix \(-\boldsymbol{L}^{-1}\). The preconditioning $\boldsymbol{M} = -\boldsymbol{L}$ makes sense
  in that setting. Large eigenvalues of \(-\boldsymbol{L}\) correspond to directions in state space with
  small variances/large forces; the mass in those direction is then chosen to be large so as to ensure
  small displacements and avoid fast frequencies. This idea may be extended:  general targets may be locally approximated by a state-dependent Gaussian model and one may then choose a state-dependent  \(M\) as the inverse of the covariance matrix of
the Gaussian approximation \cite{GiCa2011}. Unfortunately the state dependence of the mass matrix introduces
additional terms in Hamilton's equations, which are not any longer of the form \eqref{eq:newton2}, with the
unwelcome consequence that \emph{explicit} volume-preserving reversible integrators do not exist.
\end{remark}

In the PHMC algorithm, the system \eqref{eq:preconditioned_semidiscrete} is integrated by means of a
symplectic, reversible scheme with second order of accuracy. We use the Strang splitting \eqref{eq:vv} with
\begin{equation} \tag{A}
\begin{bmatrix}
\boldsymbol{\dot u}(t) \\
 \boldsymbol{\dot p}(t)  \end{bmatrix} = \begin{bmatrix} -\boldsymbol{L}^{-1} \boldsymbol{p}(t)  \\  c^2 \boldsymbol{L} \boldsymbol{u}(t) \end{bmatrix}
\end{equation}
and
\begin{equation} \tag{B}
    \begin{bmatrix}
    \boldsymbol{\dot u}(t) \\
    \boldsymbol{\dot p}(t)  \end{bmatrix}
    = \begin{bmatrix} \boldsymbol{0}  \\  (1-c^2) \boldsymbol{L} \boldsymbol{u}(t)
    - \nabla G_d(  \boldsymbol{u}(t)  ) \end{bmatrix}
\end{equation}
where $c \in [0,1]$ is a parameter. The choice \(c=0\) leads to the velocity Verlet method.

Thus, in PHMC, a transition of the HMC Markov chain starts by drawing a fresh momentum from the marginal
distribution \(N(\boldsymbol{0}, (\Delta s)^{-1} \boldsymbol{M}) = N(\boldsymbol{0}, -(\Delta s)^{-1}
\boldsymbol{L})\), takes  \(m = \lfloor \lambda/h\rfloor\) steps of the splitting integrator and
accepts/rejects with acceptance probability \(\min\{1, e^{-\Delta \mathcal{H}_d}\}\) (see
Algorithm~\ref{algo:numerical_hmc}). A variant, that we call PRMHC, with randomized duration as in
Algorithm~\ref{algo:numerical_rhmc} is clearly possible; in that variant the number of steps is chosen to be
geometrically distributed with mean \(\lambda/h\).

\begin{remark}\label{rem:pv} For implementation purposes, it is advisable to use the variable
\(\boldsymbol{v}=-\boldsymbol{L}^{-1}\boldsymbol{p}\) rather than \(\boldsymbol{p}\). In the variables \(
(\boldsymbol{u},\boldsymbol{v}) \), the split system (A) takes the trivial form \(\boldsymbol{\dot u}(t) =
\boldsymbol{ v}(t)\), \(\boldsymbol{\dot v}(t) = -c^2\boldsymbol{u}(t)\). Before beginning an integration leg,
the fresh value of \(\boldsymbol{v}\) is drawn from the corresponding distribution \(N(\boldsymbol{0}, (\Delta
s)^{-1}\boldsymbol{L}^ {-1})\). When computing the acceptance probability, the Hamiltonian \(\mathcal{H}_d\)
is correspondingly expressed in terms of \((\boldsymbol{u},\boldsymbol{v}) \).
\end{remark}

\subsubsection{Numerical illustration}
\begin{figure}
\begin{center}
\includegraphics[width=0.65\textwidth]{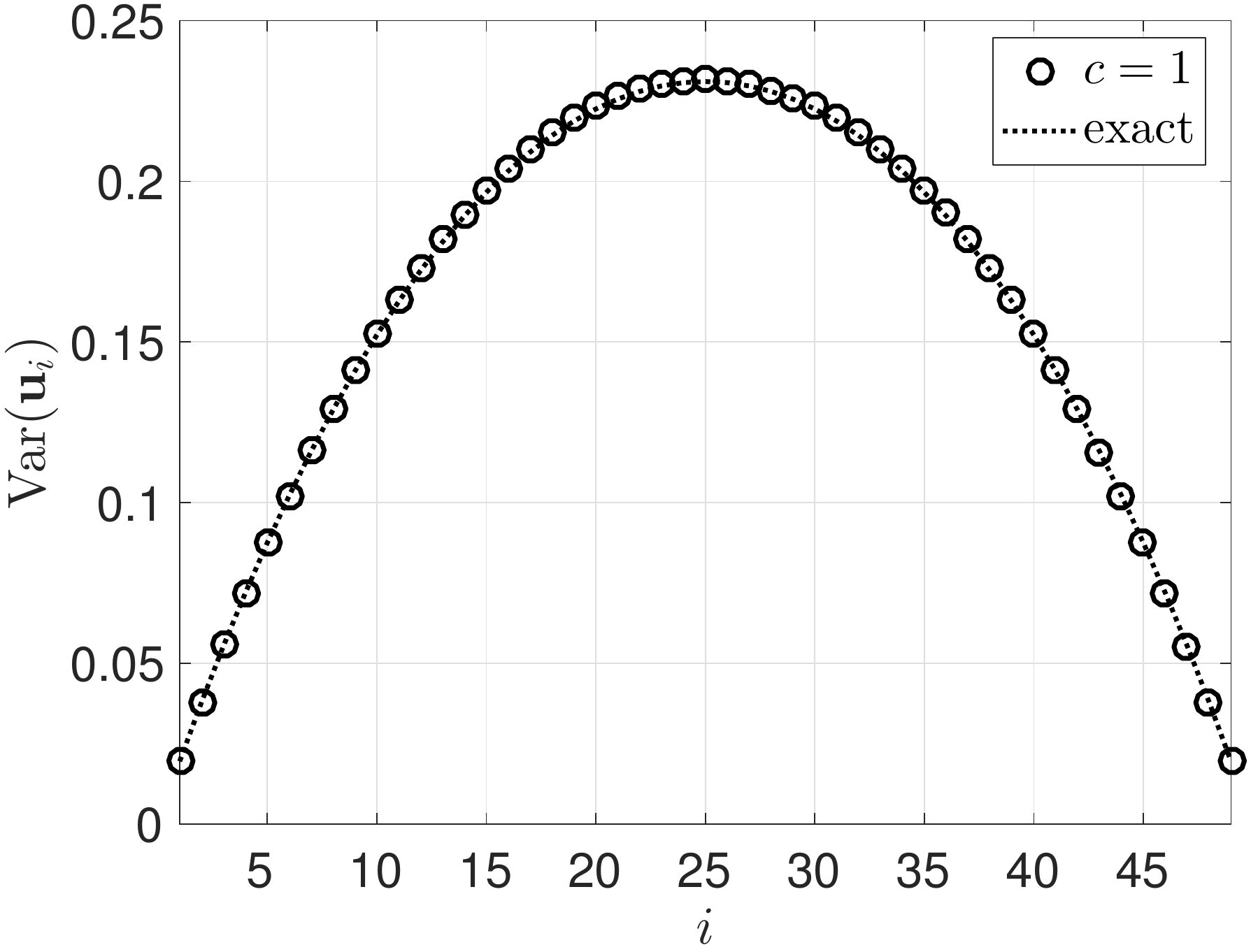}
\end{center}
\caption{\small Ornstein-Uhlenbeck Bridge Example.
This figure assesses the sampling accuracy of the PRHMC algorithm.
The plot graphs, for the distribution \(\Pi_{d}\), the exact and empirical values of the variance of
 $\boldsymbol{u}_i$ when $\Delta s = 0.02$ corresponding to $d=49$ interior grid points
  (the horizontal axis labels the components).   The time step size is $h=2.0$.
  The number of time-steps \(n\) at each integration leg is random with a geometric distribution chosen
in such a way that the average length \(\lambda=nh\) of the integration interval is 20. The number of samples is $10^6$.
  The acceptance rate was $95\%$. When measured in the \(L^2\) norm of \(\mathbb{R}^{49}\),
  the relative error of the vector of empirical variances is \(0.36\%\). The values \(c=0\), \(c=0.5\) were also tested
  but the acceptance rates were virtually zero.
  }
  \label{fig:ou_phmc_errors}
\end{figure}

We have implemented the randomized duration PRHMC algorithm in the particular case of the Orstein-Uhlenbeck
bridge in Example~\ref{ex:ohbridge}, with \(S=1\) and homogeneous Dirichlet boundary conditions (see
Section~\ref{sec:randomizing} for the case of constant-duration PHMC). Note that due to the linearity of the
Orstein-Uhlenbeck process, the target \(\Pi_d\) is Gaussian and there is no need to use MCMC techniques to
sample from it; however using this target provides a convenient test problem. Figure~\ref{fig:ou_phmc_mean_ap}
plots the mean acceptance probability of PRHMC as a function of $h$ for three values of the parameter $c$  in
the integrator and three choices of \(\Delta s\). Clearly \(c=1\) provides the best performance. Note in
particular  that for \(c=1\)  the acceptance probability is independent of \(\Delta s\); for the other values
of \(c\) and fixed \(h\) the acceptance probability decreases with \(\Delta s\).
 Figure~\ref{fig:ou_phmc_errors} shows  that the use of \(c=1\) leads to accurate sampling even with \(h\) is large.  The theory in what follows clarifies these numerical results.

\subsubsection{Analysis}
The results that follow are concerned with the application of PHMC and PRHMC
to target associated with \eqref{eq:pathpotential} in the particular case of the Orstein-Uhlenbeck bridge \eqref{eq:potential_of_diffusion_bridge}. The proofs are given in Section~\ref{sec:proofsbridge}.

In the following stability theorem we require the lowest eigenvalue
\begin{equation}\label{eq:omega1}
\omega_1^2 = \frac{4}{\Delta s^2} \sin^2 \left( \frac{ \pi}{2 (d+1)} \right)
\end{equation}
of \(-\boldsymbol{L}\). It is trivial to show that \(2/S \leq \omega_1\); therefore
the stability requirements  does not become more stringent as \(\Delta s\rightarrow 0\). This is a consequence of the preconditioning of the dynamics.

\begin{theorem}\label{th:stabPHMC}
If $h>0$ satisfies \begin{equation} \label{eq:stability_requirement_ND}
\begin{cases}  c h +2 \arctan\left( \dfrac{h ( 1+(1-c^2) \omega_1^2  )}{2 c \omega_1^2} \right) < \pi & \text{if $c \in (0,1]$}, \\
h< \dfrac{2 \omega_1}{\sqrt{1+\omega_1^2}} & \text{if $c= 0$}, \end{cases}
\end{equation} then the splitting integrator used in the PHMC and PRMHC  algorithms is stable.
\end{theorem}

We now turn to the mean energy error after an integration leg.

\begin{theorem} \label{thm:mean_DeltaN}
Choose  \(\eta\in(0,\pi)\) and restrict the attention to step sizes such that \(ch \leq \eta\) and the
stability requirement \eqref{eq:stability_requirement_ND} is satisfied.
\begin{itemize}
\item For \(c\in[0,1)\), there exist positive constants \(h_0=h_0(c,\eta)\) and \(C=C(c,\eta)\), such that
    for \(h<h_0\), the mean energy error after \(n\) time-steps has the bounds
\[
0 \le \E( \Delta \mathcal{H}_d)  \leq C dh^4.
\]

\item In the  case \(c=1\), there exist positive constants \(h_0=h_0(\eta)\) and
    \(C^\prime=C^\prime(\eta)\), such that for \(h<h_0\),
\[
0 \le \E( \Delta \mathcal{H}_d)  \leq C^\prime h^4.
\]
\end{itemize}

Here, the expected value is over  random initial conditions with non-normalized density
$e^{-\mathcal{H}_d(\boldsymbol{u}, \boldsymbol{p})}$.
\end{theorem}

The exponent of \(h\) in the bounds should not be surprising by now. We emphasize that the bounds are
\emph{independent} of the number of steps  in the integration leg. They therefore apply to the PHMC case where
the duration of the integration leg is fixed  and  to cases where that duration is randomized as in the PRHMC
algorithm in the experiments above.  For \(c\neq 1\) the energy error grows linearly with the number of
degrees of freedom. This is the behaviour we found in the scenario studied in Section~\ref{sec:highD_hmc}, but
we note that here the different components \(\boldsymbol{u}_j\) \emph{are not independent nor are they
identically distributed} (this point is discussed further in Section~\ref{sec:proofsbridge}). On the other
hand, for \(c=1\) the energy error bound is \emph{uniform} in \(\Delta s\).

Even though full details will not be given, we mention that the bounds
 in Theorem~\ref{thm:mean_DeltaN} may be used to derive results on the acceptance rate,
 in analogy to what we did in Section~\ref{sec:highD_hmc}. If \(c=1\), the mean acceptance
  rate for fixed \(h\) is independent of \(\Delta s\). However for \(c\neq 1\), \(h\) has to be scaled as \((\Delta s)^{1/4}\) to ensure that the mean acceptance rate is bounded away from 0 as \(\Delta s\rightarrow 0\). This agrees
  with the experiments in Figure~\ref{fig:ou_phmc_mean_ap}. Note that the scaling \(h\sim (\Delta s)^{1/4}\) arises from \emph{accuracy} considerations rather than from \emph{stability} restrictions.

\subsection{Hilbert space HMC}
The  non-normalized density  \(\propto \exp(-\mathcal{U}_d(\boldsymbol{u}))\) of the target \(\Pi_d\) associated with
\(\mathcal{U}_d(\boldsymbol{u})\) in \eqref{eq:truncated_potential_energy} may be factored as
\begin{equation}\label{eq:factorhilbert}
\exp(-\Delta s G_d(\boldsymbol{u}))\times
 \exp\Big(  - \frac{\Delta s}{2} \boldsymbol{u}^T \boldsymbol{L} \boldsymbol{u}\Big).
\end{equation}
Thus \(\Pi_d\) has non-normalized density
\begin{equation}\label{eq:densitywrtgauss}
\exp(-\Delta s G_d(\boldsymbol{u}))
\end{equation}
with respect to the Gaussian distribution \(\Pi_d^0\) in \(\mathbb{R}^d\) with mean \(\boldsymbol{0}\) and
covariance matrix \((\Delta s)^{-1}\boldsymbol{L}^{-1}\). This observation is useful because as \(d\uparrow
\infty\),  \(\Pi_d^0\) approaches the centered Gaussian distribution \(\Pi^0\) in the Hilbert space
\(L^2(0,S)\) with covariance operator \((-\partial_{ss})^{-1}\), where the differential operator
\(-\partial_{ss}\) has homogeneous boundary conditions \cite{DaZa2014}. It is well known that \(\Pi^0\) is the
distribution of the \emph{Brownian bridge in \([0,S]\)} with homogeneous boundary conditions. On the other
hand the first factor in \eqref{eq:factorhilbert} is a discretization of
\[
\exp\left(-\int_0^S g(s,u(s))ds\right).
\]
In this way the path distribution \(\Pi\) may be described as the measure whose density with respect to the
Gaussian \(\Pi^0\) is given by the last display. There is an important difference between the finite and infinite
dimensional cases. The finite dimensional \(\Pi_d\) may described in two equivalent ways: (i) as having
non-normalized density \(\exp(-\mathcal{U}_d(\boldsymbol{u}))\) with respect to the standard Lebesgue measure \(dq\) or (ii)
as having non-normalized density \eqref{eq:densitywrtgauss} with respect to the Gaussian \(\Pi_d^0\). The
infinite-dimensional \(\Pi\) cannot be defined by a density with respect to the standard Lebesgue
measure in \(L^2\), simply because that measure does not exist. Thus, necessarily \(\Pi\) has to be defined by its density with respect to \(\Pi^0\).

These considerations lead to studying the general problem of  sampling from a target \(\Pi\) defined by a density
\(\exp(-\Phi(u))\)
with respect to a given centered Gaussian reference measure \(\Pi_0\) in a Hilbert space. The PHMC and PRHMC
algorithms described above in the restricted scenario of \eqref{eq:pathpotential} may be extended without difficulty
 to that general problem.  In the extension,
the matrix \(\boldsymbol{L}\) in \eqref{eq:discrete_laplacian} is replaced by a discretization of the inverse of
 the covariance operator of \(\Pi_0\) and the function \(G_d\) arises from discretizing \(\Phi\).
Again the key point in the preconditioned algorithms is to choose the mass matrix \(\boldsymbol{M}\) to coincide with \(-\boldsymbol{L}\); this ensures that if \(\Phi=0\) the dynamics is given by \(\boldsymbol{\ddot u}= -\boldsymbol{ u}\) where all frequencies are \(1\).

The general problem just described has been addressed by \cite{BePiSaSt2011} who introduced an HMC algorithm
that is formulated in \emph{the Hilbert space itself}. Of course in practice that algorithm can only be
implemented after a suitable discretization and, once the discretization has been performed,
 it \emph{coincides with PHMC with} \(c=1\) implemented as in Remark~\ref{rem:pv}.
We emphasize that in PHMC  we first discretized the target and then formulated the sampling method. In contrast,
\cite{BePiSaSt2011} proceed in the reverse order: discretization comes \emph{after} formulating the sampling
method. The route in \cite{BePiSaSt2011}, which requires nontrivial use of functional analytic
techniques, is mathematically more sophisticated than the approach we have followed here. What is then the
advantage of formulating the algorithm in the infinite-dimensional scenario?
 A sampling method that works in the Hilbert space itself may be expected to work uniformly well
 as the dimension \(d\) of the discretization tends to \(\infty\). This is what happens to the
 algorithm in \cite{BePiSaSt2011} in view of Theorem~\ref{thm:mean_DeltaN}. On the other hand,
 for \(c\neq 1\),  PHMC cannot arise from discretizing a Hilbert space algorithm, because we know that
  its performance becomes worse and worse as \(d\uparrow\infty\) with \(h\) fixed.

To finish this section, we remark that for MALA preconditioned and
    non-preconditioned versions are available \cite{BeRoStVo2008},
    and analogous to the latter, there is a non-preconditioned version of HMC \cite{Bo2017}.

\vspace{2mm}

\section{Supplementary material}
\label{sec:supplemenary} The paper concludes with material that complements a number of points considered in
the preceding sections.

\subsection{Finding the modified equation of a splitting integrator}
\label{sec:bch}

We now show how to obtain the expansion in Theorem~\ref{th:modinfsplit}. Full details will be given for the
particular case of the Lie-Trotter splitting algorithm \eqref{eq:lt}; but the method is the same for more
involved integrators. There are two steps. In the first, flows are represented as exponentials. In the second,
exponentials are combined by means of the Baker-Campbell-Hausdorff (BCH) formula.

\subsubsection{Lie derivatives}

Associated with the vector field \(f(x)\) in the differential system \eqref{eq:ode} with flow \(\varphi_t\),
there is a \emph{Lie derivative} \(D_f\). This is the first-order differential operator that maps each smooth
function \(\chi: \mathbb{R}^D\to \mathbb{R}\) into a new function \(D_f\chi: \mathbb{R}^D\to \mathbb{R}\)
defined as follows:
\[
(D_f\chi)(x) = \sum_{i=1}^D f^i(x) \frac{\partial}{\partial x^i}\chi(x)
\]
(superscripts denote components of vectors). The chain rule leads to the formula
\[
\left. \frac{d}{dt} \chi\big(\varphi_t(x)\big)\right|_{t=0} = (D_f\chi)(x),
\]
which shows the meaning of  \((D_f\chi)(x)\) as a rate of change of \(\chi\) along the solution \(t\mapsto
\varphi_t(x)\) of \eqref{eq:ode}. By successively applying this formula to the functions \(D_f\chi\), \(
D_f(D_f\chi) \), \dots , we find
\[
\left. \frac{d^k}{dt^k} \chi\big(\varphi_t(x)\big)\right|_{t=0} = (D_f^k\chi)(x), \qquad k = 2,3,\dots,
\]
where \(D^k_f\) is the  \(k\)-th order differential operator defined inductively as
\[
(D^k_f\chi)(x) = (D_f(D_f^{k-1}\chi))(x),\qquad k = 2, 3, \dots
\]
Therefore the Taylor expansion of \(\chi(\varphi_t(x))\) at \(t=0\) reads
\[
\chi\big(\varphi_t(x)\big) =\sum_{k=0}^\infty \frac{t^k}{k!} (D^k_f\chi)(x),
\]
or
\[ \chi\big(\varphi_t(x)\big) = \big(\exp(tD_f)\chi\big)(x) \]
 a formula that may be used to retrieve, at least formally,  \(\varphi_t\):
its application with \(\chi\) equal to the coordinate function \(\chi(x) = x^i\), \(i = 1,\dots, D\), yields
the \(i\)-th component of the vector \(\varphi_t(x)\). In conclusion, the equality
\begin{equation}\label{eq:flowasexponential}
\chi\circ \varphi_ t = \exp(tD_f) \chi
\end{equation}
may be understood as a representation of the flow of a differential system as the exponential of the Lie
operator of its vector field.

If \(f^{(A)}\) and \(f^{(B)}\) are vector fields, then the compositions  \(D_{f^{(A)}}D_{f^{(B)}}\) and
\(D_{f^{(B)}}D_{f^{(A)}}\) are second-order differential operators. However, it is easily checked that
\(D_{f^{(A)}}D_{f^{(B)}}-D_{f^{(B)}}D_{f^{{(A)}}}\) is a first-order differential operator. In fact the
following result holds \cite[Section 39]{arnold}.

\begin{proposition}\label{prop:commutatorvslie}
\(D_{f^{(A)}}D_{f^{(B)}}-D_{f^{(B)}}D_{f^{{(A)}}}\) is  the Lie operator corresponding to
the vector field \([f^{(A)},f^{(B)}]\)  \ie\ to the Lie bracket of \(f^{(A)}\) and \(f^{(B)}\) defined in
\eqref{eq:liebracket}.
\end{proposition}

If now \(\varphi_t^{(A)}\), \(\varphi_s^{(B)}\) are the flows corresponding to the vector fields \(f^{(A)}\)
and \(f^{(B)}\) respectively, two applications of  \eqref{eq:flowasexponential} give
\begin{align}\label{eq:chifg}
\chi\circ\Big(\varphi_s^{{(B)}}\circ\varphi_t^{(A)}\Big) &=
\Big(\chi\circ\varphi_s^{{(B)}}\Big)
\circ \varphi_t^{(A)}  = \exp(tD_{f^{(A)}}) \Big(\chi\circ\varphi_s^{{(B)}}\Big)\\\nonumber
& = \exp(tD_{f^{(A)}}) \Big(\exp(sD_{f^{(B)}}) \chi\Big)\\
& = \Big(\exp(tD_{f^{(A)}}) \exp(sD_{f^{(B)}})\Big) \chi. \nonumber
\end{align}

Thus the operator \(\exp(tD_{f^{(A)}}) \exp(sD_{f^{(B)}})\) represents the composition
\(\varphi_s^{{(B)}}\circ\varphi_t^{(A)}\) in analogy with  \eqref{eq:flowasexponential}. Note that the
\(A\)-flow acts \emph{first} in the composition \(\varphi_s^{{(B)}}\circ\varphi_t^{(A)}\) while
\(\exp(tD_{f^{(A)}})\) acts \emph{second} in the product of operators \(\exp(tD_{f^{(A)}})
\exp(sD_{f^{(B)}})\). Our task now is to write \(\exp(tD_{f^{(A)}}) \exp(sD_{f^{(B)}})\) as a single
exponential.

\subsubsection{The Baker-Campbell-Hausdorff formula}

Assume for the time being that \(X\) and \(Y\) are square matrices of the same dimension. It is well known
that the product \(\exp(X)\exp(Y)\) of their exponentials only coincides with \(\exp(X+Y)\) if \(X\) and \(Y\)
commute. In fact, by multiplying out
\[
\exp(X) = I + X + \frac{1}{2} X^2+\frac{1}{6} X^3+\cdots
\]
and
\[
\exp(Y) = I + Y + \frac{1}{2} Y^2+\frac{1}{6} Y^3+\cdots
\]
we find
\begin{eqnarray*}
\exp(X)\exp(Y) &=& I+X+Y +\frac{1}{2} X^2+XY+\frac{1}{2} Y^2\\
&&\frac{1}{6} X^3+\frac{1}{2}X^2Y+\frac{1}{2}XY^2+\frac{1}{6} X^3+\cdots
\end{eqnarray*}
The products in the right-hand side are all of the form \(X^kY^{\ell}\), while the expansion of \(\exp(X+Y)\)
gives rise to products like \(YX\), \(Y^2X\), \(YX^2\), \(XYX\), \(YXY\), etc. The BCH formula
\cite{SaCa1994,HaLuWa2010} writes \(\exp(X)\exp(Y)\) as the exponential \(\exp(Z)\) of a matrix
\begin{eqnarray*}
Z& =& X+Y + \frac{1}{2} [X,Y]+\frac{1}{12} [X,[X,Y]]+\frac{1}{12}[Y,[Y,X]]\\
&& +\frac{1}{24} [X,[Y,[Y,X]]] -\frac{1}{720}[Y,[Y,[Y,[Y,X]]]]+\cdots
\end{eqnarray*}
where \([\cdot,\cdot]\) denotes the  commutator, e.g.\ \([X,Y] = XY-YX\), etc. The recipe to write down the
terms in the right-hand side is of no consequence for our purposes; what is remarkable is that this right-hand
side is a combination of \(X\), \(Y\) and \emph{iterated commutators}.

Now the BCH is valid, at least formally (\ie\ disregarding the convergence of the series involved), if \(X\)
and \(Y\), instead of matrices, are elements of any associative, non-commutative algebra. In particular it may
be applied to the case where \(X\) and \(Y\) are first-order differential operators as those considered above.
Going back to \eqref{eq:chifg} and recalling that the commutator of the Lie derivatives corresponds to the Lie
bracket of the vector fields (Proposition~\ref{prop:commutatorvslie}), the BCH formula with \(s=t=h\) then
yields
\[
\varphi_h^{{(B)}}\circ\varphi_h^{(A)} = \exp(h D_{\tilde f_h^\infty}),
\]
where \(f_h^\infty\) is the vector field
\[
f_h^\infty = f^{(A)}+f^{(B)}+\frac{h}{2} [f^{(A)},f^{(B)}] +\frac{h^2}{12}[f^{(A)}, [f^{(A)},f^{(B)}]]+\cdots
\]
Now a comparison with \eqref{eq:flowasexponential} shows that \(\varphi_h^{{(B)}}\circ\varphi_h^{(A)}\) is
formally the \(h\)-flow of \(f_h^\infty\), or, in other words, that \(f_h^\infty\) is the modified
vector-field for the Lie-Trotter splitting algorithm \eqref{eq:lt}.

For Strang's method and for more involved splitting integrators the meth\-odology is the same: each of the
individual flows whose composition yields \(\psi_h\) is written as an exponential and then the exponentials
are combined via the BCH formula. For an integrator where \(\psi_h\) is the composition of \(m\) flows,
\(m-1\) applications of the BCH are required. The task may be demanding due to the combinatorial intricacies
of the BCH formula.

\begin{remark}As noted in Section~\ref{ss:geometricintegration}, finding the modified vector field as above
and then resorting to Theorem~\ref{th:ordercond} is the most common way of investigating the consistency of
splitting algorithms.  An alternative direct technique, not based on modified equations, was suggested by
\cite{MuSa1999}. More recently, \emph{word series} \cite{AlSa2016,MuSa2017} have been introduced as a simpler
means to deal with this kind of question. A survey of the combinatorial techniques used to analyze integrators
is provided by \cite{SaMu2015}.
\end{remark}

\subsection{Distributions with light tails}
\label{sec:qcubed}

\begin{figure}
\begin{center}
\includegraphics[width=0.45\textwidth]{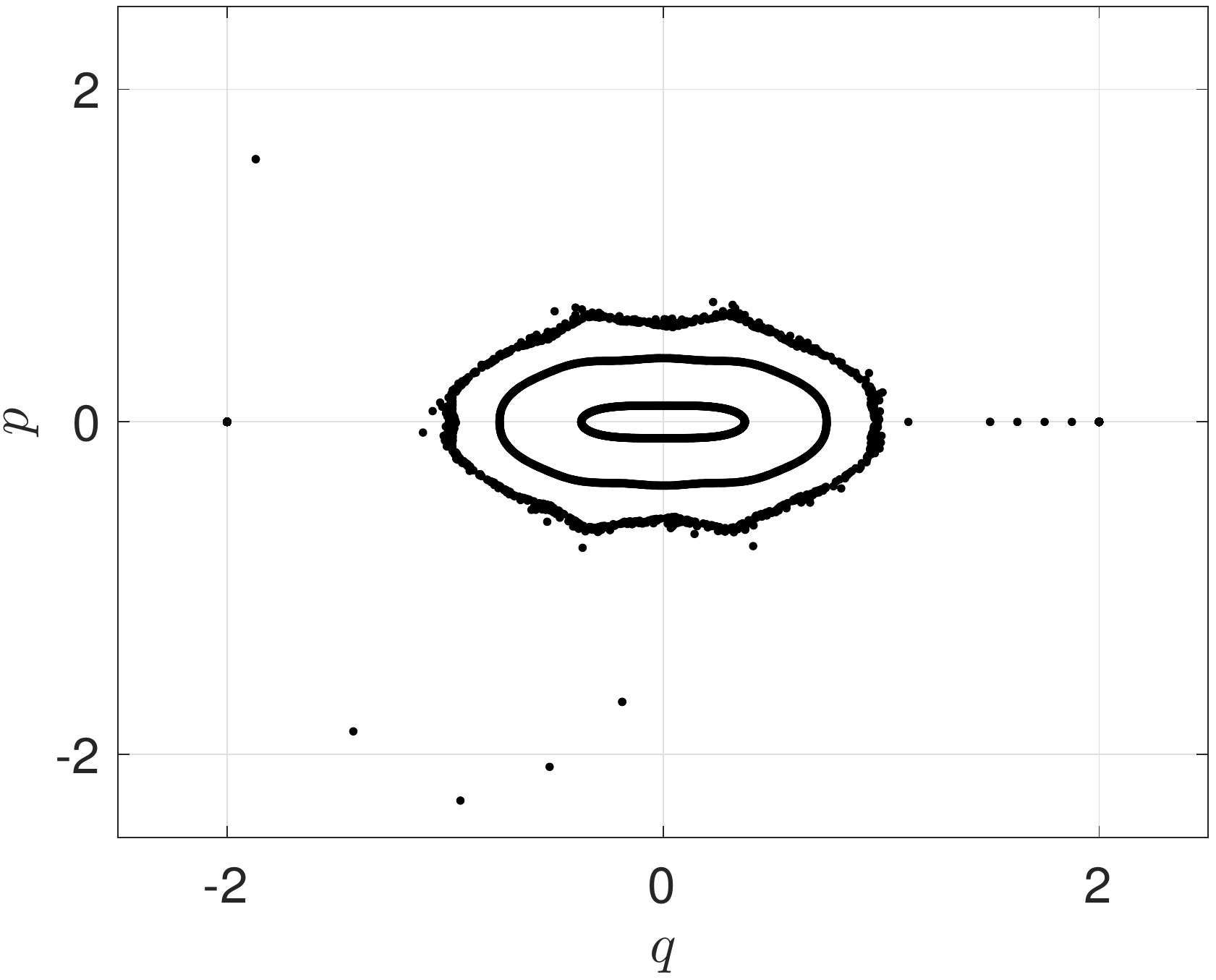}   \hspace{0.1in}
\includegraphics[width=0.45\textwidth]{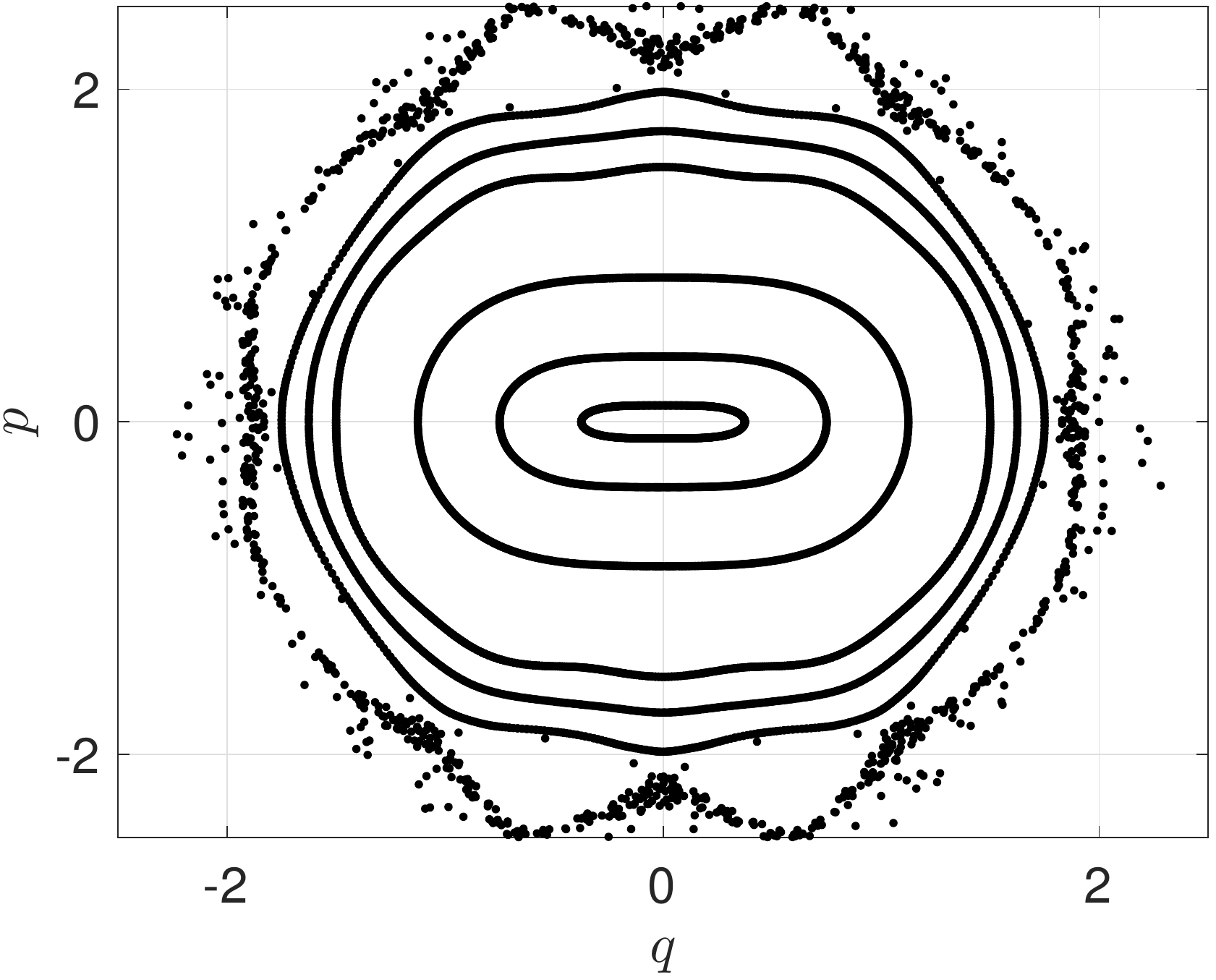}
\end{center}
\caption{\small  Verlet scheme for time step $h = 1$
(left panel) and $h=1/2$ (right panel) applied to $H(q,p) = (1/2) p^2 + (1/4) q^4$.
 The numerical solutions mimic the behaviour of the exact solution
 \emph{only} for initial conditions close to the origin. Decreasing \(h\) enlarges the region where the
 integrator performs satisfactorily, but does not eliminate the problem.
}
 \label{fig:conditional_stability}
\end{figure}

The numerical integration of Hamiltonian systems with a quadratic Hamiltonian \eqref{eq:MKmodel} (which
corresponds to HMC sampling of Gaussian distributions) was discussed in Section~\ref{sec:integrators}. For
distributions whose tails are lighter than those of a Gaussian, the potential energy \(\mathcal{U}(q)\) grows
 faster than quadratially as \(|q|\rightarrow \infty\) and this causes difficulties when integrating
numerically the Hamiltonian dynamics.

We illustrate this for the univariate distribution with non-normalized density \(\exp(-q^3)\).
Figure~\ref{fig:conditional_stability} shows, for two values of \(h\), nine velocity Verlet trajectories
corresponding to initial conditions with \(p=0\) and \(q\in(0,2]\). While, for smaller values of the initial
\(q\) the numerical trajectory mimics the   behaviour of the true solution, large initial values of \(q\) lead
to trajectories that quickly escape to infinity. Reducing the value of \(h\) enlarges the size of the domain
where Verlet performs satisfactorily. However no matter how small \(h\) is chosen, there will be an outer
region where the performance is bad. This is easily understood. The kicks of the Verlet algorithm update \(p\)
by using the formula \(p\mapsto p-(h/2)q^3\). For this update to be a reasonable approximation to the true
dynamics, it is obvious that the magnitude of \((h/2)q^3\) should be small when compared with the magnitude of
\(p\); a requirement that does not hold for any fixed \(h\) and \(p\) if \(|q|\) is sufficiently large. The
problem is the same for other explicit integrators. For implicit integrators, such as the midpoint rule, the
solution of the algebraic equations to be solved at each step also demands smaller values of \(h\) for larger
values of \(q\) as discussed in \cite[Section 3.3.3]{SaCa1994}.

The conclusion is that, when using an integrator with fixed \(h\), one's  attention has to be restricted to a
bounded subset \(D\) of the phase space \(\mathbb{R}^{2d}\) chosen to guarantee that the complement
\(\mathbb{R}^{2d}\backslash D\) has negligible probability with respect to \(\PiBG\). After fixing \(D\), a
suitably small value of \(h\) has to be chosen. While the situation is well understood for the sampling
algorithm MALA \cite{BoHa2013}, the corresponding analysis for HMC has not yet been carried out.

\subsection{Randomizing the time step and the duration}
\label{sec:randomizing}

\begin{figure}
\begin{center}
\includegraphics[width=0.33\textwidth]{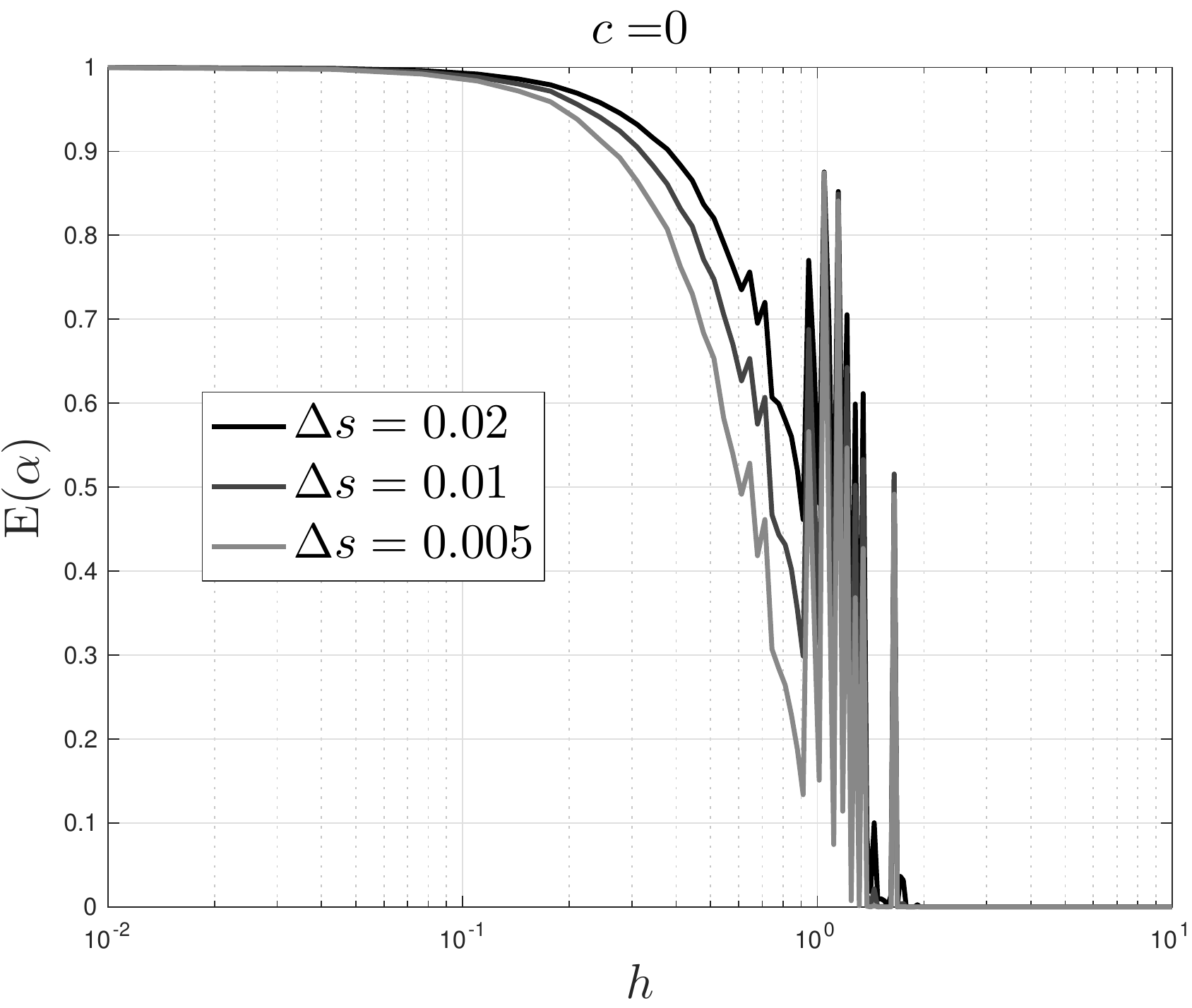}   \hspace{0.1in}
\includegraphics[width=0.33\textwidth]{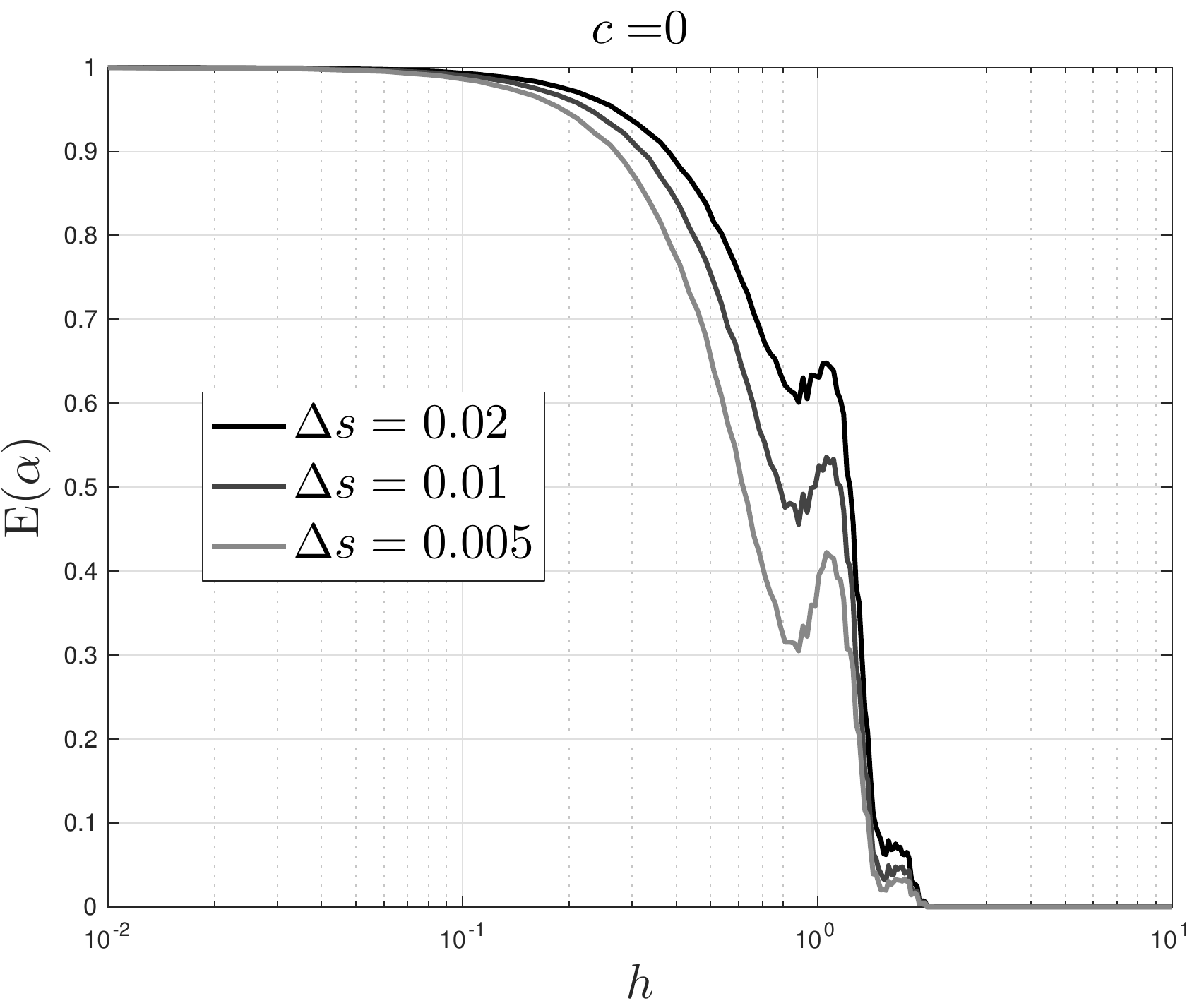}  \hspace{0.1in}
\includegraphics[width=0.33\textwidth]{prhmc_mean_ap_c_0.pdf}
\end{center}
\caption{\small  {\bf Time-Step or Duration Randomization.} Ornstein-Uhlenbeck Bridge Example.
Mean acceptance probability as a function of $h$ for PHMC operated with constant step sizes
(first panel), PHMC operated with randomized step sizes of average \(h\) (second panel),
and PRHMC (third panel)  for $c=0$ and three values of $\Delta s$. In all three cases the duration parameter
 is $\lambda = 20$. The number of samples is $10^4$.   The PRHMC algorithm is the least susceptible to periodicities
 in the underlying Hamiltonian dynamics.
}
  \label{fig:randomizations}
\end{figure}

The Exact RHMC algorithm, where the length of the integration leg is a random variable with exponential
distribution was discussed at the end of Section~\ref{sec:montecarlomethods}, along with its numerical
counterpart. We also discussed there an alternative procedure, based  on randomizing the step length. We now
illustrate the performance of these randomized algorithms in the case of the Ornstein-Uhlenbeck bridge example
considered in Section~\ref{sec:path}. With the notation employed there, we set \(c=0\), which leads to the
standard Verlet integrator. We have implemented three algorithms: (i) PHMC operated with (constant) step sizes
of length \(h\), (ii) PHMC operated with step sizes of length \(\Delta t\) uniformly distributed in the
interval \([0.9h, 1.1h]\) and (iii) PRHMC. The results are given in Figure~\ref{fig:randomizations}. (The
third panel here  is the same as the first panel in Figure~\ref{fig:ou_phmc_mean_ap}.) As \(h\) varies, the
behaviour of the un-randomized algorithm (i) is very irregular due to resonances between the step size and the
periods of the dynamics (see Section~\ref{ss:markov}). Randomizing \(h\) regularizes the behaviour, but not as
successfully as PRMHC does. This is consistent with the study carried out in \S4 and \S5 of \cite{BoSa2016}.

The results of experiments with \(c=0.5\) and \(c=1\), not reported here, show the same patterns.

\subsection{PHMC for the Orstein-Uhlenbeck bridge: proofs}
\label{sec:proofsbridge}
We now prove Theorems ~\ref{th:stabPHMC} and \ref{thm:mean_DeltaN}.

\subsubsection{One degree of freedom}

We begin by studying the application of the integrator used in PHMC to  a one-degree-of-freedom,
 linear analog of the dynamics in \eqref{eq:preconditioned_semidiscrete} given by the Hamiltonian function
\[
H(q,p) = \frac{1}{2} \omega^{-2} p^2 + \frac{1}{2} (\omega^2 + 1) q^2.
\]
If \(\omega \gg 1\) this corresponds to the motion of a particle with large mass \(\omega^2\) attached to a
stiff spring of constant \(1+\omega^2\). The solutions are periodic with a frequency \(1+1/\omega^2\) that
decreases as \(\omega\) increases.

 The one-step propagation matrix (cf.\ \eqref{eq:harmonicintegrator}) is
 given explicitly by \[
\tilde M_{h,c} = \begin{bmatrix} \cos(c h) - \dfrac{h \tilde \omega^2}{2 c \omega^2} \sin(c h) & \dfrac{\sin(c h)}{c \omega^2} \\[6pt]
 - h \tilde \omega^2 \cos(c h) - \dfrac{4 c^2 \omega^4 - h^2 \tilde \omega^4}{4 c \omega^2} \sin(c h)  & \cos(c h) - \dfrac{h \tilde \omega^2}{2 c \omega^2} \sin(c h) \end{bmatrix},
\] with  $\tilde \omega^2 = 1+(1-c^2) \omega^2$,  if $c \in (0,1]$, or
\[
\tilde M_{h,c} = \begin{bmatrix} 1 - \dfrac{h^2}{2 \omega^2} (1+\omega^2)  & \dfrac{h}{\omega^2} \\[6pt]
\dfrac{h (1+\omega^2) (h^2 - (4-h^2) \omega^2)}{4 \omega^2} &  1 - \dfrac{h^2}{2 \omega^2} (1+\omega^2),
\end{bmatrix}
\]  if $c=0$.

The following result is concerned with the stability of the integrator.
Note that the stability restriction on \(h\) weakens as \(\omega\) gets larger.

\begin{proposition} \label{prop:Psic_Stability_1D}
If $\omega >0$,  $c \in [0,1]$, and  $h>0$ satisfies \begin{equation} \label{eq:stability_requirement_1D}
\begin{cases}  c h +2 \arctan\left( \dfrac{h ( 1+(1-c^2) \omega^2  )}{2 c \omega^2} \right) < \pi & \text{if $c \in (0,1]$} \;, \\
h< \dfrac{2 \omega}{\sqrt{1+\omega^2}} & \text{if $c= 0$}\;, \end{cases}
\end{equation} then the splitting integrator is stable.
\end{proposition}

\begin{proof} Since the integrator is symplectic \(\det(\tilde M_{h,c}) = 1\) and stability is ensured if the magnitude  of the
trace is \(<2\). The case \(c=0\) is trivial. For the case $c \in (0,1]$, by means of elementary trigonometric
formulas we find
 \[
\frac{1}{2} \tr( \tilde M_{h,c} )  = \dfrac{\cos(c h + \arctan(\dfrac{h \tilde \omega^2}{2 c \omega^2} ) )}{\cos(\arctan(\dfrac{h \tilde \omega^2}{2 c \omega^2} ) )}.
\]  As \(h\) increases from 0, stability is lost when \(h\) is such that the sum of the arguments in the cosine functions reaches  \(\pi\), so that the numerator and the denominator are opposite.
\end{proof}

As in Section~\ref{ss:hmp}, for stable integrations, \(\tilde M_{h,c}\) may be written in the form \eqref{eq:tildemh}, with
$\cos(\theta_h) = \tr(\tilde M_{h,c})/2$ and
\[
\begin{aligned}
 \chi_h =  \begin{cases} \dfrac{\sin(ch)}{c \omega^2 \sin(\theta_h) } & \text{if $c \in (0,1]$}. \\
 \dfrac{h}{\omega^2 \sin(\theta_h)} & \text{if $c=0$}. \end{cases}
 \end{aligned}
\]

For the error \(\Delta\) in \(H\) after an integration leg, by proceeding
as in the proof of Proposition~\ref{prop:average}, we may show:
\begin{proposition} \label{prop:Psic_Mean_Energy_Error_1D}
For  any positive integer \(n\), $\omega >0$,  $c \in [0,1]$,  \(\beta>0\) and  $h>0$ satisfying the stability
requirement \eqref{eq:stability_requirement_1D}, the energy error after \(n\) steps satisfies  \[ \E ( \Delta )
= \beta^{-1}\sin^2(n \theta_h)\rho(c,\omega,h),
\]
where the expected value is over random initial conditions with non-normalized density $e^{-\beta H(q,p)}$ and
\[
\rho(c,\omega,h) = \frac{1}{2} \left( \widehat \chi_h^2+\frac{1}{\widehat \chi_h^2}-2\right)
= \frac{1}{2} \left(\widehat \chi_h   -\frac{1}{\widehat \chi_h}\right)^2 \geq 0
\]
with
 \[
\widehat \chi_h^2 = (\omega^2 + \omega^4) \chi_h^2 .
\]
Accordingly
\[
0\leq \E ( \Delta )\leq \rho(c,\omega,h).
\]
\end{proposition}
\medskip

Our next result gives a representation for  \(\rho\). The proof is a tedious exercise in algebra and will not be given.

\begin{proposition} \label{prop:chi2}
For $\omega >0$,  $c \in [0,1]$,  and  $h>0$  satisfying the stability requirement
\eqref{eq:stability_requirement_1D}, we have
\begin{equation}\label{eq:rhosupp}
\rho(c,\omega,h) = \frac{1}{2}\frac{h^4 r^2}{(1+\omega^{-2})(1+\omega^{-2}-h^2 r)},
\end{equation}
where \(r\) is the function
\[
r(c,\omega,h) = \frac{1}{4}(1-c^2+\omega^{-2})^2+c^2(1-c^2+\omega^{-2}) R(ch),
\]
with
\[
R(ch) = \frac{1-ch \cot(ch)}{(ch)^2}.
\]
\end{proposition}
\medskip

 The function
\(R(\zeta)\) has the  finite limit \(1/3\) as \(\zeta\rightarrow 0\). If \(c>0\), \(R(ch)\) is positive and
increases monotonically in the interval \(0<ch<\pi\), with \(R(ch)\uparrow \infty\) as \(ch\uparrow \pi\). As
a consequence, \(r\) is an increasing function of \(h\in(0,\pi)\); the term \(1+\omega^{-2}-h^2r\)  in
\eqref{eq:rhosupp} approaches \(0\) as \(h\) approaches the upper limit of the stability interval. The
preceding propositions imply that the mean energy error has an \(\mathcal{O}(h^4)\) bound. By now this result
is hardly a surprise, because we are dealing with an integrator of order \(\nu= 2\).

\subsubsection{\(d\) degrees of freedom}

We now address the application of PHMC to the target associated with \eqref{eq:pathpotential} when \(g(u) = (1/2)u^2\).
Theorems~\ref{th:stabPHMC} and \ref{thm:mean_DeltaN} are proved by using
in \(\mathbb{R}^d\) an orthonormal basis of eigenvectors
of \(-\boldsymbol{L}\) (cf.\ the proof of \eqref{eq:multivariate}).
In the basis of eigenvectors the different components are stochastically
independent but not identically distributed.

For Theorem~\ref{th:stabPHMC} it is enough to recall that in Proposition~\ref{prop:Psic_Stability_1D} larger
values of \(\omega\) lead to less stringent stability requirements. Therefore stability is limited by the
smallest eigenvalue of \(-\boldsymbol{L}\) given in \eqref{eq:omega1}.

Turning to the proof of  Theorem~\ref{thm:mean_DeltaN}, by using the spectral decomposition and
Proposition~\ref{prop:Psic_Mean_Energy_Error_1D} with \(\beta =\Delta s\), we may write
\[
0\leq \E(\Delta\mathcal{H}_d) \leq \sum_{j=1}^d \rho(c,\omega_j,h),
\]
with \(\omega_j^2\) equal to the \(j\)-th eigenvalue of \(-\boldsymbol{L}\)
\[\omega_j^2 = \frac{4}{\Delta s^2}
\sin^2 \left( \frac{ j\pi}{2 (d+1)} \right)\geq \frac{4j^2}{S^2}.\] We now use
 \eqref{eq:rhosupp} to obtain the bound
\[
0\leq \E(\Delta\mathcal{H}_d) \leq \frac{h^4}{2}
\sum_{j=1}^d \frac{r(c,\omega_j,h)^2}{1-h^2r(c,\omega_j,h)},
\]
and, since \(r\) is a decreasing function of \(\omega\),
\[
0\leq \E(\Delta\mathcal{H}_d) \leq d\frac{h^4}{2}
\frac{r(c,\omega_1,h)^2}{1-h^2r(c,\omega_1,h)}.
\]
The first bound in the theorem now follows easily: for \(h\) bounded away from \(\pi\), \(r\) is bounded,
and then \(h\) may be reduced further to ensure that \(1-h^2r(c,\omega_1,h)\) is bounded away from 0.

When \(c=1\),
\[
r(1,\omega_j,h)= \frac{1}{4}\omega_j^{-4}+\omega_j^{-2}R(h)\leq \frac{S^4}{64j^4}+\frac{S^2}{4j^2}R(h),
\]
and therefore for \(h\) bounded away from \(\pi\), \(r(1,\omega_j,h)=\mathcal{O}(j^{-2})\). Thus, if we write,
\[
0\leq \E(\Delta\mathcal{H}_d) \leq \frac{h^4}{2}
\sum_{j=1}^d \frac{r(1,\omega_j,h)^2}{1-h^2r(1,\omega_1,h)},
\]
the sum in the right-hand side may be bounded  independently of \(d\).

\subsection{Convergence}
\label{sec:convergence}

Here is a brief discussion of the sampling properties of HMC.

 Intuitively speaking, the point of using HMC is that its Hamiltonian dynamics breaks
the random walk behaviour that impairs the convergence of simple MCMC methods like the random walk Metropolis
or the MALA algorithm \cite{Gu1998,DiHoNe2000}.  This slow convergence can be understood by using a diffusion
approximation of these simple MCMC methods.  In particular,  the processes associated to simple MCMC methods
are often well approximated (in a weak sense) by overdamped Langevin dynamics \cite[see Theorem 5.1 for a
precise mathematical statement]{BoDoVa2014}, also called Brownian dynamics.  This may be seen as a consequence
of the detailed balance property of the algorithms. As a result, the corresponding paths meander around state
space and explore the target distribution mainly by diffusion. Qualitatively speaking, this is random walk
behavior. The basic idea in HMC is to replace the first-order Langevin dynamics in order to avoid this random
walk behaviour.

Unfortunately, beyond this intuition, the improvement in convergence afforded by HMC  is not well understood.
Under certain conditions on the time-step and target distribution, it has been known for a while that
numerical HMC converges to the right distribution asymptotically \cite{Sc1999,St2007,CaLeSt2007}. More
recently, the piecewise deterministic process associated with the exact RHMC Algorithm~\ref{algo:exact_rhmc}
was shown to be geometrically ergodic; i.e., its transition probabilities converge to the target distribution
in the total variation metric at an exponential rate \cite[see Theorem 3.9]{BoSa2016}.  The proof of geometric
ergodicity of RHMC relies on Harris theorem, which requires a (local) version of Doeblin's condition: a
minorization condition for the transition probabilities  at a finite time and in a compact set. Unfortunately,
given the complicated form of these transition probabilities, the minorization condition involves non-explicit
constants, and in particular, the dependence of the resulting convergence rate on the mean duration parameter
is not clear.

Using a coupling approach, a contraction rate for exact RHMC in a Wasserstein distance has been derived
\cite{BoEb2017}.  The dependence of the contraction rate on the mean duration parameter is explicit.  It
follows from this rate that, if the mean duration parameter is proportional to the standard deviation of the
target distribution, then the algorithm converges at a kinetic speed (as opposed to a diffusive speed).  The
main tools used in the proof are a Markovian coupling and a coupling distance that are both tailored to the
structure of exact RHMC.    The coupling is based on the framework introduced in \cite{Eb2016A}, and is
related to a recently developed coupling for second-order Langevin dynamics \cite{Eb2016B,EbGuZi2016}.

\subsection{Related Algorithms}

We finally discuss generalizations and variants of the basic HMC Algorithm~\ref{algo:numerical_hmc}. The
corresponding literature is extensive and rapidly growing;  we limit ourselves to a few of the many available
techniques.

% GHMC

\subsubsection{Generalized HMC}

Generalized HMC (GHMC) allows the possibility of partial randomization of momentum at every step of the HMC
algorithm \cite{Ho1991,KePe2001}.  Specifically, in GHMC, one introduces a  so-called Horowitz angle $\phi \in
(0, \pi/2]$  and replaces the randomized momentum $\xi_0$ in Step 1 of Algorithm~\ref{algo:numerical_hmc} with
$\cos(\phi) p_0 + \sin(\phi) \xi_0$.   This partial momentum randomization preserves the $p$-marginal of
$\PiBG$, and hence, GHMC leaves $\PiBG$ invariant.   When $\phi=\pi/2$, the initial momentum $p_0$ is fully
randomized and we recover HMC.  If $\phi \in (0, \pi/2)$, then $p_0$ is only partially randomized; this may be
of interest in molecular dynamics applications because then the samples may give more information on the
Hamiltonian dynamics than when each integration leg starts with a completely new momentum. The case $\phi=0$
is excluded, since then there is no momentum randomization, and consequently, the GHMC chain is confined to a
single level set of $H$, and in general, is not ergodic with respect to $\PiBG$. Since the momentum
randomizations are the source of dissipation in HMC \cite{BoSa2016}, partial momentum randomization makes the
algorithm less dissipative and less diffusive, as discussed in \S 5.5.3 of \cite{Ne2011}. One can similarly
introduce a partial momentum randomization to the randomized HMC Algorithm~\ref{algo:numerical_rhmc}.

% XHMC

\subsubsection{Extra Chance HMC}

As we pointed out before, rejections increase correlations along the HMC chain and waste computational effort.
Extra chance HMC (XHMC) \cite{CaSa2015} provides a simple way for HMC to get extra chances to obtain
acceptance, and hence, delay or even avert rejection as in \cite{Mi2001,GrMi2001}.  To describe the XHMC
algorithm, it suffices to describe its accept/reject mechanism, which we state in terms of a symplectic and
reversible integrator $\psi_h$, and its $n$-step map, $\Psi_{\lambda} = \psi_h^{n}$ where $n = \lfloor \lambda
/ h \rfloor$.   Given a user-prescribed number of extra chances $K$ and $(q,p) \in \mathbb{R}^{2 d}$, define
$\Gamma^0(q,p) = 0$, $\Gamma^{K+1}(q,p) = 1$, and \[ \Gamma^j(q,p) = \alpha(q,p) \vee \cdots \vee
\alpha^j(q,p) \;,  \qquad 1 \le j \le K \;,
\]
where $\alpha^j(q,p) = \min\left\{ 1,  e^{-\left(H(\Psi_{\lambda}^j(q,p)) - H(q,p) \right)} \right\}$.  Note
that $\Gamma^j$ is monotonically increasing with $j$. Instead of flipping a coin at every step as in
Algorithm~\ref{algo:numerical_hmc}, XHMC rolls a $(K+1)$-sided die with probabilities $\{ \Gamma^{j} -
\Gamma^{j-1} \}_{j=1}^{K+1}$.   Given the current state of the chain $(  q_0,  p_0 ) \in \mathbb{R}^{2d}$, the
transition in Step 3 of Algorithm~\ref{algo:numerical_hmc} is accordingly replaced by \[
(q_1, p_1) = \begin{cases}  \Psi_{\lambda}^j(q_0,\xi_0) & \text{if  $j$-th side of die comes up} \\
(q_0, -\xi_0) &  \text{if  $(K+1)$-th side of die comes up} \end{cases}
\] where $\xi_0 \sim \mathcal{N}(0,M)$. It is straightforward to show that XHMC leaves $\PiBG$
invariant \cite[Appendix A]{CaSa2015}.  Moreover, thanks to the possibility of extra chances,
the probability of rejection at every step drops from $\alpha(q_0,p_0)$ to $1 - \Gamma^K(q_0,p_0)$.
As its name suggests, the accept/reject mechanism in XHMC can be implemented
recursively, so that the proposal move $\Psi_{\lambda}^j(q_0,p_0)$ and its associated
 transition probability $ \Gamma^{j} - \Gamma^{j-1} $ are only
  computed if $\Psi_{\lambda}^{j-1}(q_0,p_0)$ is not accepted.  One can also easily
  incorporate a partial momentum randomization into XHMC.

% Parallel Tempering

\subsubsection{Tempering techniques}

Probability distributions with multiple modes arise frequently in practice.  Different modes may represent
metastable conformations of a large molecular system \cite{FeFrDo1990,HaOk1993,Ha1997,SuOk1999,Wa2003} or
important features of posterior distributions in Bayesian inference problems
\cite{GeTh1995,Ne1996,LiWo2001,JiWo2006,KoZhWo2006}. Unfortunately, HMC may not converge well in situations
where the potential energy has multiple minima separated by large barriers.    This is a well-known limitation
of HMC.  It may be due to the low probability of the momentum randomizations in HMC to produce an initial
condition whose kinetic energy is high enough so that the corresponding solution to Hamilton's equations
crosses the potential energy barriers. Also the momentum randomization may have difficulties to produce an
initial momentum \(\xi_0\) with the right direction for the solution to Hamilton's equation to exit the
current potential well. For such multi-modal target distributions, it is advantageous to use HMC in
conjunction with a tempering technique such as parallel tempering, simulated tempering, and tempered
transition methods \cite{Ge1991,MaPa1992,GeTh1995,Ne1996}.  Tempering methods typically introduce an
artificial inverse temperature parameter \(\beta\) and apply HMC to the Boltzmann-Gibbs distribution at
multiple values of this parameter. The basic idea is that at sufficiently high temperature, the tempered
Boltzmann-Gibbs distribution becomes flat enough that barrier crossings in the Hamiltonian integration legs
become more likely.  Of course, the states of the tempered chains cannot be directly used to compute sample
averages with respect to $\PiBG$, which emphasizes that the computational cost of these tempering techniques
is by no means negligible.

% NUTS

\subsubsection{No-U-Turns}

A practical difficulty with HMC is tuning the durations of the Hamiltonian integration legs.   The No-U-Turn
Sampler (NUTS) adaptively chooses this parameter by increasing the length of the numerical orbit until the
distance between the endpoints of the orbit stops increasing.  Unfortunately, this procedure breaks
reversibility, since, roughly speaking, if this criterion is applied in the reverse direction starting from
the end of a trajectory, one does not expect to land at the start of the trajectory.  Since reversibility is a
key ingredient to the accept/reject mechanism in Algorithm~\ref{algo:metropolized_reversible_map}, and the
proof of the associated Theorem~\ref{thm:numerical_hmc}, NUTS uses a different approach: it applies the
criterion to the discrete trajectory obtained by integrating the Hamiltonian dynamics in the forward and
backward directions. Starting with a single step, the number of forward and backward integration steps is
doubled until a U-turn first occurs.  NUTS then ends the simulation and samples from all of these points in a
way that ensures that the algorithm leaves $\PiBG$ invariant.  The procedure does not use the
Metropolis-Hastings method, but a variant of the slice sampling method \cite{Ne2003}.  Since this sampling
step involves all of the states produced by this procedure, the algorithm is thought to have some of the
appealing properties of the windowed HMC method, which employs an accept/reject mechanism based on windows of
states \cite{Ne1994,Ne2011}.
% SHMC

\subsubsection{Shadow HMC}

Given a symplectic and reversible integrator $\psi_h$ for the Hamiltonian dynamics, shadow HMC is an
importance sampling technique that replaces the Boltzmann-Gibbs density $\exp(-H)$ with a shadow density
$\exp(-\tilde H_h^{[\mu]})$ where $\tilde H_h^{[\mu]}$ is the modified (or shadow) Hamiltonian of $\psi_h$ up
to order $\mu$ \cite{IzHa2004,SwHaSkIz2009,AkRe2008}.  In so doing, shadow HMC samples from a different target
density and the samples have to be reweighed  in order to compute averages with respect to the target
distribution.  This method is not to be confused with the so-called surrogate transition method \cite{Li2008},
since in shadow HMC the target density is altered. The rationale for shadow HMC is that $\exp(-\tilde H_h)$ is
easy to sample from, in the sense that $\psi_h$ preserves $\tilde H_h$ much more accurately than $H$, and
hence, for a given duration parameter, shadow HMC increases the acceptance probability without  having to
reduce the time step $h$. There is also a shadow generalized HMC method, which similarly replaces the target
density in GHMC with $\exp(-\tilde H_h)$. Several aspects of shadow HMC have not been sufficiently analyzed
from a mathematical viewpoint. For instance the integrability of $\exp(-\tilde H_h^{[\mu]})$ has not been
studied in depth. Moreover, shadow HMC requires estimating  the modified Hamiltonian to order $\mu$, which is
nontrivial since it requires computing (or estimating) higher derivatives of $\mathcal{U}$
\cite{skeel2001practical}. The momentum randomization step in Algorithm~\ref{algo:numerical_hmc} must be
modified, since the $p$-marginal of $\exp(-\tilde H_h^{[\mu]})$ is no longer a Gaussian distribution.

\vspace{2mm}
\section*{Acknowledgements}
This work was supported in part by a Catalyst grant to N.B-R. from the Provost's Fund for Research at Rutgers
University--€"Camden under Project No.~205536. J.M.S. has been supported by project MTM2016-77660-P(AEI/FEDER,
UE) funded by MINECO (Spain).

\vspace{2mm}
\bibliographystyle{actaagsm}
\bibliography{BoSaActaN2017}
\label{lastpage}
\end{document}